\documentclass[a4paper,reqno,11pt]{amsart}

\usepackage{fullpage}

\usepackage{amssymb}
\usepackage{dsfont}

\usepackage[utf8]{inputenc} % Required for inputting international characters
\usepackage[T1]{fontenc} % Output font encoding for international characters

\usepackage{hyperref}

\usepackage[nobysame,alphabetic,initials,msc-links]{amsrefs}

\DefineSimpleKey{bib}{how}

\renewcommand{\eprint}[1]{#1}
\BibSpec{misc}{%
  +{}{\PrintAuthors}  {author}
  +{,}{ \textit}      {title}
  +{,}{ }             {how}
  +{}{ \parenthesize} {date}
  +{,} { available at \eprint}        {eprint}
  +{,}{ available at \url}{url}
  +{,}{ }             {note}
  +{.}{}              {transition}
}

\numberwithin{equation}{section}

\theoremstyle{plain}%default
\newtheorem{thm}{Theorem}[section]
\newtheorem{prop}[thm]{Proposition}
\newtheorem{lemma}[thm]{Lemma}

\newtheorem{cor}[thm]{Corollary}
\theoremstyle{definition}

\newtheorem{defn}[thm]{Definition}
\theoremstyle{remark}

\newtheorem{question}[thm]{Question}
\newtheorem{remark}[thm]{Remark}
\newtheorem{example}[thm]{Example}

\theoremstyle{plain}
\newtheorem{innercustomthm}{Theorem}
\newenvironment{customthm}[1]
  {\renewcommand\theinnercustomthm{#1}\innercustomthm}
  {\endinnercustomthm}

\newcommand\bp{\begin{proof}}
\newcommand\ep{\end{proof}}

\newcommand\dach{{\!\widehat{\ \ }}}

\newcommand\C{\mathbb{C}}
\newcommand\N{\mathbb{N}}
\newcommand\Q{\mathbb{Q}}
\newcommand\R{\mathbb{R}}
\newcommand\T{\mathbb{T}}
\newcommand\Z{\mathbb{Z}}

\newcommand{\G}{\mathcal{G}}
\newcommand{\HH}{\mathcal{H}}

\newcommand{\M}{\mathcal{M}}

\newcommand{\QQ}{\mathcal{Q}}
\newcommand{\SSS}{\mathcal{S}}

\newcommand{\UU}{\mathcal{U}}

\newcommand\g{\mathfrak{g}}

\newcommand\Ad{\operatorname{Ad}}

\newcommand{\Cl}{\operatorname{Cl}}
\newcommand\diag{\operatorname{diag}}

\newcommand{\Gu}{{\mathcal{G}^{(0)}}}
\newcommand{\Gxx}{\mathcal{G}^x_x}
\newcommand{\Gyy}{\mathcal{G}^y_y}

\newcommand\ev{\mathrm{ev}}
\newcommand\fin{\mathrm{fin}}
\newcommand\hull{\operatorname{hull}}
\newcommand\HS{\operatorname{HS}}
\newcommand\Ind{\operatorname{Ind}}
\newcommand\Mat{\operatorname{Mat}}

\newcommand\Stab{\operatorname{Stab}}
\newcommand\Sub{\operatorname{Sub}}
\newcommand\supp{\operatorname{supp}}
\newcommand\ord{\operatorname{ord}}
\newcommand\prim{\mathrm{prim}}
\newcommand\Prim{\operatorname{Prim}}

\newcommand\Res{\operatorname{Res}}

\newcommand\SL{\operatorname{SL}}
\newcommand\Tr{\operatorname{Tr}}
\newcommand\eps{\varepsilon}
\newcommand\Iso{\operatorname{Iso}(\mathcal G)^{\circ}}
\newcommand\IsoG{\operatorname{Iso}(\mathcal G)}

\newcommand\ee{\nopagebreak\mbox{\ }\hfill$\diamond$}

\begin{document}

\title{The ideal structure of C$^*$-algebras of \'etale groupoids with isotropy groups of local polynomial growth}

\date{December 16, 2024; revised May 30, 2025; minor changes July 5, 2025}

\author{Johannes Christensen}
\address{Department of Mathematics, Aarhus University, Denmark}
\email{johannes@math.au.dk}

\author{Sergey Neshveyev}
\address{Department of Mathematics, University of Oslo, Norway}
\email{sergeyn@math.uio.no}
%\thanks{$^{2}$University of Oslo, Mathematics institute (Norway).  E-mail:  sergeyn@math.uio.no}

\thanks{J.C. is supported by the postdoctoral fellowship 1291823N of the Research Foundation Flanders and by a research
grant (VIL72080) from Villum Fonden. S.N. is partially supported by the NFR funded project 300837 ``Quantum Symmetry''.}

\begin{abstract}
Given an amenable second countable Hausdorff locally compact \'etale groupoid $\mathcal G$ such that each isotropy group~$\mathcal G^x_x$ has local polynomial growth, we give a description of $\operatorname{Prim} C^*(\mathcal G)$ as a topological space in terms of the topology on $\mathcal G$ and representation theory of the isotropy groups and their subgroups. The description simplifies when either the isotropy groups are FC-hypercentral or $\mathcal G$ is the transformation groupoid $\Gamma\ltimes X$ defined by an action $\Gamma\curvearrowright X$ with locally finite stabilizers.  To illustrate the class of C$^*$-algebras for which our results can provide a complete description of the ideal structure, we compute the primitive spectrum of $\mathrm{SL}_3(\mathbb Z)\ltimes C_0(\mathrm{SL}_3(\mathbb R)/U_3(\mathbb R))$, where $U_3(\mathbb R)$ is the group of unipotent upper triangular matrices.
\end{abstract}

\maketitle

\section*{Introduction}

In this paper we continue our work~\cite{CN3} on the ideal structure of C$^*$-algebras of amenable second countable Hausdorff locally compact \'etale groupoids $\G$. The basis of our analysis is a version of the Effros--Hahn conjecture stating that for such groupoids every primitive ideal is induced from an isotropy group, which in the case of transformation groupoids was proved by Sauvageot~\cite{Sau} and in the general case by Ionescu and Williams~\cite{IW}, see also~\cite{GR} and~\cite{R}. Therefore there is a surjective induction map
\begin{equation}\tag{$*$}\label{eq:VVW0}
\Ind\colon \Stab(\G)^\prim\to \Prim C^*(\G),
\end{equation}
where $\Stab(\G)^\prim$ is the set of pairs $(x,J)$ such that $x\in\Gu$ is a unit and $J$ is a primitive ideal in~$C^*(\Gxx)$. The problem is to understand the Jacobson topology on $\Prim C^*(\G)$ in terms of this map. In some form it goes back to the origins of the Mackey theory~\cite{MR0031489} that aims to describe the unitary dual of a group with a normal subgroup having a simpler representation theory.

The problem simplifies significantly when the groupoid $\G$ is principal and therefore $\Stab(\G)^\prim$ can be identified with $\Gu$. As was shown already by Effros and Hahn~\cite{EH} for transformation groupoids (and, in fact, even earlier by Glimm~\cite{MR0146297}, if we accept surjectivity of the induction map), we then get a homeomorphism $(\G\backslash\Gu)^\sim\cong \Prim C^*(\G)$, were~$\sim$ denotes the $T_0$-ization, or the Kolmogorov quotient, obtained by identifying points with identical closures; in other words, $(\G\backslash\Gu)^\sim$ is the quasi-orbit space for the action $\G\curvearrowright\Gu$. The same is true when~$\G$ is assumed to be only essentially principal, meaning that every closed invariant subset $X\subset\Gu$ has a dense subset of points with trivial isotropy groups, see~\cites{Rbook,SW} and Section~\ref{ssec:IIP} below. This is closely related to the fact that in this case the induction map~\eqref{eq:VVW0} factors through the unit space~$\Gu$. When~$\G$ is not essentially principal, this property no longer holds and we are forced to work with the entire space $\Stab(\G)^\prim$.

For transformation groupoids $\G=\Gamma\ltimes X$ with stabilizer groups $\Gamma_x$ depending continuously on~$x$, the problem of describing the topology on the induced primitive ideals was solved long ago by Glimm~\cite{MR0146297}, and similar arguments work for more general groupoids with continuous isotropy, see~\cite{MR2966476} and references there, as well as Section~\ref{ssec:questions} below. Since then there hasn't been much progress for groupoids with discontinuous nonabelian isotropy, although a significant progress has been made in the adjacent area of primitive ideal spaces of group C$^*$-algebras, particularly in the nilpotent case, see \citelist{\cite{MR0352326}\cite{BP}\cite{MR1473630}\cite{MR3012147}} and references there. There has been no solution even for type I groupoid C$^*$-algebras, where one would expect the groupoid structure to be much more manageable and the original goals of the Mackey theory to be most achievable. One of the few general results obtained since~\cite{MR0146297} that does not require continuous isotropy is the description of $\Prim(\Gamma\ltimes C_0(X))$ for proper actions $\Gamma\curvearrowright X$  by Echterhoff and Emerson~\cite{EE}. Note that in this case the isotropy structure is still relatively benign: the stabilizers~$\Gamma_x$ are finite, the map $x\mapsto\Gamma_x$ is continuous outside a closed set with empty interior and, in fact, the whole problem quickly reduces to the case when $\Gamma$ is finite. (It should be said that \cite{EE} deals with proper actions of nondiscrete groups as well, in which case the isotropy structure is more complicated.)

The situation is better for groupoids with abelian isotropy. In~\cite{MR0617538} Williams described the topology on $\Prim(\Gamma\ltimes C_0(X))$ under the assumption that the stabilizers of the action $\Gamma\curvearrowright X$ are contained in one abelian subgroup of $\Gamma$. Although his result requires an assumption on the global structure of the isotropy bundle, an important point is that this assumption is purely algebraic. A few years ago van Wyk and Williams~\cite{MR4395600} proposed a hypothetical extension of results of~\cite{MR0617538} to groupoids with abelian isotropy groups. Specifically, for such groupoids, they defined a topology on $\Stab(\G)^\prim$ such that $\Ind$ gives rise to a continuous map
\begin{equation}\tag{$**$}\label{eq:VVW}
\Ind^\sim\colon(\G\backslash \Stab(\G)^\prim)^\sim\to\Prim C^*(\G),
\end{equation}
and conjectured that this map is a homeomorphism ``in most circumstances''. We remark that this can be seen as a more refined version of an old conjecture of Baggett~\cite{MR0409720}*{Conjecture~2}, when the latter is adapted to groupoids with abelian isotropy.

In our previous paper~\cite{CN3} we proved that~\eqref{eq:VVW} is indeed a homeomorphism for all amenable second countable \'etale groupoids with abelian isotropy groups. This covered a fairly large class of C$^*$-algebras including higher rank graph C$^*$-algebras and crossed products $\Gamma\ltimes C_0(X)$ defined by amenable actions $\Gamma\curvearrowright X$ with abelian stabilizers. Although there are plenty of nontrivial examples of such actions, it is nevertheless clear that unless $\Gamma$ is already abelian, the assumption of commutativity of the stabilizers is quite restrictive. At the same time we are not aware of any results in the literature that would describe $\Prim(\Gamma\ltimes C_0(X))$ for some infinite noncommutative group $\Gamma$ and arbitrary actions $\Gamma\curvearrowright X$.

In the present paper we develop further the ideas of~\cite{CN3} and significantly expand the class of groupoids for which $\Prim C^*(\G)$ can be described in terms of the isotropy groups and the topology on $\G$. In order to formulate our main results, denote by $\Sub(\G)^\prim$ the set of pairs~$(S,I)$ such that $S$ is a subgroup of $\G$ and $I\in\Prim C^*(S)$, and equip it with Fell's topology for subgroup-representation pairs, see Section~\ref{ssec:Fell} for details. Then $\Stab(\G)^\prim$ can be viewed as a subset of $\Sub(\G)^\prim$.

\begin{customthm}{A}\label{thm:A}
Assume $\G$ is an amenable second countable Hausdorff locally compact \'etale groupoid such that each isotropy group $\Gxx$ has local polynomial growth. Then the topology on $\Prim C^*(\G)$ is described as follows. Suppose $\Omega$ is a $\G$-invariant subset of $\Stab(\G)^\prim$ and $(x,J)\in\Stab(\G)^\prim$. Then the following conditions are equivalent:
\begin{enumerate}
  \item $\Ind(x,J)$ lies in the closure of $\Ind\Omega$ in $\Prim C^*(\G)$;
  \item $\bigcap_{(S,I)}\Ind^{\Gxx}_S I\subset J$, where the intersection is taken over all points $(S,I)\in\Sub(\G)^\prim$ in the closure of $\Omega$ such that $S\subset\Gxx$.
\end{enumerate}
\end{customthm}

This in particular gives a description of $\Prim(\Gamma\ltimes C_0(X))$ for arbitrary actions $\Gamma\curvearrowright X$ of countable groups of local polynomial growth on second countable locally compact spaces. The theorem also covers arbitrary \emph{amenable} actions $\Gamma\curvearrowright X$ of some groups of exponential growth. For example, we can take free products of groups of local polynomial growth or, more generally, groups that are relatively hyperbolic groups with respect to a family of subgroups of local polynomial growth, see Section~\ref{ssec:main}.

In fact, our results are a bit more general than formulated above. First, we allow a strictly larger class of isotropy groups than groups of local polynomial growth, but we do not have a purely algebraic characterization of this class, see again Section~\ref{ssec:main}. Second, our main technical results do not need amenability of the entire groupoid and second countability, but then they can be used to describe only a part of $\Prim C^*(\G)$. Note also that Theorem~\ref{thm:A} does not have a form suggested by~\eqref{eq:VVW} or~\cite{MR0409720}*{Conjecture~2}, see Section~\ref{ssec:questions} for further discussion.

For more restricted classes of groupoids Theorem~\ref{thm:A} can be improved. In particular, we will prove the following result.

\begin{customthm}{B}\label{thm:B}
Assume $\G$ is an amenable second countable Hausdorff locally compact \'etale groupoid such that each isotropy group $\Gxx$ is virtually nilpotent or, more generally, FC-hypercentral. Then the topology on $\Prim C^*(\G)$ is described as follows. Suppose $\Omega$ is a $\G$-invariant subset of $\Stab(\G)^\prim$ and $(x,J)\in\Stab(\G)^\prim$. Then the following conditions are equivalent:
\begin{enumerate}
  \item $\Ind(x,J)$ lies in the closure of $\Ind\Omega$ in $\Prim C^*(\G)$;
  \item the closure of $\Omega$ in $\Sub(\G)^\prim$ contains a point $(S,I)$ such that $S\subset\Gxx$ and $\Res^{\Gxx}_S J\subset I$.
\end{enumerate}
\end{customthm}

The class of FC-hypercentral groups is, in fact, the largest class of isotropy groups for which one can have such a simple description of the Jacobson topology, see Section~\ref{ssec:FC}. Theorem~\ref{thm:B} describes in particular $\Prim(\Gamma\ltimes C_0(X))$ for actions of FC-hypercentral groups and for amenable actions of groups that are relatively hyperbolic with respect to a family of FC-hypercentral subgroups. The class of such relatively hyperbolic groups includes notably all lattices in semisimple Lie groups of real rank one. The theorem also covers all $\G$ such that $C^*(\G)$ is a type~I C$^*$-algebra, in certain sense completing the Mackey program in the \'etale case. It also implies
that (an adaptation of) \cite{MR0409720}*{Conjecture~2} is true for \'etale groupoids with FC-hyper\-central isotropy groups.

\smallskip

Let us now describe the contents of the paper in more detail. In Section~\ref{sec:prelim} we explain our notation and recall a few results and constructions for \'etale groupoids and associated C$^*$-algebras. A~large part of this section is devoted to a discussion of the topology on $\Sub(\G)^\prim$.

Section~\ref{sec:Frobenius} has two goals. One is to setup a connection between weak containment and $L^2$-estimates for approximate Hilbert--Schmidt intertwiners, which will be our main tool for going from representations of groupoid C$^*$-algebras to those of C$^*$-algebras of isotropy groups and back. The other related goal is to obtain weak forms of Frobenius reciprocity. By now it is well understood when such forms exist for all subgroups/representations of a given discrete group, but for our purposes we need results that are only valid for some subgroups/representations. Along the way we provide short self-contained proofs of a few known results.

Section~\ref{sec:main} contains our main technical results. Here we study the question when a family of induced, not necessarily irreducible, representations weakly contains another induced representation and give a few necessary and, under suitable assumptions, sufficient conditions.

In Section~\ref{sec:primitive} we use the results of the previous section to give a description of $\Prim C^*(\G)$ for certain classes of groupoids. We first prove a more general version of Theorem~\ref{thm:A}. Then we give more transparent descriptions of the primitive spectrum for groupoids satisfying additional assumptions. One such class is the groupoids with FC-hypercentral isotropy groups, as shown by Theorem~\ref{thm:B}. Another is the transformation groupoids defined by actions $\Gamma\curvearrowright X$ with locally finite stabilizers. The description of $\Prim(\Gamma\ltimes C_0(X))$ simplifies even more if we assume that the stabilizers are finite and have the property that, given $x\in X$, we have $\Gamma_y\subset\Gamma_x$ for all~$y$ close to~$x$. This is in particular true for proper actions $\Gamma\curvearrowright X$, in which case we recover the result of Echterhoff and Emerson~\cite{EE}. We then consider some open problems and a few simple examples. We finish the section with a discussion of essentially principal groupoids and the ideal intersection property.

In Section~\ref{sec:SL3} we use Theorem~\ref{thm:B} to give a complete description of the primitive ideal space of $\SL_3(\Z)\ltimes C_0(\SL_3(\R)/U_3(\R))$, where $U_3(\R)\subset\SL_3(\R)$ is the subgroup of unipotent upper triangular matrices. The stabilizers of the action $\SL_3(\Z)\curvearrowright \SL_3(\R)/U_3(\R)$ are either trivial, isomorphic to $\Z^2$ or  isomorphic to the discrete Heisenberg group, and the points of each of these three types are dense in $\SL_3(\R)/U_3(\R)$. In particular, the map $x\mapsto\SL_3(\Z)_x$ is discontinuous and takes values in infinite noncommutative subgroups of $\SL_3(\Z)$ on a dense set of points. The analysis of the primitive ideal space in this case requires a good understanding of quasi-orbit spaces for various actions of $\SL_2(\Z)$ and $\SL_3(\Z)$. As sets, these quasi-orbit spaces are not difficult to understand using the theory of unipotent flows, and our main work lies in describing the topology on them.

\bigskip

\section{Preliminaries} \label{sec:prelim}

\subsection{\'Etale groupoids, their \texorpdfstring{C$^*$}{C*}-algebras, and induction}\label{ssec:groupoids}
We use the same conventions and notation for groupoids and the associated algebras as in~\cite{CN3}. Briefly, we work with Hausdorff locally compact \'etale groupoids~$\G$. Additional assumptions such as second countability or amenability are formulated explicitly when they are needed. The range and source maps are denoted by $r\colon\G\to\Gu$ and $s\colon\G\to\Gu$. A subset $W\subset\G$ is called a bisection if both~$r$ and~$s$ are injective on $W$. We let $\G_x:=s^{-1}(x)$, $\G^x:=r^{-1}(x)$ and $\Gxx:=\G_x\cap\G^x$.
%We denote by $[x]:=r(\G_x)$ the $\G$-orbit of $x$ in $\Gu$.
The isotropy bundle is defined by
$$
\IsoG:=\{g\in\G:s(g)=r(g)\}.
$$
We denote by $\Stab(\G)$ the entire collection of the isotropy groups $\Gxx$. Therefore $\IsoG=\bigcup_{H\in\Stab(\G)}H$.

\smallskip

The space $C_c(\G)$ of continuous compactly supported functions on $\G$ is a $*$-algebra with convolution product
\begin{equation*} \label{eprod}
(f_{1}*f_{2})(g) := \sum_{h \in \G^{r(g)}} f_{1}(h) f_{2}(h^{-1}g)
\end{equation*}
and involution $f^{*}(g):=\overline{f(g^{-1})}$. The full groupoid C$^*$-algebra $C^*(\G)$ is by definition the completion of $C_c(\G)$ with respect to the norm $\lVert f \rVert: =\sup_{\pi}\|\pi(f)\|$, where the supremum is taken over all representations of $ C_c(\G)$ by bounded operators on Hilbert spaces.

\smallskip

Given $x\in\Gu$, a subgroup $S\subset\Gxx$ and a unitary representation $\pi\colon S\to U(H)$ on a Hilbert space $H$, we define the induced representation $\Ind\pi=\Ind^\G_S\pi$ of $C^*(\G)$ as follows. The underlying Hilbert space $\Ind H$ of $\Ind\pi$ consists of the functions $\xi \colon  \G_{x} \to H$ such that
\begin{equation*}
\xi(gh)=\pi(h)^{*}\xi(g)\quad (g \in \G_{x},\ h\in S)\quad\text{and}\quad  \sum_{g\in \G_{x}/S}\lVert \xi(g)\rVert^{2}<\infty,
\end{equation*}
and the inner product on $\Ind H$ is given by
\begin{equation*}
(\xi_{1}, \xi_{2} ):=\sum_{g\in \G_{x}/S}( \xi_{1}(g), \xi_{2}(g)).
\end{equation*}
The induced representation is then defined by
\begin{equation}\label{eq:ind-rep}
\big((\Ind\pi)(f)\xi\big)(g) :=
\sum_{h\in \G^{r(g)}}f(h) \xi(h^{-1}g)\quad (g\in\G_x,\ \xi\in\Ind H,\ f\in C_c(\G)).
\end{equation}
We will often view $H$ as a subspace of $\Ind H$ by identifying $\xi\in H$ with the function $\G_x\to H$ supported on $S$ such that $S\ni h\mapsto\pi(h)^*\xi$.

\smallskip

It is known that the kernel of $\Ind\pi$ depends only on the kernel of $\pi$, by which we always mean the kernel of the extension of $\pi$ to $C^*(S)$. In fact, we can write
$$
\ker(\Ind\pi)=\{a\in C^{*}(\mathcal{G}) \mid \vartheta_{S}(b^{*} a^{*}a b)\in \ker\pi\ \text{for all}\ b\in C^{*}(\G)\},
$$
where $\vartheta_{S}\colon C^{*}(\mathcal{G}) \to C^{*}(S)$ is the unique completely positive contraction such that $\vartheta_S(f)=f|_S$ for $f\in C_c(\G)$, cf.~\cite{CN2}*{Lemma 2.6}.
We thus get a well-defined procedure of inducing ideals of $C^*(S)$: if $\ker\pi=J\subset C^*(S)$, then $\Ind^{\G}_SJ:=\ker \Ind^\G_S\pi$ is an ideal in $C^*(\G)$. Note that by an ideal in a C$^{*}$-algebra we always mean a closed two-sided ideal.

The induction applies in particular to subgroups $S\subset\Gamma$ of a discrete group $\Gamma$. In this case~$C^*(S)$ is a subalgebra of $C^*(\Gamma)$ and we also get a restriction map: given an ideal $J\subset C^*(\Gamma)$, we~let
$$
\Res^\Gamma_S J:=J\cap C^*(S).
$$
Therefore if $J=\ker \pi$ for a unitary representation $\pi$ of $\Gamma$, then $\Res^\Gamma_S J=\ker(\pi|_S)$.

\smallskip

Returning to groupoids, if $\pi$ is an irreducible unitary representation of $\Gxx$, then it is not difficult to show that $\Ind^{\G}_{\Gxx}\pi$ is irreducible as well (see~\cite{IW0} for a more general statement in the second countable case). Denote by $\Stab(\G)^\prim$ the set of pairs $(x,J)$ such that $x\in\Gu$ and $J\in\Prim C^*(\Gxx)$. We thus get a well-defined map
\begin{equation}\label{eq:Ind}
\Ind\colon\Stab(\G)^\prim\to\Prim C^*(\G),\quad \Ind(x,J):=\Ind^\G_{\Gxx}J=\ker\Ind^\G_{\Gxx}\pi_J,
\end{equation}
where $\pi_J$ denotes any irreducible representation of $C^*(\Gxx)$ with kernel $J$. We have an action of~$\G$ on $\Stab(\G)^\prim$ defined by
$$
g(x,J):=(r(g),(\Ad g)(J))\quad (g\in\G_x).
$$
The induction map $\Ind$ is constant on every $\G$-orbit (see~\cite{CN3}*{Lemma 1.1}).

A fundamental result is that the induction map~\eqref{eq:Ind} is surjective if $\G$ is second countable and amenable~\cite{IW}. This, and the fact that groupoid amenability implies that the isotropy groups $\Gxx$ are amenable~\cite{MR1799683}*{Proposition~5.1.1}, are essentially the only things we need to know about amenability of groupoids in this paper.

\subsection{Weak containment and the Jacobson topology}
Recall that, given a C$^*$-algebra $A$, the Jacobson topology on the  primitive ideal space $\Prim A$ is the topology in which the closed sets have the form
$$
\operatorname{hull}(J):=\{I\in\Prim A: J\subset I\},
$$
where $J\subset A$ is an ideal. Therefore the closure of a set $\Omega\subset\Prim A$  is $\hull(\bigcap_{I\in\Omega}I)$.

Given a nondegenerate representation $\pi\colon A\to B(H)$, denote by $\SSS_\pi(A)$ the collection of states on $A$ of the form $(\pi(\cdot)\xi,\xi)$, where $\xi\in H$ is a unit vector. One says that a representation~$\pi$ is weakly contained in a representation $\rho$, if $\SSS_\pi(A)$ is contained in the weak$^*$ closure of the convex hull of $\SSS_\rho(A)$, equivalently, $\ker\rho\subset\ker\pi$. We then write $\pi\prec\rho$. If $\pi\prec\rho$ and $\rho\prec\pi$, then the representations $\pi$ and $\rho$ are called \emph{weakly equivalent}, in which case we write $\pi\sim\rho$.

\begin{lemma}[{\cite{MR0146681}*{Theorem~1.4}}]\label{lem:Fell}
Assume $A$ is a C$^*$-algebra, $\pi$ is an irreducible representation of $A$ and $(\pi_i)_{i\in I}$ is a collection of nondegenerate representations of $A$. Then the following conditions are equivalent:
\begin{enumerate}
%\item $\bigcap_{i\in I}\ker\pi_i\subset\ker\pi$;
\item $\pi\prec\bigoplus_{i\in I}\pi_i$;
\item $\SSS_\pi(A)\subset\overline{\bigcup_{i\in I}\SSS_{\pi_i}(A)}^{w^*}$;
\item $\SSS_\pi(A)\cap\overline{\bigcup_{i\in I}\SSS_{\pi_i}(A)}^{w^*}\ne\emptyset$.
\end{enumerate}
\end{lemma}

We note that formally \cite{MR0146681}*{Theorem~1.4} assumes that $\pi_i$ are irreducible as well, but the proof relies only on irreducibility of $\pi$.

\smallskip

Unitary equivalence of unitary representations $\pi$ and $\rho$ of $A$ obviously implies weak equivalence. An intermediate property is \emph{quasi-equivalence}, which means that the map $\pi(A)\to\rho(A)$, $\pi(a)\mapsto\rho(a)$, is well-defined and extends to an isomorphism $\pi(A)''\cong\rho(A)''$.

\subsection{Fell's topologies and the subgroup bundle}\label{ssec:Fell}

Recall that if $X$ is a topological space, then the Fell topology on the set $\Cl(X)$ of closed subsets of $X$ is defined using as a basis the sets
$$
\UU(K;(U_i)^n_{i=1}):=\{A\in\Cl(X):A\cap K=\emptyset,\ A\cap U_i\ne\emptyset\ \ \text{for}\ \ i=1,\dots,n\},
$$
where $K\subset X$ is compact and $U_i\subset X$ are open~\cite{MR0139135}. Note that it is allowed to have $K=\emptyset$ and $n=0$. With this topology $\Cl(X)$ becomes a compact space, and if~$X$ is locally compact, then $\Cl(X)$ is Hausdorff.

Given a Hausdorff locally compact \'etale groupoid $\G$, we denote by $\Sub(\G)$ the set of subgroups of $\G$ (equivalently, of $\IsoG$), equipped with the Fell topology. The closure of $\Sub(\G)$ in $\Cl(\G)$ is either $\Sub(\G)$, when $\Gu$ is compact, or $\Sub(\G)\cup\{\emptyset\}$, when $\Gu$ is only locally compact. Therefore the space $\Sub(\G)$ is always locally compact.

A bit more explicitly the topology on $\Sub(\G)$ can be described as follows.

\begin{lemma} \label{lem:Sigma-basis}
Assume $S$ is a subgroup of $\Gxx$ for some $x\in\Gu$. For every $g\in\Gxx$, choose an open bisection $W_g\subset\G$ containing $g$. Then a base of open neighbourhoods of $S$ in $\Sub(\G)$ is given by the sets
\begin{equation}\label{eq:Sigma-basis}
\Sub(\G)\cap \UU\Big(\bigcup_{g\in A}(r^{-1}(\bar U)\cap W_g);(r^{-1}(U)\cap W_g)_{g\in B}\Big),
\end{equation}
where $A\subset\Gxx\setminus S$ and $B\subset S$ are finite sets and $U\subset\Gu$ is a relatively compact open neighbourhood of $x$ such that $\bar U\subset r(W_g)$ for all $g\in A\cup B$.
\end{lemma}
\bp
%Let us first remark that when considering the open sets for the relative topology on $\Sub(\G)$, it suffices to consider open sets $\UU(K;(U_i)^n_{i=1})$ %with $K\subseteq \text{Iso}(\G)$ compact and the family $\{U_{i}\}_{i=1}^{n}$ not void.
Let $K\subset\G$ be a compact set and $\{U_{i}\}_{i=1}^{n}$ be a family of open sets such that $S \in \UU(K;(U_i)^n_{i=1})$. It is not difficult to see that if we put $A:=K\cap \Gxx$, then there is an open neighbourhood~$U$ of~$x$ such that $r^{-1}(U)\cap\IsoG\cap K\subset\bigcup_{g\in A} W_g$, cf.~the proof of~\cite{CN3}*{Lemma~2.1}. Note that $A$ is a finite subset of~$\Gxx\setminus S$. For every $i=1, \dots, n$, pick any $g_{i} \in S\cap U_{i}$. Put $B:=\{e, g_{1}, \dots, g_{n}\}$, where $e=x$ is the unit of~$S$. We can, by possibly picking a smaller neighbourhood $U$ of $x$, assume that~$U$ is relatively compact, $\bar U\subset r(W_g)$ for all $g\in A\cup B$ and $r^{-1}(U) \cap W_{g_{i}} \subset U_{i}$ for $i=1, \dots, n$. We claim that
$$
\Sub(\G)\cap \UU\Big(\bigcup_{g\in A}(r^{-1}(\bar U)\cap W_g);(r^{-1}(U)\cap W_g)_{g\in B}\Big)
\subset
\Sub(\G)\cap \UU(K;(U_i)^n_{i=1}) \; .
$$

Indeed, assume a subgroup $T\subset\G^y_y$, for some $y\in\Gu$, is contained in the neighbourhood on the left. Then $T\cap r^{-1}(U)\cap W_e\ne\emptyset$, hence $y\in U$. It follows that $T\cap K= T\cap r^{-1}(U)\cap\IsoG\cap K=\emptyset$, since $r^{-1}(U)\cap\IsoG\cap K\subset\bigcup_{g\in A}(r^{-1}(\bar U)\cap W_g)$. Finally, for every $i=1,\dots,n$, we have $T\cap U_i\supset T\cap r^{-1}(U)\cap W_{g_i}\ne\emptyset$. Hence $T\in \UU(K;(U_i)^n_{i=1})$, proving the desired inclusion.
\ep

Using Lemma \ref{lem:Sigma-basis} we immediately get the following description of convergence in~$\Sub(\G)$.

\begin{lemma} \label{lem:subG-convergence}
Assume $S$ is a subgroup of $\Gxx$ for some $x\in\Gu$. For every $g\in\Gxx$, choose an open bisection $W_g\subset\G$ containing $g$. Then a net $(S_i)_i$ in $\Sub(\G)$ converges to $S$ if and only if the following holds:
\begin{enumerate}
  \item[(i)] $r(S_i)\to x$ in $\Gu$;
  \item[(ii)] for every $g\in\Gxx$, we have, for all $i$ large enough, that $W_g\cap S_i\ne\emptyset$ if and only if $g\in S$.
\end{enumerate}
\end{lemma}

Now, consider the space
$$
\Sigma(\G):=\{(S,g)\in\Sub(\G)\times\G: g\in S\}.
$$
This is a bundle of groups over $\Sub(\G)$, that is, a groupoid which coincides with its isotropy.

\begin{lemma} \label{lem:Sigma-etale}
With the topology inherited from $\Sub(\G)\times\G$, the groupoid $\Sigma(\G)$ is locally compact and \'etale.
\end{lemma}

\bp
Since $\Sigma(\G)$ is closed in $\Sub(\G)\times\G$, it is easy to see that $\Sigma(\G)$ becomes a locally compact topological groupoid in the topology inherited from $\Sub(\G)\times\G$. Using the projection map $\Sigma(\G) \to \Sub(\G)$ we identify $\Sigma(\G)^{(0)}$ with $\Sub(\G)$. Take $(S,g)\in\Sigma(\G)$ and put $x:=r(S)\in\Gu$. Consider a neighbourhood $V$ of $S$ of the form~\eqref{eq:Sigma-basis} such that $B$ contains $g$. Then $W:=\Sigma(\G)\cap (V\times W_g)$ is an open bisection of $\Sigma(\G)$ containing $(S,g)$. Namely, the inverse to the projection map $W\to V$ maps $T\in V$ into $(T,g_T)$, where $g_T$ is the unique element of $W_g\cap T$.
\ep

We can therefore consider the C$^*$-algebra $C^*(\Sigma(\G))$, see~\citelist{\cite{MR0159898},\cite{R}}. Since it is by construction a $C_0(\Sub(\G))$-algebra with the fiber $C^*(S)$ at $S\in\Sigma(\G)^{(0)}=\Sub(\G)$, its primitive spectrum can be identified with the set $\Sub(\G)^{\prim}$ of pairs $(S,J)$ such that $S\in\Sub(\G)$ and $J\in\Prim C^*(S)$. The main point of introducing $C^*(\Sigma(\G))$ is that we get a natural topology on $\Sub(\G)^{\prim}$ defined by the Jacobson topology on $\Prim C^*(\Sigma(\G))$. Convergence in this topology can be described as follows. In order to have a cleaner formulation we will restrict ourselves to the second countable groupoids and comment on the general case afterwards.

\begin{lemma}[cf.~\cite{MR0159898}*{Theorem~3.1}]\label{lem:Fell-convergence}
Assume $\G$ is second countable and $(S,J)\in\Sub(\G)^\prim$. For every $g\in S$, choose an open bisection $W_g\subset\G$ containing $g$. Then a sequence $((S_n,J_n))_n$ converges to $(S,J)$ in $\Sub(\G)^\prim$ if and only if $S_n\to S$ in $\Sub(\G)$ and, for some (equivalently, every) $\varphi\in\SSS_{\pi_J}(C^*(S))$, there exist states $\varphi_{n}\in \SSS_{\pi_{J_{n}}}(C^*(S_{n}))$ such that $\varphi_{n}(g_{n})\to\varphi(g)$ for all $g\in S$, where $g_{n}$ is the unique element of $W_g\cap S_{n}$ (which exists for $n$ large enough).
\end{lemma}

\bp
Since $\G$ is second countable, the groupoid $\Sigma(\G)$ is second countable as well, hence $C^*(\Sigma(\G))$ is separable and $\Sub(\G)^\prim$ is second countable.
Denote by $\pi\colon C^*(\Sigma(\G))\to B(H)$ the representation $\Ind^{\Sigma(\G)}_S\pi_J$, where we identify $S$ with $\{S\}\times S\subset\Sigma(\G)$. In other words, $\pi$ is the composition of $\pi_J$ with the quotient map $C^*(\Sigma(\G))\to C^*(S)$. Similarly, denote by $\pi_n\colon C^*(\Sigma(\G))\to B(H_n)$ the representation $\Ind^{\Sigma(\G)}_{S_n}\pi_{J_n}$. Then by the definition of the topology on $\Sub(\G)^\prim$ and Lemma~\ref{lem:Fell}, we have $(S_n,J_n)\to(S,J)$ if and only if for some (equivalently, every) unit vector $\xi\in H$,  there exist unit vectors $\xi_{n}\in H_{n}$ such that
\begin{equation} \label{eq:subgroup-rep-conv}
(\pi_{n}(a)\xi_{n},\xi_{n})\to(\pi(a)\xi,\xi)
\end{equation}
for all $a\in C^*(\Sigma(\G))$. By taking functions $a\in C_0(\Sub(\G))$ we see that a necessary condition for~\eqref{eq:subgroup-rep-conv} is that $S_n\to S$ in $\Sub(\G)$.

Assume from now on that we indeed have $S_n\to S$. Then, when vectors $\xi_{n}$ are fixed,~\eqref{eq:subgroup-rep-conv} holds for all $a\in C^*(\Sigma(\G))$ if and only if it holds for all functions $a\in C_c(\Sigma(\G))$ that are supported on the bisections $\Sigma(\G)\cap (V\times W_g)$ from the proof of Lemma~\ref{lem:Sigma-etale} and that have the form $a(T,g_T)=f(T)$, with $f\in C_c(V)$. But for every such function $a$ and all $n$ large enough we have
$$
(\pi(a)\xi,\xi)=f(S)(\pi_J(g)\xi,\xi)\quad\text{and}\quad (\pi_{n}(a)\xi_{n},\xi_{n})=f(S_{n})(\pi_{J_{n}}(g_{n})\xi_{n},\xi_{n}),
$$
where $g_{n}$ is the unique element of $W_g\cap S_{n}$, which exists by Lemma~\ref{lem:subG-convergence} when $n$ is sufficiently big. As $f(S_n)\to f(S)$, it follows that~\eqref{eq:subgroup-rep-conv} holds for all $a\in C^*(\Sigma(\G))$ if and only if $(\pi_{J_{n}}(g_{n})\xi_{n},\xi_{n})\to(\pi_J(g)\xi,\xi)$ for all $g\in S$. This gives the result.
\ep

\begin{remark}\label{rem:Fell-continuity}
Without second countability one can argue in a similar way, but the formulation becomes more cumbersome: a net $((S_i,J_i))_i$ converges to $(S,J)$ in $\Sub(\G)^\prim$ if and only if $S_i\to S$ in $\Sub(\G)$ and, for some (equivalently, every) $\varphi\in\SSS_{\pi_J}(C^*(S))$ and every subnet of $((S_i,J_i))_i$,  there exist a subnet $((S_{i_j},I_{i_j}))_j$ of that subnet and states $\varphi_{j}\in \SSS_{\pi_{J_{i_j}}}(C^*(S_{i_j}))$ such that $\varphi_{j}(g_{i_j})\to\varphi(g)$ for all $g\in S$, where $g_{i_j}$ is the unique element of $W_g\cap S_{i_j}$. \ee
\end{remark}

Although $\Stab(\G)^\prim$ can be viewed as a subset of $\Sub(\G)^\prim$, in general this subset is not closed and the induced topology on $\Stab(\G)^\prim$ is not the right one to describe the topology on~$\Prim C^*(\G)$. We will say more on this in Section~\ref{ssec:questions}.

\smallskip

We will need a continuity result for the induction maps $\Ind^\G_S$. A general result of this sort, albeit formally only for transformation groupoids, is established in~\cite{MR0159898}*{Theorem~4.2}. It is also worth noting that a more conceptual way of obtaining such results compared to~\cite{MR0159898} is to interpret $\Ind^\G_S$ as a particular case of Rieffel's induction via a $C^*(\G)$-$C^*(\Sigma(\G))$-correspondence as, for example, in~\cite{MR4395600}*{Section~2}. But for our purposes it will be enough to have the following simple lemma.

\begin{lemma}\label{lem:Ind-continuity}
Assume we are given a subgroup $S\subset\Gxx$ and a state $\varphi$ on $C^*(S)$. For every $g\in S$, choose an open bisection $W_g\subset\G$ containing $g$. Assume a net $(S_i)_i$ converges to $S$ in~$\Sub(\G)$ and~$\varphi_i$ are states on~$C^*(S_i)$ such that $\varphi_i(g_i)\to\varphi(g)$ for all $g\in S$, where $g_i$ is the unique element of $W_g\cap S_i$ (which exists for $i$ large enough). Consider the GNS-representations~$\pi_\varphi$ and~$\pi_{\varphi_i}$ associated with~$\varphi$ and~$\varphi_i$. Then $\Ind^\G_S\pi_\varphi\prec\bigoplus_i\Ind^\G_{S_i}\pi_{\varphi_i}$.
\end{lemma}

\bp
Let us write $\pi$ and $\pi_i$ for $\pi_\varphi$ and $\pi_{\varphi_i}$. Consider the GNS-triples $(H,\pi,\xi)$ and $(H_i,\pi_i,\xi_i)$ associated with $\varphi$ and $\varphi_i$. Viewing $H$ as a subspace of $\Ind^\G_SH$ and $H_i$ as a subspace of $\Ind^\G_{S_i}H_i$, consider the states
$$
\tilde\varphi:=\big((\Ind^\G_S\pi)(\cdot)\xi,\xi),\qquad \tilde\varphi_i:=\big((\Ind^\G_{S_i}\pi_i)(\cdot)\xi_i,\xi_i\big).
$$
Since the vector $\xi$ is cyclic for $\Ind^\G_S\pi$, it is then enough to show that $\tilde\varphi_i\to\tilde\varphi$ weakly$^*$.

For every $g\in\Gxx\setminus S$, choose an open bisection $W_g\subset\G$ containing $g$. Then it suffices to show that $\tilde\varphi_i(f)\to\tilde\varphi(f)$ for functions $f\in C_c(\G)$ such that $f\in C_c(W)$ for some open bisection $W$ and one of the following properties holds:
\begin{enumerate}
\item $W=W_g$ for some $g\in S$;
\item $W=W_g$ for some $g\in\Gxx\setminus S$;
\item $x\notin\overline{s(W)}\cap\overline{r(W)}$.
\end{enumerate}
A straightforward computation using Lemma~\ref{lem:subG-convergence} shows that, for $i$ large enough, we have $\tilde\varphi(f)=f(g)\varphi(g)$ and $\tilde\varphi_i(f)=f(g_i)\varphi_i(g_i)$ in case (1), $\tilde\varphi(f)=\tilde\varphi_i(f)=0$ in cases (2) and (3), cf.~the proof of~\cite{CN3}*{Lemma~2.3}. As $f(g_i)\to f(g)$ and $\varphi_i(g_i)\to\varphi(g)$ for $g\in S$, this finishes the proof of the lemma.
\ep

The assumptions of the above lemma can be formulated in terms of convergence in $\Sub(\G)^\prim$, as follows. Again, for simplicity we consider only the second countable case.

\begin{lemma}\label{lem:Fell-convergence2}
Assume $\G$ is second countable and we are given a subgroup $S\subset\Gxx$, a state $\varphi$ on~$C^*(S)$, a sequence $(S_n)_n$ in $\Sub(\G)$ converging to $S$ and, for every $n$, a unitary representation~$\pi_n$ of~$S_n$. For every $g\in S$, choose an open bisection $W_g\subset\G$ containing $g$. Then the following conditions are equivalent:
\begin{enumerate}
  \item there exist states~$\varphi_n$ on~$C^*_{\pi_n}(S_n):=C^*(S_n)/\ker\pi_n$ such that $\varphi_n(g_n)\to\varphi(g)$ for all $g\in S$, where $g_n$ is the unique element of $W_g\cap S_n$ (which exists for $n$ large enough);
  \item for every $J\in\hull(\ker\pi_\varphi)$, there exist ideals $J_n\in\hull(\ker\pi_n)$ such that $(S_n,J_n)\to(S,J)$ in $\Sub(\G)^\prim$.
\end{enumerate}
\end{lemma}

\bp (1)$\Rightarrow$(2) Since $\Sub(\G)^\prim$ is second countable and (1) holds also for any subsequence of~$(S_n)_n$, it suffices to show that for every $J\in\hull(\ker\pi_\varphi)$ and every neighbourhood $U$ of $(S,J)$ in $\Sub(\G)^\prim$, we can find $n$ and $I\in\hull(\ker\pi_n)$ such that $(S_n,I)\in U$.

By the previous lemma applied to $\Sigma(\G)$ instead of $\G$ we have
$$
\Ind^{\Sigma(\G)}_S\pi_\varphi\prec\bigoplus_n\Ind^{\Sigma(\G)}_{S_n}\pi_{\varphi_n}\prec \bigoplus_n\Ind^{\Sigma(\G) }_{S_n}\pi_n.
$$
Hence, for every $J\in\hull(\ker\pi_\varphi)$, we have $\Ind^{\Sigma(\G)}_S\pi_J\prec\oplus_n\Ind^{\Sigma(\G) }_{S_n}\pi_n$. Since $\pi_n$ is weakly equivalent to the direct sum of the representations $\pi_I$, $I\in\hull(\ker\pi_n)$, this implies that the point $(S,J)$ lies in the closure of $\cup_n\{(S_n,I):I\in\hull(\ker\pi_n)\}$ in $\Sub(\G)^\prim$, which is what we need.

\smallskip

(2)$\Rightarrow$(1) We will prove that (2) implies a formally stronger property than (1): for any state~$\psi$ on~$C^*_{\pi_\varphi}(S)$, there exist states~$\psi_n$ on~$C^*_{\pi_n}(S_n)$ such that $\psi_n(g_n)\to\psi(g)$ for all $g\in S$. Since $S$ is countable and every state on $C^*(S)$ can be approximated by finite convex combinations of pure states, it suffices to consider pure states $\psi$.

Consider the primitive ideal $J:=\ker\pi_\psi\in\hull(\ker\pi_\varphi)$. By assumption we can find $J_n\in\hull(\ker\pi_n)$ such that $(S_n,J_n)\to(S,J)$.
Since $C^*(S_n)/J_n$ is a quotient of $C^*_{\pi_n}(S_n)$, an application of Lemma~\ref{lem:Fell-convergence} gives the required states $\psi_n$.
\ep

\bigskip

\section{Approximate Hilbert--Schmidt intertwiners and weak Frobenius reciprocity}\label{sec:Frobenius}

\subsection{Amenability and duality for unitary representations}
There are two statements in representation theory of finite groups that are called Frobenius reciprocity, both are particular cases of the tensor-hom adjunction. A study of their weak analogues for infinite groups was initiated by Fell~\citelist{\cite{MR0155932}\cite{MR0159898}}. We will need a few results in this direction. We will include some known results and arguments in order to make the discussion essentially self-contained.

\smallskip

We begin by discussing a weak analogue of the isomorphism
\begin{equation}\label{eq:Frobenius}
\operatorname{Hom}_\Gamma(\pi\otimes\theta,\rho)\cong \operatorname{Hom}_\Gamma(\theta,\overline\pi\otimes\rho)
\end{equation}
for finite dimensional representations $\pi$. We are particularly interested in the case when $\theta$ is the trivial representation.

Let us first fix our notation. Assume $\Gamma$ is a discrete group. We denote by $\eps$ or $\eps_\Gamma$ the trivial representation of $\Gamma$.  Given a unitary representation $\pi$, whenever convenient we denote by $H_\pi$ the underlying Hilbert space. We use the same letter $\pi$ for the extension of $\pi$ to a representation of~$C^*(\Gamma)$. As in Section~\ref{sec:prelim}, we denote by $\ker\pi\subset C^*(\Gamma)$ the kernel of this extension and by~$C^*_\pi(\Gamma)$ the quotient $C^*(\Gamma)/\ker\pi$. Denote by~$\bar\pi$ the conjugate, or contragredient, representation. Thus, its underlying Hilbert space is the complex conjugate Hilbert space $\bar H_\pi$ and $\bar\pi(g)\bar\xi=\overline{\pi(g)\xi}$.

For Hilbert spaces $H_{1} $ and $H_{2}$, we denote by $\HS(H_{1}, H_{2})$ the space of Hilbert--Schmidt operators $H_1\to H_2$, and we use the notation $\HS(H)$ when $H_{1}=H_{2}=H$. The space $\HS(H_1,H_2)$ is a Hilbert space with inner product $(T,S):=\Tr(S^{*}T)$ and the corresponding norm $\|T\|_2:=\Tr(T^*T)^{1/2}$.

Given two unitary representations $\pi$ and $\rho$, we often identify $\bar H_\pi\otimes H_\rho$ with  $\HS(H_{\pi}, H_{\rho})$, with a vector $\bar\xi\otimes\zeta$ corresponding to the rank-one operator $\theta_{\zeta,\xi}$ defined~by
\begin{equation}\label{eq:theta}
\theta_{\zeta,\xi}\eta:=(\eta,\xi)\zeta.
\end{equation}
In this picture the representation $\bar\pi\otimes\rho$ takes the form
$$
(\bar\pi\otimes\rho)(g)T=\rho(g)T\pi(g)^*\quad\text{for}\quad T\in\HS(H_\pi,H_\rho).
$$

We will use the following standard observation repeatedly.

\begin{lemma}\label{lem:HSformulation}
Let $\pi$ and $\rho$ be nonzero unitary representations of a discrete group $\Gamma$. Then $\eps \prec \bar\pi\otimes\rho$ if and only if there exists a net of operators $T_{i} \in \HS(H_{\pi}, H_{\rho})$ such that $\lVert T_{i} \rVert_{2}=1$ for all $i$ and $\lVert T_{i} \pi(g)-\rho(g) T_{i} \rVert_{2} \to_i 0$ for all $g\in \Gamma$.
\end{lemma}

\bp
By Lemma~\ref{lem:Fell}, $\eps \prec \bar\pi\otimes\rho$ if and only if there exists a net of unit vectors $\xi_{i}\in\bar H_\pi\otimes H_\rho$ such that $((\bar\pi\otimes\rho)(g)\xi_i,\xi_i)\to1$, equivalently, $\lVert  (\bar\pi\otimes\rho)(g) \xi_{i} - \xi_{i} \rVert \to 0$ for all $g\in \Gamma$. When we identify $\bar H_\pi\otimes H_\rho$ with $\HS(H_{\pi}, H_{\rho})$, this is the same as saying that there exists a net of operators $T_{i} \in \HS(H_{\pi}, H_{\rho})$ such that $\lVert T_{i} \rVert_{2}=1$ and $\lVert \rho(g)T_{i} \pi(g)^{*}- T_{i} \rVert_{2} \to 0$ for all $g\in \Gamma$. Using that $\lVert \rho(g)T_{i} \pi(g)^{*}- T_{i} \rVert_{2}=\lVert T_{i} \pi(g)- \rho(g)T_{i} \rVert_{2}$, we obtain the lemma.
\ep

A nonzero unitary representation $\pi$ is called \emph{amenable}, if $\eps\prec\bar\pi\otimes\pi$ \cite{MR1047140}. Every nonzero finite dimensional unitary representation $\pi$ is amenable, e.g., because the identity operator on~$H_\pi$ is Hilbert--Schmidt.

If $\Gamma$ is amenable, then every nonzero unitary representation is amenable. A quick way to see this is as follows \cite{MR0998613}. Take a rank one projection $p\in B(H_\pi)$, a finite set $F\subset\Gamma$ and define
$$
T:=\Big(\frac{1}{|F|}\sum_{h\in F}\pi(h)p\pi(h)^*\Big)^{1/2}.
$$
Then by the Powers--St{\o}rmer inequality we have
\begin{multline*}
\|\pi(g)T\pi(g)^*-T\|^2_2 \le \|\pi(g)T^2\pi(g)^*-T^2\|_1 \\
= \frac{1}{|F|}\Big\|\sum_{h\in gF\setminus F}\pi(h)p\pi(h)^*-\sum_{h\in F\setminus gF}\pi(h)p\pi(h)^*\Big\|_1
\le \frac{|gF\Delta F|}{|F|}.
\end{multline*}
Therefore by taking a F{\o}lner net $(F_i)_i$ in $\Gamma$ we obtain a net $(T_i)_i$ in $\HS(H_\pi)$ such that $\|T_i\|_2=1$ and $\|\pi(g)T_i\pi(g)^*-T_i\|_2\to0$ for all $g\in\Gamma$, which by Lemma~\ref{lem:HSformulation} means that $\eps\prec\bar\pi\otimes\pi$.

Similar arguments can be used to prove the following result.

\begin{lemma}[\cite{MR0246999}]\label{lem:greenleaf}
Let $\Gamma$ be an amenable discrete group. Assume we are given an action of~$\Gamma$ on a set~$X$ and consider the corresponding permutation representation $\pi_X$ of $\Gamma$ on~$\ell^2(X)$, so $\pi_X(g)f=f(g^{-1}\cdot)$. Then $\eps\prec\pi_X$.
\end{lemma}

\bp
Fix $x\in X$ and a F{\o}lner net $(F_i)_i$ in $\Gamma$. Define $f_i\in\ell^2(X)$ by
$$
f_i:=\Big(\frac{1}{|F_{i}|}\sum_{h\in F_{i}}\delta_{hx}\Big)^{1/2}.
$$
Using the inequality $\lvert r^{1/2}-s^{1/2}\rvert^{2} \leq |r-s|$ for $r,s \geq 0$, we get $\|\pi_X(g)f_i-f_i\|_2\to0$ for all $g\in\Gamma$. Since $\|f_i\|_2=1$, this proves that $\eps\prec\pi_X$.
\ep

Given two weakly equivalent representations $\pi$ and $\rho$ of a group $\Gamma$, we have $\bar\pi\otimes\pi\sim\bar\pi\otimes\rho$, and hence $\eps\prec\bar\pi\otimes\rho$ if $\pi$ is amenable. We will need the following quantitative version of this.

\begin{lemma}\label{lem:app-intertwiner}
Assume $\Gamma$ is a discrete group and $\varphi$ is a state on $C^*(\Gamma)$ such that the corresponding GNS-representation $\pi_\varphi$ is amenable. Fix $\eps>0$ and a finite subset $F\subset\Gamma$ containing the unit $e\in\Gamma$. Then there exist $\delta>0$ and a finite subset $E\subset\Gamma$ satisfying the following property: if we are given Hilbert spaces $H$ and $H'$,  contractions $L_g\in B(H)$ and $L'_g\in B(H')$ ($g\in F$) and vectors $\xi_h\in H$ and $\xi'_h\in H'$ ($h\in E$) of norm $\le1$ such that
$$
|\varphi(h_2^{-1}gh_1)-(L_{g}\xi_{h_1},\xi_{h_2})|<\delta\quad\text{and}\quad |\varphi(h_2^{-1}gh_1)-(L'_{g}\xi'_{h_1},\xi'_{h_2})|<\delta
$$
for all $g\in F$ and $h_1,h_2\in E$, then there exists a Hilbert--Schmidt operator $T\colon H\to H'$ such that
$$
\|T\|_2=1\quad \text{and}\quad \|TL_g-L'_gT\|_2<\eps\quad\text{for all}\quad g\in F.
$$
\end{lemma}

\bp
We may assume that $\eps<1$. Consider the GNS-representation $\pi=\pi_\varphi\colon C^*(\Gamma)\to B(H_\varphi)$ with the corresponding cyclic vector $\xi=\xi_\varphi$. By amenability of $\pi_\varphi$ and Lemma~\ref{lem:HSformulation} we can find $Q \in \HS( H_\varphi)$ such that $\|Q\|_2=1$ and $\|[\pi(g),Q]\|_2<\eps/2$ for all $g\in F$. For a finite subset $E\subset\Gamma$, denote by $p_E$ the projection onto $\operatorname{span}\{\pi(g) \xi  \mid g\in E\}$. Since $p_E\xrightarrow{so}1$ as $E\nearrow\Gamma$, and hence  $\|Q-p_EQp_E\|_2\to0$ and $\lVert [\pi(g),p_EQp_E] \rVert_{2} \leq 2\lVert  p_EQp_E- Q \rVert_{2}+ \lVert [\pi(g),Q] \rVert_{2}\to \lVert [\pi(g),Q] \rVert_{2}$, we can replace $Q$ by $\|p_EQp_E\|_2^{-1}p_EQp_E$ for a sufficiently large $E$ and therefore assume that $Q=p_EQp_E$. Fix an orthonormal basis $(e_k)^n_{k=1}$ in~$p_EH_\varphi$. Let $(q_{kl})_{k,l}$ be the matrix of~$Q|_{p_EH_\varphi}$ in this basis. Write $e_k=\sum_{h\in E}a_{k,h}\pi(h)\xi$ for some constants~$a_{k,h}$.

Assume now that we are given $H$, $H'$, $L_g$, $L'_g$, $\xi_h$, $\xi'_h$ as in the statement of the lemma, with~$\delta$ to be determined later. Consider the operator
$$
\tilde T\colon H\to H',\quad \tilde T:=\sum^n_{k,k'=1}\sum_{h,h'\in E}q_{k'k}a_{k',h'}\bar a_{k,h}\theta_{\xi'_{h'},\xi_h},
$$
where $\theta_{\xi'_{h'},\xi_h}$ is given by~\eqref{eq:theta}. Since $\lVert L_{e} \xi_{h} - \xi_{h} \rVert^{2} \leq 2-2\operatorname{Re}(L_{e} \xi_{h},\xi_{h})$, the vector $L_{e} \xi_{h}$ must be close to $\xi_{h}$ for $h\in E$ when $\delta$ is small, hence $(\xi_{h_1},\xi_{h_2})$ must be close to $\varphi(h_2^{-1}h_1)$ for $h_1,h_2\in E$. Then
$\sum_{h_1, h_2\in E} a_{l,h_1} \bar a_{k, h_2} (\xi_{h_1}, \xi_{h_2})$ is close to
$$
\sum_{h_2, h_1\in E} a_{l,h_1} \bar a_{k, h_2}  \varphi(h_2^{-1}h_1)=( e_{l}, e_{k}).
$$
Similar calculations reveal that $\sum_{h'_1, h'_2\in E} a_{k',h'_1} \bar a_{l', h'_2} (\xi_{h'_1}', \xi_{h'_2}')$ is close to $( e_{k'}, e_{l'} )$ when $\delta$ is small. This implies that
$$
\lVert \tilde T\rVert_{2}^{2} = \sum_{l, l', k, k'=1}^{n}  q_{ k' k} \bar q_{ l' l} \Big(\sum_{h_1, h_2\in E} a_{l,h_2} \bar a_{k, h_1} (\xi_{h_2}, \xi_{h_1})\Big) \Big( \sum_{h_1', h_2'\in E} a_{k',h_1'} \bar a_{l', h_2'} (\xi_{h_1'}', \xi_{h_2'}')\Big)
$$
is close to $\sum_{k, k'=1}^{n}q_{k' k} \bar q_{k' k}=\|Q\|^2_2=1$.

%In conclusion $\lVert \tilde T\rVert_{2}$ approximates $\lVert Q \rVert_{2}=1$ for $\delta\to 0$.
Next, for any $g\in F$, we have
$$
\lVert \tilde T L_{g} - L'_{g} \tilde T \rVert_{2}^{2}\leq \lVert L_{g}\rVert^{2} \lVert \tilde T  \rVert_{2}^{2} + \lVert L'_{g}\rVert^{2} \lVert \tilde T  \rVert_{2}^{2} - 2\operatorname{Re}(\tilde T L_{g}, L'_{g} \tilde T) \leq 2 \lVert \tilde T  \rVert_{2}^{2} - 2\operatorname{Re}(\tilde T L_{g}, L'_{g} \tilde T) \; .
$$
Since $(\theta_{\xi_{h_1'}', \xi_{h_1}} L_{g} , L'_{g} \theta_{\xi_{h_2'}', \xi_{h_2}})= (L_{g}\xi_{h_2},\xi_{h_1})(\xi_{h_1'} ', L'_{g} \xi_{h_2'} ') $, similar calculations to the ones above show that when $\delta$ is small, $(\tilde T L_{g}, L'_{g} \tilde T)$ becomes close to $(Q \pi(g), \pi(g) Q)$, and therefore the expression $2 \lVert \tilde T  \rVert_{2}^{2} - 2\operatorname{Re}(\tilde T L_{g}, L'_{g} \tilde T)$ is close~to
$$
2-2\operatorname{Re}(Q \pi(g), \pi(g)Q)=\lVert Q \pi(g) - \pi(g) Q \rVert_{2}^{2}<\frac{\eps^2}{4}.
$$

It follows that we can choose $\delta$ small enough, depending only on $\eps$ and our choice of $Q$, $E$ and the constants $a_{k,h}$, so that
$$
\lvert 1 -\|\tilde T\|_2 \rvert <\frac{\eps}{2}\quad\text{and}\quad \|\tilde TL_g-L'_g\tilde T\|_2<\frac{\eps}{2}
$$
for all $g\in F$. Then the operator $T:=\tilde T/\|\tilde T\|_2$ has the right properties.
\ep

Before we proceed, let us write out the following known observation.

\begin{lemma} \label{lem:Hulanicki}
Let $\Gamma$ be a discrete group. Assume there exists an amenable unitary representation~$\pi$ of $\Gamma$ that is weakly contained in the left regular representation $\lambda$. Then $\Gamma$ is amenable.
\end{lemma}
\bp
By assumption we have
$
\eps\prec\bar\pi\otimes\pi\prec\bar\lambda\otimes\lambda\sim\lambda,
$
where the last weak equivalence holds by the Fell absorption principle. Hence $\Gamma$ is amenable by the well-known result of Hulanicki.
\ep

In the terminology of \cite{MR0769602} the following result shows that for discrete groups amenability is equivalent to property (TP2).

\begin{prop}\label{prop:TP2}
For any discrete group $\Gamma$ the following conditions are equivalent:
\begin{enumerate}
  \item $\Gamma$ is amenable;
  \item for all nonzero unitary representations $\pi$ and $\rho$ of $\Gamma$, we have $\eps\prec\bar\pi\otimes\rho$ if and only if $\hull(\ker\pi)\cap\hull(\ker\rho)\ne\emptyset$ in $\Prim C^*(\Gamma)$;
  \item for all irreducible unitary representations $\pi$ and $\rho$ of $\Gamma$, we have $\eps\prec\bar\pi\otimes\rho$ if and only if $\hull(\ker\pi)\cap\hull(\ker\rho)\ne\emptyset$ in $\Prim C^*(\Gamma)$.
\end{enumerate}
\end{prop}

\bp
Obviously, (2) implies (3). Since there always exists an irreducible unitary representation weakly contained in the left regular representation, (3) implies (1) by Lemma~\ref{lem:Hulanicki}.

It remains to prove that (1) implies (2). Assume therefore that $\Gamma$ is amenable. Take two nonzero unitary representations $\pi$ and $\rho$ of $\Gamma$. Assume $\hull(\ker\pi)\cap\hull(\ker\rho)\ne\emptyset$, so there exists an irreducible unitary representation $\tau$ weakly contained in both $\pi$ and $\rho$. Then $\tau$ is amenable by amenability of $\Gamma$, hence $\eps\prec \bar\tau\otimes\tau\prec \bar\pi\otimes\rho$.

Conversely, assume $\eps\prec \bar\pi\otimes\rho$. By Lemma~\ref{lem:HSformulation} there exists a net of operators $T_i \in \HS( H_\pi, H_\rho)$ such that $\|T_i\|_2=1$ and $\|T_i\pi(g)-\rho(g)T_i\|_2\to0$ for all $g\in\Gamma$. Consider the states
$$
\varphi_i:=\Tr(\pi(\cdot)T_i^*T_i)=\Tr(T_i\pi(\cdot)T_i^*)\quad\text{and}\quad \psi_i:=\Tr(\rho(\cdot)T_iT^*_i)
$$
on $C^*(\Gamma)$. By passing to a subnet we may assume that we have a weak$^*$ convergence $\varphi_i\to\varphi$ for some state $\varphi$. As $|\varphi_i(g)-\psi_i(g)|\to0$ by the Cauchy--Schwarz inequality, we also have $\psi_i\to\varphi$. Therefore $\varphi$ is a state on $C^*(\Gamma)$ that factors through both $C^*_\pi(\Gamma)$ and $C^*_\rho(\Gamma)$. Then $\hull(\ker\pi_\varphi)\subset\hull(\ker\pi)\cap\hull(\ker\rho)$. Thus, (1) implies (2).
\ep

Note that the proof of the implication ($\eps\prec \bar\pi\otimes\rho\Rightarrow \hull(\ker\pi)\cap\hull(\ker\rho)\ne\emptyset$) works without any amenability assumptions. Therefore the same arguments as in the above proof give the following result.

\begin{cor}\label{cor:Frob1}
Assume $\Gamma$ is a discrete group, $\pi$ and $\rho$ are nonzero unitary representations of~$\Gamma$. Assume also that $\pi$ is amenable and $C^*_\pi(\Gamma)$ is simple. Then $\eps\prec\bar\pi\otimes\rho$ if and only if $\pi\prec\rho$.
\end{cor}

We also get the following result, which for countable groups has been already obtained in~\cite{MR0769602}*{Corollary~2}.

\begin{cor}
For any discrete group $\Gamma$ the following conditions are equivalent:
\begin{enumerate}
  \item $\Gamma$ is amenable and every primitive ideal of $C^*(\Gamma)$ is maximal;
  \item for all irreducible unitary representations $\pi$ and $\rho$ of $\Gamma$, we have $\eps\prec\bar\pi\otimes\rho$ if and only if $\pi\sim\rho$.
\end{enumerate}
\end{cor}

\bp
The implication (1)$\Rightarrow$(2) follows from the previous corollary. Conversely, assume (2) holds. By Lemma~\ref{lem:Hulanicki} then $\Gamma$ is amenable. Next, assume there exists an irreducible unitary representation $\pi$ such that $\ker\pi$ is not maximal. This implies that there is an irreducible unitary representation $\rho$ such that $\rho\prec\pi$, but $\rho\not\sim\pi$. As $\eps\prec\bar\pi\otimes\rho$ by Proposition~\ref{prop:TP2}(3), this contradicts our assumption (2). Hence (2)$\Rightarrow$(1).
\ep

Note that by a result of Echterhoff~\cite{MR1066810} if $\Gamma$ is FC-hypercentral, then $\Gamma$ is amenable and every primitive ideal of $C^*(\Gamma)$ is maximal, and the converse is true if $\Gamma$ is countable. We will return to this in Section~\ref{ssec:main}.

\subsection{Duality of induction and restriction}

We next turn to induced representations and a weak analogue of the  isomorphism
\begin{equation}\label{eq:Frobenius2}
\operatorname{Hom}_\Gamma(\Ind^\Gamma_S\rho,\pi)\cong\operatorname{Hom}_S(\rho,\Res^\Gamma_S\pi)
\end{equation}
for finite groups, which is what one more often means by Frobenius reciprocity instead of~\eqref{eq:Frobenius}.

First let us recall the following generalization of the Fell absorption principle: if $\Gamma$ is a discrete group, $S\subset \Gamma$ is a subgroup, $\pi$ is a unitary representation of $\Gamma$ and $\rho$ is a unitary representation of $S$, then we have a unitary equivalence
\begin{equation}\label{eq:absorption}
\pi\otimes\Ind^\Gamma_S\rho\cong \Ind^\Gamma_S(\pi|_S\otimes\rho).
\end{equation}
Explicitly, a unitary intertwiner $H_\pi\otimes \Ind^\Gamma_S H_\rho\to \Ind^\Gamma_S(H_\pi\otimes H_\rho)$ is obtained by mapping $\xi\otimes f\in H_\pi\otimes \Ind^\Gamma_S H_\rho$ into the function $(g\mapsto\pi(g)^*\xi\otimes f(g))$.

In particular, if $\Gamma$ is amenable, we can take $\rho=\eps_S$, use that $\eps_\Gamma\prec \Ind^\Gamma_S\eps_S$ by Lemma~\ref{lem:greenleaf} applied to $\Gamma \curvearrowright \Gamma / S$, and conclude that
$$
\pi\prec \Ind^\Gamma_S(\pi|_S).
$$
This is a result of Greenleaf~\cite{MR0246999}*{Theorem~5.1}. The following is a related result that is closer to~\eqref{eq:Frobenius2}.

\begin{prop}\label{prop:Frob2}
Assume $\Gamma$ is an amenable discrete group, $S\subset\Gamma$ is a subgroup, $\pi$ is a nonzero unitary representation of $\Gamma$ and $\rho$ is a nonzero unitary representation of $S$ such that $\rho\prec\pi|_S$. Then $\eps_\Gamma\prec\bar\pi\otimes\Ind^\Gamma_S\rho$.
\end{prop}

\bp
We can argue as in the proof of \cite{MR0884559}*{Lemma~1.1}. Since $S$ is amenable, we have $\eps_S\prec\bar\rho\otimes\rho$. Using first that $\eps_\Gamma\prec \Ind^\Gamma_S\eps_S$ by Lemma~\ref{lem:greenleaf} and then applying again~\eqref{eq:absorption}, we get
$$
\eps_\Gamma\prec \Ind^\Gamma_S\eps_S\prec \Ind^\Gamma_S(\bar\rho\otimes\rho)\prec \Ind^\Gamma_S(\bar\pi|_S\otimes\rho)\cong\bar\pi\otimes\Ind^\Gamma_S\rho,
$$
which proves the proposition.
\ep

Together with Corollary~\ref{cor:Frob1} this gives the following result.

\begin{cor}\label{cor:Frob2}
In the setting of Proposition~\ref{prop:Frob2} assume in addition that $C^*_\pi(\Gamma)$ is simple. Then $\pi\prec \Ind^\Gamma_S\rho$.
\end{cor}

The following corollary has been already obtained in~\cite{MR0884559}*{Theorem~1.2} for countable groups by more sophisticated methods.

\begin{cor}\label{cor:Frob3}
For any discrete group $\Gamma$ the following conditions are equivalent:
\begin{enumerate}
  \item $\Gamma$ is amenable and every primitive ideal of $C^*(\Gamma)$ is maximal;
  \item for every subgroup $S\subset\Gamma$ and all irreducible unitary representations $\pi$ of $\Gamma$ and $\rho$ of $S$ such that $\rho\prec\pi|_S$, we have $\pi\prec \Ind^\Gamma_S\rho$.
\end{enumerate}
\end{cor}

\bp
The implication (1)$\Rightarrow$(2) follows from the previous corollary. Conversely, assume (2) holds. Taking $S=\{e\}$ and $\pi=\eps_\Gamma$ we see that $\eps_\Gamma$ is weakly contained in the regular representation of $\Gamma$, so $\Gamma$ is amenable. Taking $S=\Gamma$ we also conclude that if two irreducible representations~$\pi$ and~$\rho$ of $\Gamma$ satisfy $\rho\prec\pi$, then $\rho\sim\pi$. This is only possible if every primitive ideal of $C^*(\Gamma)$ is maximal.
\ep

Finally, let us remark that by~\cite{MR0884559}*{Corollary~2.8} we have $(\rho\prec\pi|_S\Leftrightarrow\pi\prec \Ind^\Gamma_S\rho)$ for all subgroups $S$ and all irreducible unitary representations $\pi$ and $\rho$ if and only if $\Gamma$ is an FC-group, that is, every conjugacy class in $\Gamma$ is finite.

\bigskip

\section{Weak containment of induced representations}\label{sec:main}

\subsection{Necessary conditions for weak containment}
We now turn to the central technical result of the paper. In order to formulate it, let us introduce the following notation. Given $x\in\Gu$, a unitary representation $\pi$ of $\Gxx$ and an element $g\in\G_x$, we define a representation~$\pi^g$ of~$\G^{r(g)}_{r(g)}$ on~$H_\pi$~by
$$
\pi^g(h):=\pi(g^{-1}hg).
$$
Recall that $\Ind^\G_{\Gxx}\pi$ and $\Ind^\G_{\G^{r(g)}_{r(g)}}\pi^g$ are unitarily equivalent (see~\cite{CN3}*{Lemma~1.1}).

\begin{thm}\label{thm:main1}
Assume $\G$ is a Hausdorff locally compact \'etale groupoid, $x\in\Gu$, $\Gamma\subset\Gxx$ is a subgroup. For every $g\in\Gamma$, fix an open bisection $W_g$ containing $g$. Assume we are given a unitary representation $\pi$ of $\Gxx$ and an amenable irreducible unitary representation $\rho$ of $\Gamma$ such that $\rho\prec\pi|_\Gamma$. Assume also that we are given a subset $Y\subset\Gu$ and, for every $y\in Y$, a unitary representation $\pi_y$ of $\G^y_y$. Consider the following conditions:
\begin{enumerate}
  \item $\Ind^\G_{\Gxx}\pi\prec\bigoplus_{y\in Y}\Ind^\G_{\G^y_y}\pi_y$;
  \item for every neighbourhood $U$ of $x$, every finite set $F\subset\Gamma$ and every $\eps>0$, there exist $y\in Y$, $\gamma\in\G_y$ and a Hilbert--Schmidt operator $T\colon H_\rho\to H_{\pi_y}$ such that $\|T\|_2=1$, $r(\gamma)\in U$ and, for every $g\in F$,  we have either $W_g\cap\G^{r(\gamma)}_{r(\gamma)}=\emptyset$, or $W_g\cap\G^{r(\gamma)}_{r(\gamma)}=\{h\}$ for some $h$ and $\|T\rho(g)-\pi^{\gamma}_{y}(h)T\|_2<\eps$;
\item there is a state $\varphi$ on $C^*_\rho(\Gamma)$ such that for every neighbourhood $U$ of $x$, every finite set $F\subset\Gamma$ and every $\eps>0$, there exist $y\in Y$, $\gamma\in\G_y$ and a  state $\nu$ on $C^*_{\pi^{\gamma}_{y}}(\G^{r(\gamma)}_{r(\gamma)})$ such that the following properties are satisfied: $r(\gamma)\in U$ and, for every $g\in F$,  we have either $W_g\cap\G^{r(\gamma)}_{r(\gamma)}=\emptyset$, or $W_g\cap\G^{r(\gamma)}_{r(\gamma)}=\{h\}$ for some $h$ and $|\varphi(g)-\nu(h)|<\eps$.
\end{enumerate}
Then $(1)\Rightarrow(2)\Rightarrow(3)$.
\end{thm}

Since one-dimensional representations are obviously amenable, this theorem, together with Remark~\ref{rem:strength_main} below, generalize~\cite{CN3}*{Theorem~2.6}. In fact, even when $\pi$ is one-dimensional, we get a stronger result than in~\cite{CN3}, since we make no assumptions on the representations~$\pi_y$.

\bp[Proof of Theorem~\ref{thm:main1}]
(2)$\Rightarrow$(3) Using (2) we can find a net $(y_i)_{i\in I}$ in $Y$, elements $\gamma_i\in\G_{y_i}$ and Hilbert--Schmidt operators $T_i\colon H_\rho\to H_{\pi_{y_i}}$ with $\|T_i\|_2=1$ such that the following properties are satisfied: $r(\gamma_i)\to x$  and, for every $\eps>0$ and $g\in\Gamma$, there is an index $i_0$ such that for all $i\ge i_0$ we have either $W_g\cap\G^{r(\gamma_i)}_{r(\gamma_i)}=\emptyset$, or $W_g\cap\G^{r(\gamma_i)}_{r(\gamma_i)}=\{h\}$ for some $h$ and
$\|T_i\rho(g)-\pi^{\gamma_i}_{y_i}(h)T_i\|_2<\eps$. Then we get~(3) by arguing similarly to the proof of Proposition~\ref{prop:TP2}: by passing to a subnet we may assume that the states $\varphi_i:=\Tr(\rho(\cdot)T^*_iT_i)$ converge weakly$^*$ to a state $\varphi$, and then let $y:=y_i$, $\gamma:=\gamma_i$ and $\nu:=\Tr(\pi^{\gamma_i}_{y_i}(\cdot)T_iT_i^*)$ for $i$ large enough.

\smallskip

(1)$\Rightarrow$(2) Fix an open neighbourhood $U$ of $x$ in $\Gu$, $\eps>0$ and a finite set $F\subset\Gamma\setminus\{e\}$, where $e=x$ is the unit of $\Gamma\subset\Gxx$. We assume that $U$ is small enough so that $U\subset W_e$. We may then assume that $F$ is nonempty, as otherwise there is nothing to prove. For every $g\in\Gamma\setminus\{e\}$, choose a function $f_g\in C_c(W_g)$ such that $0\le f_g(h)\le 1$ for all $h\in W_g$ and $f_g(h)=1$ for all $h$ in some neighbourhood of $g$. By replacing $U$ by a smaller neighbourhood of~$x$ if necessary, we may assume that for all $g\in F$ and $h\in r^{-1}(U)\cap W_g$ we have $f_g(h)=1$. Let us also choose a function $f\in C_c(U)$ such that $0\le f \le 1$ and $f(x)=1$. Put $f_e:=f$. Fix a number $\alpha\in(0,|F|^{-1})$.

As a first step we will prove that there exist a point $y\in Y$ and a Hilbert--Schmidt operator $S\colon H_\rho\to \Ind^\G_{\Gyy} H_{\pi_y}$ such that $\|S\|_2=1$,
\begin{equation}\label{eq:main-S1}
\|S-(\Ind^\G_{\Gyy}\pi_y)(f)S\|_2<1-(|F|\alpha)^{1/2},
\end{equation}
\begin{equation}\label{eq:main-S2}
 \|S\rho(g)-(\Ind^\G_{\Gyy}\pi_y)(f_g)S\|_2<\alpha^{1/2}\eps \quad\text{for all}\quad g\in F.
\end{equation}

For this, fix a unit vector $\zeta\in H_\rho$ and consider the pure state $\varphi:=(\rho(\cdot)\zeta,\zeta)$ on $C^*_\rho(\Gamma)$. Since~$C^*_\rho(\Gamma)$ is a quotient of $C^*_{\pi|_\Gamma}(\Gamma)$ and the latter algebra can be viewed as a subalgebra of~$C^*_\pi(\Gxx)$, we can extend $\varphi$ to a pure state on $C^*_\pi(\Gxx)$. We denote this extension by $\psi$ and view it as a state on $C^*(\Gxx)$. Consider the GNS-triple $(H_\psi,\pi_\psi,\xi_\psi)$ associated with $\psi$. As $\pi_\psi$ is an irreducible representation of~$\Gxx$ weakly contained in $\pi$, the representation $\Ind^\G_{\Gxx}\pi_\psi$ is irreducible and weakly contained in $\bigoplus_{y\in Y}\Ind^\G_{\G^y_y}\pi_y$. We view $H_\psi$ as a subspace of $\Ind^\G_{\Gxx}H_\psi$. By Lemma~\ref{lem:Fell} we can then find a net $(y_j)_{j\in J}$ in $Y$ and unit vectors $\eta_j\in\Ind^\G_{\G^{y_j}_{y_j}} H_{\pi_{y_j}}$ such that if $W$ is an open bisection containing an element $g\in\Gamma$ and $\tilde f_g\in C_c(W)$, then
\begin{equation}\label{eq:conv-main2}
\big((\Ind^\G_{\G^{y_j}_{y_j}}\pi_{y_j})(\tilde f_g)\eta_j,\eta_j\big)\to_j \big((\Ind^\G_{\Gxx}\pi_\psi)(\tilde f_g)\xi_\psi,\xi_\psi\big)=\tilde f_g(g)\psi(g)=\tilde f_g(g)\varphi(g).
\end{equation}

To find $S$ we now want to apply Lemma~\ref{lem:app-intertwiner} to the state~$\varphi$, the set $F\cup\{e\}$, the operators
$$
L_g:=\rho(g)\quad\text{and}\quad L'_g:=(\Ind^\G_{\G^{y_j}_{y_j}}\pi_{y_j})(f_g)\quad (g\in F\cup\{e\}),
$$
and the vectors
$$
\xi_h:=\rho(h)\zeta \quad\text{and}\quad \xi'_h:=(\Ind^\G_{\G^{y_j}_{y_j}}\pi_{y_j})(f_h)\eta_j\quad(h\in\Gamma),
$$
where $j\in J$ is to be determined. Note that we can identify the GNS-representation $\pi_\varphi$ with $\rho$, so $\pi_\varphi$ is amenable.
We claim that the assumptions on~$L'_g$ and~$\xi'_h$ in Lemma~\ref{lem:app-intertwiner} are satisfied for any fixed~$\eps$ and~$j$ large enough. For this we need to show that
\begin{equation*}\label{eq:conv-main4}
\big((\Ind^\G_{\G^{y_j}_{y_j}}\pi_{y_j})(f_{h_2}^**f_{g}*f_{h_1})\eta_j,\eta_j\big)\to_j\varphi(h_2^{-1}gh_1)
\end{equation*}
for all $g\in F\cup\{e\}$ and $h_1,h_2\in\Gamma$. But this follows immediately from~\eqref{eq:conv-main2}, since $f_{h_2}^**f_{g}*f_{h_1}$ is supported on the bisection $W_{h_2}^{-1}W_{g}W_{h_1}$ and takes value~$1$ at $h_2^{-1}gh_1$.

We can therefore apply Lemma~\ref{lem:app-intertwiner} and conclude that there exists $y=y_j$ and a Hilbert--Schmidt operator $S\colon H_\rho\to \Ind^\G_{\Gyy} H_{\pi_y}$, with $\|S\|_2=1$, satisfying conditions~\eqref{eq:main-S1} and~\eqref{eq:main-S2}.

\smallskip

For each $z\in[y]:=r(\G_y)$, denote by $\HH_z\subset\Ind^\G_{\Gyy}H_{\pi_y}$ the subspace of functions supported on $\G^z_y$. Then $\Ind^\G_{\Gyy}H_{\pi_y}=\bigoplus_{z\in[y]}\HH_z$. We can therefore write $S\colon H_\rho\to\Ind^\G_{\Gyy}H_{\pi_y}$ as an $so$-convergent sum $\sum_{z\in[y]}S_z$ of operators $S_z\colon H_\rho\to\HH_z$.

For every $z\in[y]$, fix an element $r_z\in\G^z_y$. Then the map $u_z\colon\HH_z\to H_{\pi_y}$, $\xi\mapsto\xi(r_z)$, is a unitary isomorphism. We put $T_z:=u_zS_z$, so $T_z$ is a Hilbert--Schmidt operator $H_\rho\to H_{\pi_y}$. We claim that there is $z\in [y]\cap U$ such that $\|T_z\|_2>0$ and, for every $g\in F$, we have either $W_g\cap\G^{z}_{z}=\emptyset$, or $W_g\cap\G^{z}_{z}=\{h\}$ for some $h$ and
$$
\|T_z\rho(g)-\pi^{r_z}_{y}(h)T_z\|_2<\eps\|T_z\|_2.
$$
If this is true, then $\gamma:=r_z$ and $T:=T_z/\|T_z\|_2$ have the properties we need.

Assume the claim is not true. As $\|S\|_2=1$,  from~\eqref{eq:main-S1} we get that
$$
\big|1-\|(\Ind^\G_{\Gyy}\pi_y)(f)S\|_2\big|<1-(|F|\alpha)^{1/2}.
$$
We have $(\Ind^\G_{\Gyy}\pi_y)(f)S=\sum_{z\in[y]} f(z)S_z$, hence $\|(\Ind^\G_{\Gyy}\pi_y)(f)S\|_2^2=\sum_{z\in[y]} |f(z)|^2\|S_z\|_2^2$ . It follows that $\sum_{z\in[y]} |f(z)|^2\|S_z\|_2^2>|F|\alpha$, and since $f$ is supported on $U$ and $0\le f\le1$, we get
\begin{equation}\label{eq:main-S3}
\sum_{z\in [y]\cap U}\|S_z\|_2^2>|F|\alpha.
\end{equation}

For every $g\in F$, denote by $Z_g$ the set of points $z\in[y]\cap U$ such that $\|T_z\|_2=\|S_z\|_2>0$, $W_g\cap\G^z_z=\{h_{g,z}\}$ for some~$h_{g,z}$ and
\begin{equation}\label{eq:main-S4}
\|T_z\rho(g)-\pi^{r_z}_{y}(h_{g,z})T_z\|_2\ge\eps\|T_z\|_2.
\end{equation}
By our assumption $\bigcup_{g\in F}Z_g$ contains all points $z\in [y]\cap U$ such that $\|S_z\|_2>0$. Therefore by~\eqref{eq:main-S3} we have
$$
\sum_{g\in F}\sum_{z\in Z_g}\|S_z\|_2^2\ge\sum_{z\in [y]\cap U}\|S_z\|_2^2>|F|\alpha.
$$
It follows that there is $g\in F$ such that
\begin{equation}\label{eq:main-S5}
\sum_{z\in Z_g}\|S_z\|_2^2>\alpha.
\end{equation}

Observe next that for every $z\in Z_g$ the space $\HH_z$ is an invariant subspace for the C$^*$-algebra generated by $(\Ind^\G_{\Gyy}\pi_y)(f_g)$, hence $\bigoplus_{z\in[y]\setminus Z_g}\HH_z$ is an invariant subspace as well. Indeed, from the definition~\eqref{eq:ind-rep} of an induced representation we get, for each $\xi\in\HH_z$ and $z'\in[y]$, that
$$
\big((\Ind^\G_{\Gyy}\pi_y)(f_g)\xi\big)(r_{z'})
=\sum_{h\in \G^{z'}} f_{g}(h) \xi(h^{-1} r_{z'})
=\sum_{h\in \G^{z'}_{z}} f_{g}(h) \xi(h^{-1}r_{z'}).
$$
The last expression is zero for $z'\ne z$, since $f_g$ is supported on $W_g$ and $W_g\cap \G_z=W_g\cap \G^z_z=\{h_{g,z}\}$. A similar computation shows that $\HH_z$ is invariant under $(\Ind^\G_{\Gyy}\pi_y)(f_g)^*$. Moreover, by continuing the above computation for $z'=z$ and recalling that $f_g(h_{g,z})=1$, we get
$$
\big((\Ind^\G_{\Gyy}\pi_y)(f_g)\xi\big)(r_z)=\xi(h_{g,z}^{-1}r_z)=\pi^{r_z}_y(h_{g,z})\xi(r_z).
$$
In other words, for all $z\in Z_g$, we have
\begin{equation}\label{eq:main-S6}
u_z(\Ind^\G_{\Gyy}\pi_y)(f_g)|_{\HH_z}=\pi^{r_z}_y(h_{g,z})u_z.
\end{equation}

Put $S':=\sum_{z\in[y]\setminus Z_g}S_z$. Then the invariance we have established implies that
\begin{multline*}
\|S\rho(g)-(\Ind^\G_{\Gyy}\pi_y)(f_g)S\|_2^2\\=\|S'\rho(g)-(\Ind^\G_{\Gyy}\pi_y)(f_g)S'\|_2^2+\sum_{z\in Z_g}\|S_z\rho(g)-(\Ind^\G_{\Gyy}\pi_y)(f_g)S_z\|_2^2.
\end{multline*}
In view  of~\eqref{eq:main-S6} this gives
$$
\|S\rho(g)-(\Ind^\G_{\Gyy}\pi_y)(f_g)S\|_2^2\ge \sum_{z\in Z_g}\|T_z\rho(g)-\pi^{r_z}_y(h_{g,z})T_z\|_2^2.
$$
By \eqref{eq:main-S4} and \eqref{eq:main-S5} the last expression is not smaller than
$$
\sum_{z\in Z_g}\eps^2\|T_z\|_2^2>\alpha\eps^2.
$$
This contradicts \eqref{eq:main-S2}, and the proof of theorem is complete.
\ep

\begin{remark}\label{rem:strength_main}
Since every $\pi_y$ is weakly equivalent to a direct sum of irreducible representations, essentially the same proof shows that (1) implies the following property that is formally stronger than (2): for every neighbourhood $U$ of $x$, every finite set $F\subset\Gamma$ and every $\eps>0$, there exist $y\in Y$, $\gamma\in\G_y$, an irreducible unitary representation $\theta$ of $\G^y_y$ and a Hilbert--Schmidt operator $T\colon H_\rho\to H_{\theta}$ such that $\|T\|_2=1$, $r(\gamma)\in U$, $\theta\prec\pi_y$ and, for every $g\in F$,  we have either $W_g\cap\G^{r(\gamma)}_{r(\gamma)}=\emptyset$, or $W_g\cap\G^{r(\gamma)}_{r(\gamma)}=\{h\}$ for some $h$ and $\|T\rho(g)-\theta^{\gamma}(h)T\|_2<\eps$. Correspondingly, in~(3) we may in addition require that $\nu$ is a state on $C^*_{\theta^{\gamma}}(\G^{r(\gamma)}_{r(\gamma)})$ for some irreducible $\theta\prec\pi_y$. This could be used to slightly simplify a couple of subsequent arguments at the expense of having to keep track of more representations in the proof of Theorem~\ref{thm:main1}.
\end{remark}

\subsection{Criteria of weak containment}
In general, even if $\pi$ is irreducible and we take $\Gamma=\Gxx$ and $\rho=\pi$ in Theorem~\ref{thm:main1}, condition (3) there is weaker than (1). One of the issues is that $\ker\pi_\varphi$ can be much larger than $\ker\rho$, see Section~\ref{ssec:FC}. But under extra assumptions we do get an equivalence between all three conditions, as follows from the equivalence of~(1) and~(4) in the following theorem.

\begin{thm}\label{thm:main2}
Assume $\G$ is a Hausdorff locally compact \'etale groupoid and $x\in\Gu$ is a unit such that the group $\Gxx$ is amenable. For every $g\in\Gxx$, fix an open bisection $W_g$ containing~$g$. Assume we are given a unitary representation~$\pi$ of~$\Gxx$, a subset $Y\subset\Gu$ and, for every $y\in Y$, a unitary representation $\pi_y$ of $\G^y_y$. Assume also that $(\Gamma_k)_{k\in K}$ is a collection of subgroups of~$\Gxx$ such that every finite subset of $\Gxx$ is contained in some $\Gamma_k$ and, for every $k\in K$, the ideal $\ker (\pi|_{\Gamma_k})\subset C^*(\Gamma_k)$ is an intersection of maximal ideals. Then the following conditions are equivalent:
 \begin{enumerate}
  \item $\Ind^\G_{\Gxx}\pi\prec\bigoplus_{y\in Y}\Ind^\G_{\G^y_y}\pi_y$;
  \item for every $k\in K$ and every irreducible unitary representation $\rho\prec\pi|_{\Gamma_k}$ such that $\ker\rho\subset C^*(\Gamma_k)$ is maximal, there exist a state~$\varphi$ on $C^*_\rho(\Gamma_k)$, a net $(y_i)_{i\in I}$ in $Y$, elements $g_i\in\G_{y_i}$ and states $\nu_i$ on $C^*_{\pi^{g_i}_{y_i}}(\G^{r(g_i)}_{r(g_i)})$ such that the following properties are satisfied: $r(g_i)\to x$  and, for every $\eps>0$ and $g\in\Gamma_k$, there is an index $i_0$ such that for all $i\ge i_0$ we have either $W_g\cap\G^{r(g_i)}_{r(g_i)}=\emptyset$, or $W_g\cap\G^{r(g_i)}_{r(g_i)}=\{h\}$ for some $h$ and $|\varphi(g)-\nu_i(h)|<\eps$;
\item $\pi\prec\bigoplus_{(S,\omega)}\Ind^{\Gxx}_S\pi_\omega$, where the summation is taken over all pairs $(S,\omega)$ such that $S\subset\Gxx$ is a subgroup, $\omega$ is a state on $C^*(S)$ and the following property holds: there exist a net $(y_i)_{i\in I}$ in~$Y$, elements $g_i\in\G_{y_i}$ and states $\omega_i$ on $C^*_{\pi^{g_i}_{y_i}}(\G^{r(g_i)}_{r(g_i)})$ such that $r(g_i)\to x$, $\G^{r(g_i)}_{r(g_i)}\to S$ in $\Sub(\G)$ and, for every $g\in S$, we have $\omega_i(h_{g,i})\to\omega(g)$, where $h_{g,i}$ is the unique element of $W_g\cap \G^{r(g_i)}_{r(g_i)}$ (which exists for $i$ large enough).
\end{enumerate}
If, in addition, the ideal $\ker\pi\subset C^*(\Gxx)$ is maximal, then conditions {\rm (1)--(3)} are equivalent~to
\begin{enumerate}
  \item[(4)] there are a subgroup $S\subset\Gxx$ and a state $\omega$ on $C^*_{\pi|_S}(S)$ such that the following property holds: there exist a net $(y_i)_{i\in I}$ in~$Y$, elements $g_i\in\G_{y_i}$ and states $\omega_i$ on $C^*_{\pi^{g_i}_{y_i}}(\G^{r(g_i)}_{r(g_i)})$ such that $r(g_i)\to x$, $\G^{r(g_i)}_{r(g_i)}\to S$ in $\Sub(\G)$ and, for every $g\in S$, we have $\omega_i(h_{g,i})\to\omega(g)$, where $h_{g,i}$ is the unique element of $W_g\cap \G^{r(g_i)}_{r(g_i)}$.
\end{enumerate}
\end{thm}

\bp
%We will prove the implications (1)$\Rightarrow$(2)$\Rightarrow$(3)$\Rightarrow$(1) and (1)$\Rightarrow$(4)$\Rightarrow$(1).
%\smallskip
(1)$\Rightarrow$(2) This follows from Theorem~\ref{thm:main1}.

\smallskip

(1)$\Rightarrow$(4) Assuming that $\ker\pi$ is maximal, by Theorem~\ref{thm:main1} we can find $\varphi$, $y_i$, $g_i$ and $\nu_i$ as in~(2), but with $\Gamma_k$ replaced by $\Gxx$ and $\rho$ any irreducible representation with $\ker\rho=\ker\pi$. By passing to a subnet we may assume that $\G^{r(g_i)}_{r(g_i)}\to S$ in $\Sub(\G)$ for some subgroup $S\subset\Gxx$. By Lemma~\ref{lem:subG-convergence} this implies in particular that, for every $g\in S$, we have $W_g\cap \G^{r(g_i)}_{r(g_i)}\ne\emptyset$ for all $i$ large enough. Then we put $\omega:=\varphi|_{C^*(S)}$, where we view $\varphi$ as a state on $C^*(\Gxx)$, and $\omega_i:=\nu_i$. Note that since $\varphi$ factors through $C^*_\rho(\Gxx)$ and $\rho\sim\pi$, we have $\rho|_S\sim\pi|_S$ and therefore $\omega$ factors through~$C^*_{\pi|_S}(S)$.

\smallskip

(4)$\Rightarrow$(1) Assume again that $\ker\pi$ is maximal. Since $\pi_\omega\prec\pi|_S$, by Corollary~\ref{cor:Frob2} we get $\pi\prec\Ind^{\Gxx}_S\pi_\omega$. Hence
$$
\Ind^\G_{\Gxx}\pi\prec\Ind^\G_{\Gxx}\Ind^{\Gxx}_S\pi_\omega\cong\Ind^\G_S\pi_\omega.
$$
On the other hand, if $y_i$, $g_i$ and $\omega_i$ are as in (4), then by Lemma~\ref{lem:Ind-continuity} we have
$$
\Ind^\G_S\pi_\omega\prec\bigoplus_i\Ind^\G_{\G^{r(g_i)}_{r(g_i)}}\pi_{\omega_i}\prec \bigoplus_i\Ind^\G_{\G^{r(g_i)}_{r(g_i)}}\pi^{g_i}_{y_i}\cong \bigoplus_i\Ind^\G_{\G^{y_i}_{y_i}}\pi_{y_i}.
$$
It follows that $\Ind^\G_{\Gxx}\pi\prec\bigoplus_{y\in Y}\Ind^\G_{\G^y_y}\pi_y$.

\smallskip

(3)$\Rightarrow$(1) This is similar to the implication (4)$\Rightarrow$(1): the assumption $\pi\prec\bigoplus_{(S,\omega)}\Ind^{\Gxx}_S\pi_\omega$ together with the unitary equivalence $\Ind^\G_{\Gxx}\big( \bigoplus_{(S,\omega)}\Ind^{\Gxx}_S\pi_\omega \big) \simeq \bigoplus_{(S,\omega)}\Ind^\G_S\pi_\omega$ imply that
$\Ind^\G_{\Gxx}\pi\prec\bigoplus_{(S,\omega)}\Ind^\G_S\pi_\omega$, while properties of $(S,\omega)$ and Lemma~\ref{lem:Ind-continuity} imply that
$\Ind^\G_S\pi_\omega\prec\bigoplus_{y\in Y}\Ind^\G_{\G^y_y}\pi_y$. Hence $\Ind^\G_{\Gxx}\pi\prec\bigoplus_{y\in Y}\Ind^\G_{\G^y_y}\pi_y$.

\smallskip

(2)$\Rightarrow$(3) This is a more elaborate version of the arguments used in the proof of (1)$\Leftrightarrow$(4). Fix a unit vector $\xi\in H_\pi$ and consider the state $(\pi(\cdot)\xi,\xi)$ on $C^*_\pi(\Gxx)$. Fix a finite subset $F\subset\Gxx$. Let $k$ be such that $F\subset\Gamma_k$. Since $\ker (\pi|_{\Gamma_k})$ is an intersection of maximal ideals, we can find irreducible unitary representations $\rho_1,\dots,\rho_m$ of~$\Gamma_k$, states $\psi_l$ on $C^*_{\rho_l}(\Gamma_k)$ and numbers $\alpha_l>0$ such that $\rho_l\prec\pi|_{\Gamma_k}$, the ideals $\ker\rho_l\subset C^*(\Gamma_k)$ are maximal, $\sum_l\alpha_l=1$ and
\begin{equation}\label{eq:psi1}
\Big|(\pi(g)\xi,\xi)-\sum^m_{l=1}\alpha_l\psi_l(g)\Big|<\delta\quad\text{for all}\quad g\in F.
\end{equation}

On the other hand, we have states $\varphi_l$ on $C^*_{\rho_l}(\Gamma_k)$ given by condition (2) of the theorem. As before, we view $\varphi_l$ as a state on $C^*(\Gamma_k)$. We claim that for every index $l$ there are a subgroup $S_l\subset\Gxx$ and a state $\omega_l$ on $C^*(S_l)$ such that $(S_l,\omega_l)$ has the property formulated in (3) and $\varphi_l=\omega_l$ on $C^*(\Gamma_k\cap S_l)$.

Assume the claim is proved. For every $l$, put $\varphi'_l:=\varphi_l|_{C^*(\Gamma_k\cap S_l)}$ and consider the corresponding GNS-representation $\pi_l'$ on $H_l'$. We can identify $H'_l$ with a subspace of $H_{\omega_l}$, namely, with the smallest closed $(\Gamma_k\cap S_l)$-invariant subspace containing the cyclic vector. We can then view $\Ind^{\Gamma_k}_{\Gamma_k\cap S_l}H'_l$ as a subspace of $\Ind^{\Gamma_k}_{\Gamma_k\cap S_l}H_{\omega_l}$, which implies that $\Ind^{\Gamma_k}_{\Gamma_k\cap S_l}\pi_l'$ is a subrepresentation of $\Ind^{\Gamma_k}_{\Gamma_k\cap S_l}\pi_{\omega_{l}}$. Next, for every $\xi \in \Ind^{\Gamma_k}_{\Gamma_k\cap S_l}H_{\omega_l}$, we can define a vector $\tilde{\xi} \in \Ind^{\Gxx}_{S_l}H_{\omega_l}$ by setting $\tilde{\xi}(gs):= \pi_{\omega_{l}}(s)^{*} \xi(g)$ for $g\in \Gamma_{k}$ and $s\in S_{l}$ and $\tilde{\xi}(h):=0$ for $h\notin \Gamma_{k}S_{l}$, and thereby identify $\Ind^{\Gamma_k}_{\Gamma_k\cap S_l}H_{\omega_l}$ with the subspace of $\Ind^{\Gxx}_{S_l}H_{\omega_l}$ consisting of functions that are zero on $\Gxx\setminus \Gamma_{k}S_{l}$. This shows that $\Ind^{\Gamma_k}_{\Gamma_k\cap S_l}\pi_{\omega_{l}}$ is a subrepresentation of $(\Ind^{\Gxx}_{S_l}\pi_{\omega_l})|_{\Gamma_k}$. Therefore $\Ind^{\Gamma_k}_{\Gamma_k\cap S_l}\pi_l'$ is a subrepresentation of $(\Ind^{\Gxx}_{S_l}\pi_{\omega_l})|_{\Gamma_k}$. On the other hand, since the ideal $\ker\rho_l\subset C^*(\Gamma_k)$ is maximal, we have $\ker\pi_{\varphi_l}=\ker\rho_l$. Since $\pi'_l$ is a subrepresentation of $\pi_{\varphi_l}|_{\Gamma_k\cap S_l}$, by Corollary~\ref{cor:Frob2} it follows that $\rho_l\sim\pi_{\varphi_l}\prec\Ind^{\Gamma_k}_{\Gamma_k\cap S_l}\pi_l'$. Therefore $\rho_l\prec (\Ind^{\Gxx}_{S_l}\pi_{\omega_l})|_{\Gamma_k}$.

It follows that we can extend each state $\psi_l$ on $C^*(\Gamma_k)$ to a state on $C^*(\Gxx)$ that vanishes on the kernel of $\Ind^{\Gxx}_{S_l}\pi_{\omega_l}$. In view of~\eqref{eq:psi1} this shows that $\pi$ is weakly contained in the direct sum of the representations $\Ind^{\Gxx}_{S_l}\pi_{\omega_l}$ taken over the entire collection of pairs $(S_l,\pi_{\omega_l})$ that we get by taking all possible $F$ and $\delta$.

\smallskip

It remains to prove the claim. Thus, we assume that we are given a state $\varphi$ on $C^*_\rho(\Gamma_k)$ as in~(2), and we need to show that there are a subgroup $S\subset\Gxx$ and a state $\omega$ on $C^*(S)$ such that $(S,\omega)$ has the property formulated in (3) and $\varphi=\omega$ on $C^*(\Gamma_k\cap S)$. Let $y_i$, $g_i$ and $\nu_i$ be as in (2). By passing to a subnet we may assume that $\G^{r(g_i)}_{r(g_i)}\to S$ for some subgroup $S\subset\Gxx$. For every $g\in S$ and all~$i$ large enough, we have an element $h_{g,i}$ such that $W_g\cap\G^{r(g_i)}_{r(g_i)}=\{h_{g,i}\}$. By passing to a subnet we may assume that for every $g\in S$ the net $(\nu_i(h_{g,i}))_i$ converges to a number $\omega(g)$. The function~$\omega$ on~$S$ is easily seen to be positive definite, so it defines a state on~$C^*(S)$. Obviously, we have $\varphi=\omega$ on~$C^*(\Gamma_k\cap S)$.
\ep

When the groups $\Gamma_k$ are finite, condition (2) can be written in terms of representation theory of finite groups, as follows.

\begin{cor}\label{cor:main-lf}
Assume $\G$ is a Hausdorff locally compact \'etale groupoid and $x\in\Gu$ is a unit such that the group $\Gxx$ is locally finite. Let $(\Gamma_k)_{k\in K}$ be an increasing net of finite subgroups of $\Gxx$ with union $\Gxx$. For every $g\in\Gxx$, fix an open bisection $W_g$ containing $g$.
Assume we are given a unitary representation~$\pi$ of~$\Gxx$, a subset $Y\subset\Gu$ and, for every $y\in Y$, a unitary representation~$\pi_y$ of~$\G^y_y$. Then the following conditions are equivalent:
 \begin{enumerate}
  \item $\Ind^\G_{\Gxx}\pi\prec\bigoplus_{y\in Y}\Ind^\G_{\G^y_y}\pi_y$;
  \item for every $k\in K$ and every irreducible unitary subrepresentation $\rho$ of $\pi|_{\Gamma_k}$, there exist a net $(y_i)_{i\in I}$ in $Y$ and elements $g_i\in\G_{y_i}$ such that $r(g_i)\to x$ and the unitary representations~$\rho|_{\Gamma_{k,i}}$ and $\pi^{g_i}_{y_i}\circ\Phi_i$ are not disjoint, where $\Gamma_{k,i}\subset\Gamma_k$ and $\Phi_i\colon\Gamma_{k,i}\to\G^{r(g_i)}_{r(g_i)}$ are defined by
 $$
 \Gamma_{k,i}:=\{g\in\Gamma_k: W_g\cap \G^{r(g_i)}_{r(g_i)}\ne\emptyset\}\quad\text{and}\quad W_g\cap \G^{r(g_i)}_{r(g_i)}=\{\Phi_i(g)\}\quad(g\in\Gamma_{k,i}).
 $$
 \end{enumerate}
\end{cor}

Note that condition (2) makes sense, since it can be easily seen that the convergence $r(g_i)\to x$ and finiteness of $\Gamma_k$ imply that for all $i$ large enough the set $\Gamma_{k,i}$ is a subgroup of $\Gamma_k$ and $\Phi_i\colon\Gamma_{k,i}\to\G^{r(g_i)}_{r(g_i)}$ is an injective homomorphism.

\bp[Proof of Corollary~\ref{cor:main-lf}]
We want to show that condition~(2) in the corollary is equivalent to condition (2) in Theorem~\ref{thm:main2}. Assume the last condition is satisfied and $\varphi$, $y_i$, $g_i$ and $\nu_i$ are as formulated there. By passing to a subnet we may assume that $\Gamma_{k,i}=\tilde\Gamma_k$ for all $i$ for some subgroup $\tilde\Gamma_k\subset\Gamma_k$. For every $a\in C^*(\tilde\Gamma_k)$ we then have $\nu_i(\Phi_i(a))\to\varphi(a)$, where we extended~$\Phi_i$ to the group algebra of~$\tilde\Gamma_k$ by linearity. In particular, if $z\in C^*(\tilde\Gamma_k)$ is a minimal central projection such that $\varphi(z)\ne0$, then $\nu_i(\Phi_i(z))\ne0$ for all $i$ large enough. Therefore if $\tilde\rho$ is the irreducible representation of $\tilde\Gamma_k$ corresponding to $z$ (so that the kernel of $\tilde{\rho}$ is $(1-z)C^*(\tilde\Gamma_k)$), then~$\tilde\rho$ is a subrepresentation of~$\rho|_{\tilde\Gamma_k}$ and $\pi^{g_i}_{y_i}\circ\Phi_i$ for all $i$ large enough. Hence condition (2) of the corollary is satisfied.

Conversely, assume condition (2) of the corollary is satisfied. Let $\rho\prec \pi|_{\Gamma_{k}}$ be an irreducible unitary representation. By passing to a subnet we may assume that $\Gamma_{k,i}=\tilde\Gamma_k$ for some subgroup $\tilde\Gamma_k\subset\Gamma_k$, $\Phi_i$ is an injective homomorphism and there is an irreducible representation~$\tilde\rho$ of~$\tilde\Gamma_k$ such that it is a subrepresentation of $\rho|_{\tilde\Gamma_k}$ and $\pi^{g_i}_{y_i}\circ\Phi_i$ for all $i$. Take any state~$\psi$ on~$C^*_{\tilde\rho}(\tilde\Gamma_k)$. Since $C^*_{\tilde\rho}(\tilde\Gamma_k)$ is a quotient of $C^*_{\rho|_{\tilde\Gamma_k}}(\tilde\Gamma_k)\subset C^*_\rho(\Gamma_k)$, we can extend it to a state~$\varphi$ on~$C^*_\rho(\Gamma_k)$. Similarly, we can extend the state $\Phi_i(a)\mapsto\psi(a)$ on~$C^*(\Phi_i(\tilde\Gamma_k))\cong C^*(\tilde\Gamma_k)$ to a state~$\nu_i$ on~$C^*_{\pi^{g_i}_{y_i}}(\G^{r(g_i)}_{r(g_i)})$. Then $\nu_i(\Phi_i(g))=\varphi(g)$ for $g\in\tilde\Gamma_k$, while for $g\in\Gamma_k\setminus\tilde\Gamma_k$ we have $W_g\cap\G^{r(g_i)}_{r(g_i)}=\emptyset$. Thus condition~(2) in Theorem~\ref{thm:main2} is satisfied.
\ep

\bigskip

\section{The primitive spectrum of groupoid \texorpdfstring{C$^*$}{C*}-algebras}\label{sec:primitive}

\subsection{Main result}\label{ssec:main}
Under suitable conditions Theorem~\ref{thm:main2} leads to a complete description of $\Prim C^*(\G)$. Let us formalize the class of isotropy groups that we can handle.

\begin{defn}\label{def:M}
Denote by $\M_0$ the class of discrete groups $\Gamma$ such that every primitive ideal in $C^*(\Gamma)$ is an intersection of maximal ideals, and by $\M$ the class of groups $\Gamma$ such that for every finite subset $F\subset \Gamma$ there is a subgroup $S\subset \Gamma$ such that $S$ contains $F$ and belongs to the class~$\M_0$.
\end{defn}

We will discuss in a moment what kind of groups are contained in $\M$, but first let us formulate and prove our main result.

\begin{thm}\label{thm:main-prim}
Assume $\G$ is an amenable second countable Hausdorff locally compact \'etale groupoid such that each isotropy group $\Gxx$ belongs to the class $\M$. Then the induction map $\Ind\colon\Stab(\G)^\prim\to\Prim C^*(\G)$ is surjective and the topology on $\Prim C^*(\G)$ is described as follows. Suppose $\Omega$ is a $\G$-invariant subset of $\Stab(\G)^\prim$ and $(x,J)\in\Stab(\G)^\prim$. Then the following conditions are equivalent:
\begin{enumerate}
  \item $\Ind(x,J)$ lies in the closure of $\Ind\Omega$ in $\Prim C^*(\G)$;
  \item $\pi_J\prec\bigoplus_{(S,I)}\Ind^{\Gxx}_S\pi_I$, where the summation is taken over all points $(S,I)\in\Sub(\G)^\prim$ in the closure of $\Omega$ such that $S\subset\Gxx$.
\end{enumerate}
\end{thm}

We remind that by $\pi_J$ we mean any fixed irreducible representation with kernel $J$. Therefore condition (2) can be written as $\bigcap_{(S,I)}\Ind^{\Gxx}_SI\subset J$, as was done in the introduction.

\bp
The map $\Ind\colon\Stab(\G)^\prim\to\Prim C^*(\G)$ is surjective by~\cite{IW}. In order to describe the topology we apply Theorem~\ref{thm:main2} to the image $Y$ of $\Omega$ in $\Gu$ and the representations $\pi_y:=\bigoplus_{I:(y,I)\in\Omega}\pi_I$. As in that theorem, fix for every $g\in \Gxx$ an open bisection $W_{g}$ containing $g$. Since the set $\Omega$ is $\G$-invariant and $\G$ is second countable, we then get that (1) is equivalent to
\begin{enumerate}
  \item[(2$'$)] $\pi_J\prec\bigoplus_{(S,\omega)}\Ind^{\Gxx}_S\pi_\omega$, where the summation is taken over all pairs $(S,\omega)$ such that $S\subset\Gxx$ is a subgroup, $\omega$ is a state on $C^*(S)$ and the following property is satisfied: there exist a sequence $(y_n)_n$ in~$Y$ and states~$\omega_n$ on~$C^*_{\pi_{y_n}}(\G^{y_n}_{y_n})$ such that $y_n\to x$, $\G^{y_n}_{y_n}\to S$ in $\Sub(\G)$ and, for every $g\in S$, we have $\omega_n(g_n)\to\omega(g)$, where $g_n$ is the unique element of $W_g\cap \G^{y_n}_{y_n}$.
\end{enumerate}

Now, if $S\subset\Gxx$ and $(S,I)$ lies in the closure of $\Omega$, then by Lemma~\ref{lem:Fell-convergence}, for any $\omega\in \SSS_{\pi_I}(C^*(S))$, the pair $(S,\omega)$ has the property formulated in (2$'$) and $\pi_I\cong\pi_\omega$. It follows that (2) implies~(2$'$). Conversely, assume (2$'$) is satisfied and take $(S,\omega)$ as in (2$'$), with $y_n\in Y$ as formulated there. Then by Lemma~\ref{lem:Fell-convergence2} and density of $\{K:(y_n,K)\in\Omega\}$ in $\hull(\ker\pi_{y_n})$, for every $I\in\hull(\ker\pi_\omega)$ we can find $I_n$ such that $(y_n,I_n)\in\Omega$ and $(\G^{y_n}_{y_n},I_n)\to(S,I)$. Since $\pi_\omega\sim\bigoplus_{I\in\hull(\ker\pi_\omega)}\pi_I$, it follows that (2) is satisfied.
\ep

Let us now discuss the classes $\M_0$ and $\M$. It is known that for any finitely generated group~$\Gamma$ the following conditions are equivalent:

\begin{enumerate}
  \item $\Gamma$ has polynomial growth;
  \item $\Gamma$ is virtually nilpotent;
  \item $\Gamma$ is amenable and every primitive ideal of $C^*(\Gamma)$ is maximal.
\end{enumerate}
The implication (2)$\Rightarrow$(1) is a consequence of a result of Wolf~\cite{MR0248688}*{Theorem~3.2}, while the converse is the famous result of Gromov~\cite{MR0623534}. The implication (2)$\Rightarrow$(3) is a result of Poguntke~\cite{MR0667315}. Finally, if (3) holds, then by an observation of Moore and Rosenberg~\cite{MR0419675}*{pp.~219--220}, the group $\Gamma$ is FC-hypercentral, hence, being also finitely generated, it is virtually nilpotent by~\cite{MR0084498}*{Theorem~2}.  Recall that FC-hypercentrality means that $\Gamma$ has no nontrivial ICC quotients, equivalently, $\Gamma$ coincides with its FC-hypercenter, which is defined as follows. The ascending FC-series $(\Gamma_\alpha)_\alpha$, where $\alpha$ are ordinals, is defined by $\Gamma_0:=\{e\}$, $\Gamma_{\alpha+1}$ consists of the elements of $\Gamma$ that have finite conjugacy classes in $\Gamma/\Gamma_\alpha$, and $\Gamma_{\alpha}:=\bigcup_{\beta<\alpha}\Gamma_\beta$ if $\alpha$ is a limit ordinal. The series $(\Gamma_\alpha)_\alpha$ eventually stabilizes, and $\bigcup_\alpha\Gamma_\alpha$ is called the \emph{FC-hypercenter} of $\Gamma$.

Therefore $\M_0$ includes all finitely generated groups of polynomial growth and hence $\M$ includes all groups of local polynomial growth, that is, such that every finitely generated subgroup has polynomial growth. Thus, Theorem~\ref{thm:main-prim} subsumes Theorem~\ref{thm:A} from the introduction.

By the already mentioned result of Echterhoff~\cite{MR1066810}, if a group $\Gamma$ is FC-hypercentral, then every primitive ideal of $C^*(\Gamma)$ is maximal. Therefore $\M_0$ includes also all FC-hypercentral groups. This does not expand the class $\M$, however, since every finitely generated subgroup of an FC-hypercentral group is itself FC-hypercentral and hence has polynomial growth, as we discussed above.

Nevertheless, as the following two examples show, the classes $\M_0$ and $\M$ are strictly larger than what we can get from groups of polynomial growth and include some groups of exponential growth.

\begin{example}\label{ex:Furstenberg}
Let $S\subset\N$ be a multiplicative submonoid satisfying the following two properties. The first is that $S$ is \emph{nonlacunary}, meaning that $S$ does not consist of powers of one number. The second is that there is a prime $p$ such that none of the elements of $S$ is divisible by~$p$. Put $\Gamma:=S^{-1}S\subset\mathbb Q^\times_+$. Consider the localization $S^{-1}\Z$ of $\Z$ by $S$, which we view as a group under addition. Thus, $S^{-1}\Z$ consists of rational numbers with denominators in $S$. The group $\Gamma$ acts on $S^{-1}\Z$ by multiplication, and we define $G_S:=\Gamma\ltimes S^{-1}\Z$.

As was observed by Cuntz, when $S$ is generated by two multiplicatively independent integers, the structure of $C^*(G_S)$ is relevant to Furstenberg's $\times2\times3$ problem. A detailed study of $C^*(G_S)$ in this case has been recently undertaken by Bruce and Scarparo~\cite{BS}. We are interested only in the primitive spectrum of $C^*(G_S)$. Let us briefly explain how the analysis in \citelist{\cite{Scar}\cite{BS}} extends to general $G_S$ as above.

Denote by $X$ the dual group of $S^{-1}\Z$. Then $C^*(G_S)$ can be identified with $\Gamma\ltimes C(X)$. Since~$S^{-1}\Z$ is an inductive limit of the cyclic groups $s^{-1}\Z$, the compact group $X$ is a projective limit of copies of $\T$ indexed by elements  of $S$. Concretely, $X$ can be identified with the closed subgroup of $\T^S$ consisting of points $z=(z(s))_{s\in S}$ satisfying $z(s)=z(ts)^t$ for all $s,t\in S$. The action of $t\in S\subset \Gamma$ on $X$ is given by $(tz)(s)=z(ts)$ for $s\in S$, with the inverse given by $(t^{-1}z)(s)=z(s)^t$.

Consider the maps $\varphi_s\colon\T\to\T$, $\varphi_s(w):=w^s$, $s\in S$. Then it can be checked similarly to~\cite{BS}*{Proposition~3.1} that there is a bijective correspondence between the closed $\Gamma$-invariant subsets $Y\subset X$ and the closed subsets $F\subset\T$ such that $\varphi_s(F)=F$ for all $s\in S$. Namely, given $Y\subset X$ we define $F_Y:=\{z(1)\mid z\in Y\}$. The inverse map sends $F\subset\T$ into $Y_F:=X\cap F^S\subset\T^S$. Moreover, see again~\cite{BS}*{Proposition~3.1}, a closed invariant subset $Y\subset X$ is finite if and only if $F_Y\subset \T$ is finite, and then $|Y|=|F_Y|$. By a result of Furstenberg~\cite{MR0213508}*{Proposition~IV.2}, see also~\cite{MR1195714}*{Theorem~1.2}, it follows that every proper closed invariant subset of $X$ is finite. In particular, every $\Gamma$-orbit in~$X$ is either finite or dense. Note also that the stabilizer of every point with a dense orbit is trivial, since the action $\Gamma\curvearrowright X$ is faithful.

Every character $z\in X$ such that $z(1)$ is not a root of unity has infinite, hence dense, orbit. On the other hand, for $n\ge1$, consider the set~$F_n$ of roots of unity of order $p^n$. Then $\varphi_s(F_n)=F_n$ for all $s\in S$, since the elements of~$S$ are not divisible by~$p$ by assumption. Consider the corresponding finite closed invariant set $Y_n:=Y_{F_n}$. Then the $\Gamma$-orbit of every element of $Y_n$ is finite, and the union of these orbits for all $n$ is dense in $X$, since $\cup_nF_n$ is dense in $\T$.

Now, every primitive ideal of $C^*(G_S)$ has the form $\ker\Ind^{\Gamma\ltimes X}_{\Gamma_x}\chi$ for a character~$\chi$ of the stabilizer $\Gamma_x$ of $x\in X$. If $x$ has dense orbit, then this kernel is $0$, while if the orbit of~$x$ is finite, then the irreducible representation $\Ind^{\Gamma\ltimes X}_{\Gamma_x}\chi$ is finite dimensional, hence its kernel is a maximal ideal. The intersection $I$ of all these maximal ideals is $0$. Indeed, by the ideal intersection property (see Section~\ref{ssec:IIP}) it suffices to check that $I\cap C(X)=0$, and this is true, since the points with finite orbits are dense in~$X$. See~\cite{BS}*{Theorem~2.4} for a complete description of the Jacobson topology on $\Prim C^*(G_S)$.

Therefore the group $G_S$ belongs to the class $\M_0$. We note in passing that since every ideal of $C^*(G_S)$ is an intersection of maximal ideals of finite codimension, the C$^*$-algebra~$C^*(G_S)$ is strongly RFD in the sense of~\cite{MR4052213}, that is, all of its quotients are RFD.

When the monoid $S$ is finitely generated, the group $G_S$ is finitely generated. It is solvable, but not virtually nilpotent. It follows that $G_S$ has exponential growth by the Milnor--Wolf theorem, which in this case is also easy to see directly. Therefore the class $\M$ is strictly larger than the class of groups of local polynomial growth.
\end{example}

\begin{example}
Since every finite subset of the $ax+b$ group $\Q^\times_+\ltimes\Q$ is contained in a subgroup~$G_S$ from the previous example, we see that $\Q^\times_+\ltimes\Q$ belongs to the class $\M$. The analysis in that example remains essentially the same if we replace $\Gamma=\langle S\rangle\subset\Q^\times_+$ by the group $\langle S, -1\rangle\subset\Q^\times$. It follows then that the $ax+b$ group $\Q^\times\ltimes\Q$ also belongs to the class $\M$.
\end{example}

\begin{example}
Consider the Baumslag--Solitar group $G:=\operatorname{BS}(1,m)$ for some $m\ge2$. For $m=2$, the primitive spectrum of $C^*(G)$ was studied already by Guichardet in~\cite{MR0147925}*{\S3}. We have $G\cong\Z\ltimes\Z[\frac{1}{m}]$, where $\Z$ acts by multiplication by $m$. Therefore $G$ is defined in the same way as in Example~\ref{ex:Furstenberg}, but now for the lacunary monoid $S=\{1,m,m^2,\dots\}$. Similarly to that example, the points with finite orbits with respect to the action of $\Z$ on the dual $X$ of $\Z[\frac{1}{m}]$ are dense, so $C^*(G)$ has maximal ideals of finite codimension whose intersection is $0$. The zero ideal is primitive, which can be shown either by exhibiting a dense orbit or by checking that $G$ is an ICC group. But this time there are primitive ideals that are not intersections of maximal ones.

To give an example, consider the character $\chi$ on $\Z[\frac{1}{m}]$ defined by
$$
\chi(a):=e^{2\pi ia}.
$$
Then it is not difficult to see that the closure of the orbit of $\chi$ in $X$ consists of the orbit itself and the trivial character $\eps$. Consider the primitive ideal $I$ of $C^*(G)=\Z\ltimes C(X)$ defined by~$\chi$, that is, the ideal obtained by inducing the trivial representation of the trivial stabilizer of~$\chi\in X$. Besides~$I$, the primitive ideals that contain $I$ are the ones that are obtained by inducing characters of the stabilizer $\Z$ of $\eps\in X$. The latter ideals contain the functions in $C(X)$ that vanish at~$\eps$, so their intersection is strictly larger than $I$, since $I\cap C(X)$ consists of the functions that vanish on the orbit of $\chi$. Therefore $I$ is not an intersection of maximal ideals.

Since $G$ is finitely generated, we thus see that $G$ does not belong to the class $\M$. Since $G$ embeds into groups from Example~\ref{ex:Furstenberg}, this shows that neither $\M$ nor $\M_0$ is closed under taking subgroups.\ee
\end{example}

Theorem~\ref{thm:main-prim} covers in particular all transformation groupoids $\G=\Gamma\ltimes X$ defined by actions $\Gamma\curvearrowright X$ of countable groups of local polynomial growth on second countable Hausdorff locally compact spaces. It also covers the second countable transformation groupoids defined by amenable actions of groups $\Gamma$ that are relatively hyperbolic with respect to a family~$\mathcal H$ of subgroups of local polynomial growth (see, e.g.,~\cite{MR2922380} for the relevant definitions). Indeed, the stabilizers of such actions are amenable, while by the strong Tits alternative~\cite{MR1313451}*{Theorems~2T} every amenable subgroup of $\Gamma$ is either virtually cyclic or conjugate to a subgroup of some $H\in\mathcal H$, hence it is of local polynomial growth. For example, $\Gamma$ can be the free product of two countable groups of local polynomial growth or the free product of a hyperbolic group with a countable group of local polynomial growth. In this case the claim about amenable subgroups follows already from the Kurosh subgroup theorem and the fact that any amenable subgroup of a hyperbolic group is virtually cyclic~\cite{GlH}*{Th\'eor\`eme~8.37}. Since the Kurosh subgroup theorem is true for free products of arbitrary collections of groups, we can in fact take as $\Gamma$ the free products of countable collections of countable groups that are either of local polynomial growth or hyperbolic.

\subsection{Groupoids with FC-hypercentral isotropy} \label{ssec:FC}

So far we have used only the equivalence of (1) and (3) in Theorem~\ref{thm:main2}. Under stronger assumptions we can apply the equivalence of (1) and (4) there to get the following result, which is Theorem~\ref{thm:B} from the introduction.

\begin{thm}\label{thm:main-prim-fc}
Assume $\G$ is an amenable second countable Hausdorff locally compact \'etale groupoid such that each isotropy group $\Gxx$ is FC-hypercentral. Then the map $\Ind\colon\Stab(\G)^\prim\to\Prim C^*(\G)$ is surjective and the topology on $\Prim C^*(\G)$ is described as follows. Suppose $\Omega$ is a $\G$-invariant subset of $\Stab(\G)^\prim$ and $(x,J)\in\Stab(\G)^\prim$. Then the following conditions are equivalent:
\begin{enumerate}
  \item $\Ind(x,J)$ lies in the closure of $\Ind\Omega$ in $\Prim C^*(\G)$;
  \item the closure of $\Omega$ in $\Sub(\G)^\prim$ contains a point $(S,I)$ such that $S\subset\Gxx$ and $\pi_I\prec\pi_J|_S$.
\end{enumerate}
\end{thm}

\bp
The map $\Ind\colon\Stab(\G)^\prim\to\Prim C^*(\G)$ is surjective by~\cite{IW}. In order to describe the topology we use that the ideal $J\subset C^*(\Gxx)$ is maximal by~\cite{MR1066810} and apply Theorem~\ref{thm:main2} to the image~$Y$ of~$\Omega$ in~$\Gu$ and the representations $\pi_y:=\bigoplus_{I:(y,I)\in\Omega}\pi_I$. Since the set $\Omega$ is $\G$-invariant and $\G$ is second countable, we then get that (1) is equivalent to
\begin{enumerate}
  \item[(2$'$)] there are a subgroup $S\subset\Gxx$ and a state $\omega$ on $C^*_{\pi_J|_S}(S)$ such that the following property holds: there exist a sequence $(y_n)_n$ in~$Y$ and states $\omega_n$ on $C^*_{\pi_{y_n}}(\G^{y_n}_{y_n})$ such that $y_n\to x$, $\G^{y_n}_{y_n}\to S$ in $\Sub(\G)$ and, for every $g\in S$, we have $\omega_n(g_n)\to\omega(g)$, where $g_n$ is the unique element of $W_g\cap \G^{y_n}_{y_n}$.
\end{enumerate}

Now, if $(S,I)$ is as in (2), then by Lemma~\ref{lem:Fell-convergence}, for any $\omega\in \SSS_{\pi_I}(C^*(S))$, the pair $(S,\omega)$ has the property formulated in (2$'$), so (2) implies~(2$'$). Conversely, assume (2$'$) is satisfied. Take any $I\in\hull(\ker\pi_\omega)$. Then by Lemma~\ref{lem:Fell-convergence2} and density of $\{K:(y_n,K)\in\Omega\}$ in $\hull(\ker\pi_{y_n})$ we can find $I_n$ such that $(y_n,I_n)\in\Omega$ and $(\G^{y_n}_{y_n},I_n)\to(S,I)$. As $\pi_I\prec\pi_\omega\prec\pi_J|_S$, we conclude that~(2) is satisfied.
\ep

When the isotropy groups are abelian, this result is equivalent to~\cite{CN3}*{Corollary~2.8}. We will return to this in Section~\ref{ssec:questions}.

\begin{remark}
The class of FC-hypercentral groups is the largest class of isotropy groups for which the Jacobson topology can be described as in Theorem~\ref{thm:main-prim-fc}. More precisely, for any class of groups that contains an amenable countable group $\Gamma$ that is not FC-hypercentral, there are groupoids with isotropy groups in this class such that the above description of the topology does not work. The group $\Gamma$ itself is an example of such a groupoid. Indeed, as we discussed in Section~\ref{ssec:main}, by an observation of Moore and Rosenberg~\cite{MR0419675}, there exists a nonmaximal primitive ideal~$J$ in~$C^*(\Gamma)$. If $I\in\Prim C^*(\Gamma)$ is such that $J\subsetneq I$, then condition (2) in Theorem~\ref{thm:main-prim-fc} is satisfied for $\Omega:=\{I\}$, but $J=\Ind J\notin\overline{\Ind\Omega}=\bar\Omega$.\ee
\end{remark}

Similarly to our discussion at the end of Section~\ref{ssec:main}, Theorem~\ref{thm:main-prim-fc} covers in particular all second countable transformation groupoids $\G=\Gamma\ltimes X$ defined by actions $\Gamma\curvearrowright X$ of FC-hypercentral groups, as well as the second countable transformation groupoids defined by amenable actions of groups that are relatively hyperbolic with respect to a family of FC-hypercentral subgroups. The latter class includes all lattices in connected semisimple Lie groups of real rank one: such lattices are known to be relatively hyperbolic with respect to a family of virtually nilpotent subgroups (see, for example, \cite{HH}*{Theorem~10.1} and \cite{MR2154349}*{Theorem~0.1}).

\smallskip

Theorem~\ref{thm:main-prim-fc} also covers all second countable Hausdorff locally compact \'etale grou\-poids~$\G$ such that $C^*(\G)$ is of type I. This follows from the classical result of Thoma~\cite{MR0160118} and the characterization of such groupoids obtained in~\citelist{\cite{MR0335681}\cite{MR2359724}\cite{MR3835454}}. In fact, the latter papers deal with not necessarily \'etale groupoids, while in the \'etale case the characterization takes the following simple definitive form.

\begin{prop}\label{prop:type-I}
For any second countable Hausdorff locally compact \'etale groupoid $\G$, the C$^*$-algebra $C^*(\G)$ is of type I if and only if every $\G$-orbit $[x]$ in $\Gu$ is discrete and every isotropy group $\Gxx$ is virtually abelian. Moreover, any groupoid $\G$ with these properties is amenable.
\end{prop}

\bp
Assume first that $C^*(\G)$ is a C$^*$-algebra of type I. For $x\in\Gu$ and any unitary representation $\pi$ of $\Gxx$, the von Neumann algebra $\pi(C^*(\Gxx))''$ is a reduction of $(\Ind^\G_{\Gxx}\pi)(C^*(\G))''$, hence it is of type I. By~\cite{MR0160118} it follows that $\Gxx$ is virtually abelian. In particular, the isotropy groups are amenable, and then by~\cite{MR2359724}*{Theorem~1.4} it follows that the orbit space $\G\backslash\Gu$ must be~$T_0$ as a consequence of continuity of the embedding $\G\backslash\Gu\hookrightarrow \Prim C^*(\G)$, $[x]\mapsto\ker\Ind^\G_{\Gxx}\eps_x$, where~$\eps_x$ denotes the trivial representation of~$\Gxx$, cf.~Lemmas~\ref{lem:Ind-continuity} and~\ref{lem:greenleaf}. By~\cite{MR1081649}*{Theorem~2.1} this means that for every $x\in\Gu$ we get a homeomorphism $\G_x/\Gxx\to[x]$, $g\Gxx\mapsto r(g)$, so $[x]$ is discrete.

Conversely, assume every orbit in $\Gu$ is discrete and every isotropy group is virtually abelian. Then $C^*(\G)$ is of type I by~\cite{MR2359724}*{Theorem~1.4}, but our \'etale case allows for the following shorter argument. Consider an irreducible representation $\pi\colon C^*(\G)\to B(H)$. The restriction of~$\pi$ to~$C_0(\Gu)$ is described by an ergodic quasi-invariant probability measure~$\mu$ on~$\Gu$ and a multiplicity $m\in\N\cup\{+\infty\}$. Since the $\G$-orbits in $\Gu$ are locally closed, by~\cite{MR1081649}*{Theorem~2.1} the measure~$\mu$ must be concentrated on one orbit $[x]$. Then $\pi$ gives rise to a unitary representation~$\rho$ of~$\Gxx$ on~$H_x:=\pi(f)H$ and $\pi$ is unitarily equivalent to $\Ind^\G_{\Gxx}\rho$, where $f\in C_0(\Gu)$ is any function such that $f(x)=1$ and $f(y)=0$ for $y\in[x]\setminus\{x\}$. The representation $\rho$ must be irreducible, and since $\Gxx$ is virtually abelian, it follows that $H_x$ is finite dimensional. Thus, $\pi(C^*(\G))$ contains the nonzero finite rank projection~$\pi(f)$, proving that $C^*(\G)$ is of type~I.

Finally, if $C^*(\G)$ is of type I, then $C^*(\G)$ is nuclear, hence the groupoid $\G$ is amenable, see~\cite{MR2391387}*{Theorem~5.6.18} or~\cite{BM}*{Theorem~4.11}.
\ep

Note that surjectivity of $\Ind\colon\Stab(\G)^\prim\to\Prim C^*(\G)$ in the type I case is basically shown in the above proof and therefore does not depend on~\cite{IW}.

\subsection{Transformation groupoids with locally finite stabilizers}\label{ssec:main-lf}
In Corollary~\ref{cor:main-lf} we saw that when the isotropy groups are locally finite, a criterion for weak containment can be given in terms of representation theory of finite groups. For transformation groupoids this criterion takes a particularly nice form, which leads to the following result.

\begin{thm}\label{thm:main-prim-lf}
Assume $\Gamma\curvearrowright X$ is an amenable action of a countable group on a second countable Hausdorff locally compact space $X$ such that each stabilizer $\Gamma_x$ is locally finite. Then the map $\Ind\colon\Stab(\Gamma\ltimes X)^\prim\to\Prim (\Gamma\ltimes C_0(X))$ is surjective and the topology on $\Prim(\Gamma\ltimes C_0(X))$ is described as follows. Suppose $\Omega$ is a $\Gamma$-invariant subset of $\Stab(\Gamma\ltimes X)^\prim$ and $(x,J)\in\Stab(\Gamma\ltimes X)^\prim$. Let $(S_k)_{k=1}^{\infty}$ be an increasing sequence of finite subgroups of $\Gamma_x$ such that $\cup_kS_k=\Gamma_x$. Then the following conditions are equivalent:
\begin{enumerate}
  \item $\Ind(x,J)$ lies in the closure of $\Ind\Omega$ in $\Prim(\Gamma\ltimes C_0(X))$;
  \item for every $k\ge1$ and every irreducible unitary subrepresentation $\rho$ of $\pi_J|_{S_k}$, there exists a sequence of elements $(y_n,J_n)\in\Omega$ such that $y_n\to x$ and, for every $n$, the unitary representations~$\rho|_{S_k\cap\Gamma_{y_n}}$ and $\pi_{J_n}|_{S_k\cap\Gamma_{y_n}}$ are not disjoint.
\end{enumerate}
\end{thm}

\bp
Surjectivity of $\Ind$ follows again from~\cite{IW}, while the description of the topology follows from Corollary~\ref{cor:main-lf}.
\ep

\begin{remark}
Theorem~\ref{thm:main-prim-lf} implies that condition (2) there is independent of the choice of~$(S_k)_k$. It is instructive to see how this works. So assume $(S_k)_k$ is fixed, condition (2) is satisfied for this particular choice and take a finite subgroup $S\subset\Gxx$ and an irreducible representation $\eta\le\pi_J|_S$, that is, $\eta$ is equivalent to a subrepresentation of $\pi_J|_{S}$. Let $k$ be large enough so that $S\subset S_k$. By decomposing $\pi_J|_{S_k}$ into irreducible representations we can find an irreducible representation $\rho\le\pi_J|_{S_k}$ such that $\eta\le\rho|_S$. Without loss of generality we may assume that $\eta$ is literally a subrepresentation of $\rho|_S$. Let $(y_n,J_n)$ be the points given by~(2) for $\rho\le\pi_J|_{S_k}$. Take a nonzero intertwiner $T_n\in\operatorname{Mor}(\rho|_{S_k\cap\Gamma_{y_n}},\pi_{J_n}|_{S_k\cap\Gamma_{y_n}})$ and a nonzero vector $\xi\in H_\eta\subset H_\rho$. By irreducibility of $\rho$ we can find $s_n\in S_k$ such that $T_n\rho(s_n^{-1})\xi\ne0$. Then $T_n\rho(s_n^{-1})|_{H_\eta}$ is a nonzero intertwiner between $\eta|_{S\cap s_n\Gamma_{y_n}s_n^{-1}}$ and $\pi_{J_n}^{s_n}|_{S\cap s_n\Gamma_{y_n}s_n^{-1}}$. Since $y_n\to x$ and the group $S_k$ is finite and stabilizes~$x$, we also have $s_ny_n\to x$. Therefore we see that (2) holds for $(S,\eta)$ in place of~$(S_k,\rho)$, with $(y_n,J_n)$ replaced by $s_n(y_n,J_n)=(s_ny_n,(\Ad s_n)(J_n))$. \ee
\end{remark}

If every stabilizer $\Gamma_x$ is finite or, more generally, of type I, then $\Prim C^*(\Gamma_x)$ can be identified with the set $\widehat\Gamma_x$ of equivalence classes of irreducible unitary representations of $\Gamma_x$. Then instead of $\Stab(\Gamma\ltimes X)^\prim$ we can consider the space $\Stab(\Gamma\ltimes X)\dach$ consisting of the pairs $(x,[\rho])$, where~$[\rho]$ is the unitary equivalence class of an irreducible unitary representation $\rho$ of $\Gamma_x$. We note that in \cite{EE} this space is denoted simply by $\Stab(X)\dach$.

To formulate the next result, let us recall the following notion. Given a topological space~$Y$, its $T_0$-ization, or the Kolmogorov quotient, $Y^\sim$ is defined by identifying points of $Y$ that have identical closures. It is easy to show (see, e.g., \cite{CN3}*{Corollary~1.6}) that if a group $\Gamma$ acts by homeomorphisms on a topological space $Y$, then $(\Gamma\backslash Y)^\sim$ is the quotient of $Y$ such that two points $y_1$ and $y_2$ have identical images if and only if $\overline{\Gamma y_1}=\overline{\Gamma y_2}$. In other words, $(\Gamma\backslash Y)^\sim$ is the space of quasi-orbits of the action.

\begin{cor}\label{cor:EE}
Assume $\Gamma\curvearrowright X$ is an amenable action of a countable group on a second countable Hausdorff locally compact space $X$ such that, for every $x\in X$, the stabilizer $\Gamma_x$ is finite and $\Gamma_y\subset\Gamma_x$ for all $y$ in a neighbourhood $U_x$ of $x$. Then there is a topology on $\Stab(\Gamma\ltimes X)\dach$ with a base of open neighbourhoods of $(x,[\rho])$ consisting of the sets
$$
\{(y,[\eta])\in\Stab(\Gamma\ltimes X)\dach: y\in U,\ \eta\le\rho|_{\Gamma_y}\},
$$
where $U$ runs through a base of open neighbourhoods of $x$ contained in $U_x$, and the induction map defines a homeomorphism
$(\Gamma\backslash\Stab(\Gamma\ltimes X)\dach)^\sim\cong\Prim (\Gamma\ltimes C_0(X))$.
\end{cor}

%Here by $\eta\le\rho|_{\Gamma_y}$ we mean . Equivalently, since we are dealing with finite groups, $\eta\prec\rho|_{\Gamma_y}$.

\bp
It is not difficult to check that the topology on $\Stab(\Gamma\ltimes X)\dach$ is well-defined, see \cite{EE}*{Lemma~4.3 and Remark~4.4}. It is also straightforward to check that the action of $\Gamma$ on $\Stab(\Gamma\ltimes X)\dach$ is continuous. Identifying $\Stab(\Gamma\ltimes X)\dach$ with $\Stab(\Gamma\ltimes X)^\prim$, we then see from Theorem~\ref{thm:main-prim-lf} that the map $\Ind\colon\Stab(\Gamma\ltimes X)\dach\to \Prim (\Gamma\ltimes C_0(X))$ is surjective and, for any $\Gamma$-invariant set $\Omega\subset \Stab(\Gamma\ltimes X)\dach$ and any $(x,[\rho])\in\Stab(\Gamma\ltimes X)\dach$, we have $\Ind(x,[\rho])\in\overline{\Ind\Omega}$ if and only if $(x,[\rho])\in\bar\Omega$. This implies that the map $\Ind\colon\Stab(\Gamma\ltimes X)\dach\to \Prim (\Gamma\ltimes C_0(X))$ is continuous, and then by \cite{CN3}*{Lemma~1.4} this map induces a homeomorphism $(\Gamma\backslash\Stab(\Gamma\ltimes X)\dach)^\sim\cong\Prim (\Gamma\ltimes C_0(X))$.
\ep

If we are given a proper action $\Gamma\curvearrowright X$, then this corollary applies,
$$
(\Gamma\backslash\Stab(\Gamma\ltimes X)\dach)^\sim=\Gamma\backslash\Stab(\Gamma\ltimes X)\dach,
$$
and we recover \cite{EE}*{Theorem~4.6} for actions of discrete groups. To be more precise, formally we also need countability of $\Gamma$ and second countability of~$X$, but we used these assumptions only to talk about sequences instead of nets and to guarantee that $\Ind\colon\Stab(\Gamma\ltimes X)^\prim\to\Prim (\Gamma\ltimes C_0(X))$ is surjective, while it is known that for proper actions the induction map is always surjective, see~\cite{EE}.

\subsection{Further results and open questions}\label{ssec:questions}

The results of Sections~\ref{ssec:main}--\ref{ssec:main-lf} lead to several questions.

\begin{question}
What is the class of isotropy groups for which an analogue of Theorem~\ref{thm:main-prim} holds?
\end{question}

Of course, it is possible to conjecture that Theorem~\ref{thm:main-prim} remains true for all amenable second countable groupoids, but we do not feel that there is sufficient evidence for doing this. Note that the implication (2)$\Rightarrow$(1) always holds, so the question is when (1)$\Rightarrow$(2) is true.

\begin{question}
Is there an algebraic characterization of the classes $\M_0$ and $\M$? Is there at least a characterization of amenable finitely generated groups in the class $\M_0$?
\end{question}

As for Theorem~\ref{thm:main-prim-fc}, as we discussed in Section~\ref{ssec:FC}, it does not hold beyond the FC-hypercentral case. But we can replace condition (2) there by the following one, as in~\cite{MR0409720}*{Conjecture~2}: the closure of $\Omega$ in $\Sub(\G)^\prim$ contains a point $(S,I)$ such that $S\subset\Gxx$ and $\pi_J\prec\Ind^{\Gxx}_S\pi_I$. At this point we do not have any examples showing that this is not possible and our weaker condition~(2) in Theorem~\ref{thm:main-prim} is indeed needed. It seems plausible that such examples can be obtained using the groupoids $S^\fin_\infty\ltimes K^\N$ discussed below.

\begin{question}\label{que:topology}
Is there a topology on $\Stab(\G)^\prim$ such that the maps $\Ind\colon \Stab(\G)^\prim\to\Prim C^*(\G)$ and $\Stab(\G)^\prim\to\Gu$, $(x,J)\mapsto x$, and the action $\G\curvearrowright\Stab(\G)^\prim$ are continuous and, under the assumptions of Theorem~\ref{thm:main-prim}, \ref{thm:main-prim-fc} or \ref{thm:main-prim-lf}, we get a homeomorphism
$(\G\backslash\Stab(\G)^\prim)^\sim\cong\Prim C^*(\G)$?
\end{question}

When $\G$ has abelian isotropy, such a topology is defined in \citelist{\cite{MR4395600}\cite{CN3}}. Namely, for such~$\G$ we can view $\Stab(\G)^\prim=\Stab(\G)\dach$ as a set of pairs $(x,\chi)$, where $x\in\Gu$ and $\chi\colon\Gxx\to\T$ is a character. Fix a point $(x,\chi)$ and, for every $g\in\Gxx$, choose an open bisection~$W_g$ containing~$g$. Then a base of open neighbourhoods of $(x,\chi)$ in $\Stab(\G)\dach$ is given by the sets $\UU_x^\chi(U,\eps,(W_g)_{g\in F})$ defined as follows, where $\eps>0$, $F\subset\Gxx$ is a finite set and $U$ is an open neighbourhood of $x$ in $\Gu$ such that $U\subset\bigcap_{g\in F}r(W_g)$: the set $\UU_x^\chi(U,\eps,(W_g)_{g\in F})$ consists of the points $(y,\eta)$ such that $y\in U$ and for every $g\in F$ we have either  $W_g\cap\G^y_y=\emptyset$, or $W_g\cap\G^y_y=\{h\}$ for some $h$ and $|\chi(g)-\eta(h)|<\eps$.

We can also identify $\Sub(\G)^\prim$ with the set $\Sub(\G)\dach$ of pairs $(S,\chi)$, where $S\subset\IsoG$ is a subgroup and $\chi\colon S\to\T$ is a character. The topologies on $\Stab(\G)\dach$ and $\Sub(\G)\dach$ are related as follows.

\begin{lemma}
Assume $\G$ is a Hausdorff locally compact \'etale groupoid with abelian isotropy groups. Suppose $\Omega\subset\Stab(\G)\dach$ and $(x,\chi)\in\Stab(\G)\dach$. Then the following conditions are equivalent:
\begin{enumerate}
  \item the closure of $\Omega$ in $\Sub(\G)\dach$ contains a point $(S,\eta)$ such that $S\subset\Gxx$ and $\eta=\chi|_S$;
  \item the closure of $\Omega$ in $\Stab(\G)\dach$ contains $(x,\chi)$.
\end{enumerate}
\end{lemma}

\bp
The implication (1)$\Rightarrow$(2) follows from the description of convergence in $\Sub(\G)\dach$ given by Lemma~\ref{lem:subG-convergence} and Lemma~\ref{lem:Fell-convergence} in the second countable case and Remark~\ref{rem:Fell-continuity} in general. For the opposite implication, assume $\Omega\ni(x_i,\chi_i)\to(x,\chi)$ in $\Stab(\G)\dach$. For every $g\in\Gxx$, choose an open bisection~$W_g$ containing~$g$. Then it is not difficult to show that by passing to a subnet we may assume that $\G^{x_i}_{x_i}\to S$ for some subgroup $S\subset\Gxx$ and $\chi_i(g_i)\to\chi(g)$ for all $g\in S$, where $g_i$ is the unique element of $W_g\cap\G^{x_i}_{x_i}$ (which exists for $i$ large enough), see the proof of~\cite{CN3}*{Lemma~2.3}. But this means that $(\G^{x_i}_{x_i},\chi_i)\to(S,\chi|_S)$ in $\Sub(\G)\dach$, so (1) holds with $\eta:=\chi|_S$.
\ep

This lemma makes it clear that for groupoids with abelian isotropy groups Theorem~\ref{thm:main-prim-fc} is equivalent to~\cite{CN3}*{Corollary~2.8}.

\smallskip

For more general groupoids, in view of the definition of the sets $\UU_x^\chi(U,\eps,(W_g)_{g\in F})$ and condition (2) in Theorem~\ref{thm:main1}, a natural attempt to define a topology on $\Stab(\G)^\prim$ is to take as a neighbourhood of $(x,J)$ the points $(y,I)$ such that $y$ is close to $x$ and there exists a Hilbert--Schmidt operator $T\colon H_{\pi_J}\to H_{\pi_I}$ satisfying $\|T\|_2=1$ and
$$
\|T\pi_J(g)-\pi_I(h)T\|_2<\eps
$$
for all $g$ lying in a finite set $F\subset\Gxx$ such that $W_g\cap\G^y_y=\{h\}$ for some $h$. But we quickly get into trouble when we start taking intersections of such neighbourhoods for different points, since the composition of Hilbert--Schmidt operators of norm one can be zero. An attempt to use condition~(3) in Theorem~\ref{thm:main1} leads to a similar problem, since it only says that there is a state $\varphi$ on $C^*(\Gxx)$ with a particular property rather than describes a property of an intrinsically defined collection of states. We are therefore inclined to believe that Question~\ref{que:topology} has a negative answer and the desired topology on $\Stab(\G)^\prim$ exists only under some strong additional assumptions as, for example, in~\cite{CN3}*{Corollary~2.8}, Corollary~\ref{cor:EE} or Theorem~\ref{thm:Glimm} below. In particular, it doesn't seem like such a topology exists even for general transformation groupoids with finite stabilizers.

The assumption $\Gamma_y\subset\Gamma_x$ in Corollary~\ref{cor:EE} can be interpreted  as semicontinuity of the stabilizer groups. The situation becomes even better, and neither of our results is needed to describe $\Prim C^*(\G)$, when the isotropy groups $\Gxx$ depend continuously on $x$. In order to see this, let us start with the following observation.

\begin{lemma} \label{lem:iso-continuity}
For any Hausdorff locally compact \'etale groupoid $\G$, the following conditions are equivalent:
\begin{enumerate}
\item the subset $\Stab(\G)^\prim\subset\Sub(\G)^\prim$ is closed;
\item the subset $\Stab(\G)\subset\Sub(\G)$ is closed;
\item the map $\Gu\to\Sub(\G)$, $x\mapsto\Gxx$, is continuous.
\end{enumerate}
\end{lemma}

\bp
The implication (1)$\Rightarrow$(2) follows, for example, from continuity of the map $\Sub(\G)\to\Sub(\G)^\prim$, $S\mapsto (S,\ker\eps_S)$. The opposite implication follows from continuity of the map $ \Sub(\G)^\prim\to\Sub(\G)$, $(S,I)\mapsto S$.

Assume (2) holds. In order to prove that then (3) holds, by compactness of $\Cl(\G)$ it suffices to show that if $x_i\to x$ in $\Gu$ and $\G^{x_i}_{x_i}\to S$ in $\Cl(\G)$ for some closed subset $S\subset\G$, then $S=\Gxx$. But this is immediate, since $S$ must be a subgroup of $\Gxx$ and hence $S=\Gxx$ by assumption~(2). The opposite implication (3)$\Rightarrow$(2) follows from continuity of the map $\Sub(\G)\to\Gu$, $S\mapsto r(S)$.
\ep

Lemma~\ref{lem:subG-convergence} implies the following description of points of continuity of the map $x\mapsto\Gxx$.

\begin{lemma}\label{lem:gxx-continuity}
For any Hausdorff locally compact \'etale groupoid $\G$, the map $\Gu\to\Sub(\G)$, $y\mapsto\G^y_y$, is continuous at a point $x\in\Gu$ if and only if $\Gxx=\Iso_x$, where $\Iso\subset\G$ is the interior of the isotropy bundle $\IsoG\subset\G$.
\end{lemma}

\begin{cor}[{cf.~\cite{GR}*{Lemma~1.1}}]
If $\G$ is second countable or, more generally, it can be covered by countably many open bisections, then the set of points of discontinuity of the map $\Gu\to\Sub(\G)$, $x\mapsto\Gxx$, is meager in $\Gu$.
\end{cor}

\bp
By \cite{CN2}*{Lemma~5.2} the set of points $x$ such that $\Gxx\ne\Iso_x$ is meager.
\ep

The next theorem essentially goes back to Glimm~\cite{MR0146297}*{Theorem~2.1}, a more general result has been proved by Goehle~\cite{MR2966476}*{Theorem~3.5}. Since the arguments in our \'etale case are particularly transparent and short, we include a complete proof.

\begin{thm}\label{thm:Glimm}
Assume $\G$ is a Hausdorff locally compact \'etale groupoid such that the map $\Gu\to\Sub(\G)$, $x\mapsto\Gxx$, is continuous. Then $\Stab(\G)^\prim$ is a closed subset of $\Sub(\G)^\prim$. Let us equip it with the relative topology. Then the induction map defines a homeomorphism of $(\G\backslash\Stab(\G)^\prim)^\sim$ onto $\Ind(\Stab(\G)^\prim)\subset\Prim C^*(\G)$.
\end{thm}

\bp
By our assumptions and Lemma~\ref{lem:iso-continuity}, the subsets $\Stab(\G)\subset\Sub(\G)$ and $\Stab(\G)^\prim\subset\Sub(\G)^\prim$ are closed. Moreover, $\IsoG\subset\G$ is a clopen subset by Lemma~\ref{lem:gxx-continuity} and we have a topological isomorphism of the reduction $\Sigma(\G)|_{\Stab(\G)}$ of $\Sigma(\G)$ onto $\IsoG$ that maps $(S,g)$ into~$g$. This gives rise to a surjective homomorphism $C^*(\Sigma(\G))\to C^*(\IsoG)$, which can be viewed as a homomorphism $\rho\colon C^*(\Sigma(\G))\to C^*(\G)$.

The map $\Ind\colon \Stab(\G)^\prim\to\Prim C^*(\G)$ is continuous by Lemma~\ref{lem:Ind-continuity} and Lemma~\ref{lem:Fell-convergence} in the second countable case and Remark~\ref{rem:Fell-continuity} in general. Hence all we have to show to prove the theorem is that if $\Omega\subset\Stab(\G)^\prim$ is a $\G$-invariant set, $(x,J)\in\Stab(\G)^\prim$ and $\Ind(x,J)\in\overline{\Ind\Omega}$, then $(x,J)$ lies in the closure of $\Omega$ in $\Stab(\G)^\prim$, see~\cite{CN3}*{Lemma~1.4}. In other words, we need to show that if $\Ind^{\G}_{\Gxx}\pi_J\prec\bigoplus_{(y,I)\in\Omega}\Ind^{\G}_{\G^y_y}\pi_I$, then $\Ind^{\Sigma(\G)}_{\Gxx}\pi_J\prec\bigoplus_{(y,I)\in\Omega}\Ind^{\Sigma(\G)}_{\G^y_y}\pi_I$. But this becomes obvious once we compose the representations $\Ind^{\G}_{\Gxx}\pi_J$ and $\Ind^{\G}_{\G^y_y}\pi_I$ with $\rho$ and observe that, given a unit $z\in\Gu$ and a unitary representation $\pi$ of $\G^z_z$, we have a unitary equivalence
$$
(\Ind^{\G}_{\G^z_z}\pi)\circ\rho\cong\bigoplus_{g\in\G_z/\G^z_z}\Ind^{\Sigma(\G)}_{\G^{r(g)}_{r(g)}}\pi^g
$$
mapping $f\in \Ind^{\G}_{\G^z_z}H_\pi$ supported on $g\G^z_z$ into $f(g)\in H_{\pi^g}=H_\pi$. Indeed, the observation implies that $\Ind^{\Sigma(\G)}_{\Gxx}\pi_J$ is a subrepresentation of $(\Ind^{\G}_{\Gxx}\pi_J)\circ\rho$, while $\bigoplus_{(y,I)\in\Omega}(\Ind^{\G}_{\G^y_y}\pi_I)\circ\rho$ is quasi-equivalent to $\bigoplus_{(y,I)\in\Omega}\Ind^{\Sigma(\G)}_{\G^y_y}\pi_I$ by $\G$-invariance of $\Omega$. Hence $\Ind^{\Sigma(\G)}_{\Gxx}\pi_J\prec\bigoplus_{(y,I)\in\Omega}\Ind^{\Sigma(\G)}_{\G^y_y}\pi_I$.
\ep

\begin{cor}\label{cor:Glimm}
If $\G$ is in addition second countable and amenable, then we get a homeomorphism $(\G\backslash\Stab(\G)^\prim)^\sim\cong\Prim C^*(\G)$.
\end{cor}

\begin{example}
Consider the group $S^\fin_\infty$ of all permutations of $\N$ with finite support. Take a compact metric space $K$. Then $S^\fin_\infty$ acts by permutations on $X:=K^\N$. If $K$ is finite, then whenever $g\in S^\fin_\infty$ stabilizes a point $x\in X$, it also stabilizes all points close to $x$. Therefore in this case the isotropy groups depend continuously on $x$ and Corollary~\ref{cor:Glimm} describes the topology on $\Prim (S^\fin_\infty\ltimes C(X))$.

On the other hand, if $K$ has no isolated points, then the points with trivial stabilizers are dense in $X$. The points with nontrivial stabilizers are also dense in $X$, and each of them is then a point of discontinuity of the map $x\mapsto (S^\fin_\infty)_x$. In this case Corollary~\ref{cor:Glimm} no longer applies and the topology on $\Prim (S^\fin_\infty\ltimes C(X))$ is described by Theorem~\ref{thm:main-prim} or~\ref{thm:main-prim-lf}.

Moreover, in this case it is easy to see that if we do equip $\Stab(S^\fin_\infty\ltimes X)^\prim\subset\Sub(S^\fin_\infty\ltimes X)^\prim$ with the relative topology, then the continuous map $\Ind^\sim\colon(S^\fin_\infty\backslash\Stab(S^\fin_\infty\ltimes X)^\prim)^\sim\to\Prim (S^\fin_\infty\ltimes C(X))$ defined by the induction map $\Ind$ is not injective. Indeed, take any point $x\in X$ with dense orbit and trivial stabilizer, that is, a point such that its coordinates $x(n)\in K$ are all different and form a dense subset of $K$. Take also a point $y$ such that $y(1)=y(2)$ and the coordinates $y(n)$ for $n\ge2$ are all different and form a dense subset of $K$. The stabilizer of~$y$ is the symmetric group~$S_2\subset S^\fin_\infty$. The orbit of $y$ is dense in $X$, moreover, for every $z\in X$ we can find $g_n\in S^\fin_\infty$ such that $g_ny\to z$ and $g_nS_2g_n^{-1}\to\{e\}$ in $\Sub(S^\fin_\infty)$. As follows, for example, from the easy implication (2)$\Rightarrow$(1) of Theorem~\ref{thm:main-prim}, we then have that $\Ind(x,0)=\Ind(y,J)=0$ for either of the two primitive ideals~$J$ of $C^*(S_2)$. At the same time the orbit closures of~$(x,0)$ and~$(y,J)$ in $\Stab(S^\fin_\infty\ltimes X)^\prim$ are different, since there is no way of approximating the nontrivial stabilizer of $y$ by the trivial stabilizers of points on the orbit of~$x$. \ee
\end{example}

Finishing this subsection, we note that it should be clear that the proofs of Theorems~\ref{thm:main-prim} and~\ref{thm:main-prim-fc} work for general groupoids $\G$ and points $(x,J)$ such that assumptions of Theorem~\ref{thm:main2} are fulfilled. For future reference we formulate this explicitly as follows.

\begin{thm}
Assume $\G$ is a Hausdorff locally compact \'etale groupoid and $x\in\Gu$ is a unit such that the group $\Gxx$ is amenable. Assume $J\in\Prim C^*(\Gxx)$ is a primitive ideal with the following property: every finite subset of $\Gxx$ is contained in a subgroup $\Gamma\subset\Gxx$ such that $\Res^{\Gxx}_\Gamma J$ is an intersection of maximal ideals in $C^*(\Gamma)$. Then, for any $\G$-invariant subset $\Omega\subset\Stab(\G)^\prim$, the following conditions are equivalent:
\begin{enumerate}
\item $\Ind(x,J)$ lies in the closure of $\Ind\Omega$ in $\Prim C^*(\G)$;
\item $\bigcap_{(S,I)}\Ind^{\Gxx}_S I\subset J$, where the intersection is taken over all points $(S,I)\in\Sub(\G)^\prim$ in the closure of $\Omega$ such that $S\subset\Gxx$.
\end{enumerate}
If $J$ itself is a maximal ideal in $C^*(\Gxx)$, then conditions {\rm (1)--(2)} are equivalent~to
\begin{enumerate}
\item[(3)] the closure of $\Omega$ in $\Sub(\G)^\prim$ contains a point $(S,I)$ such that $S\subset\Gxx$ and $\Res^{\Gxx}_S J\subset I$.
\end{enumerate}
\end{thm}

The equivalence of (1) and (3) describes in particular the topology on the subset of $\Prim C^*(\G)$ obtained by inducing finite dimensional irreducible representations of amenable isotropy groups.

\subsection{The ideal intersection property}\label{ssec:IIP}
A Hausdorff locally compact \'etale groupoid $\G$ is said to have the \emph{ideal intersection property} if for every nonzero ideal $I\subset C^*_r(\G)$ we have $C_0(\Gu)\cap I\ne 0$. One says that it has the \emph{residual intersection property} if the reduced groupoid $\G_X$ has the ideal intersection property for every closed invariant subset $X\subset\Gu$. A  characterization of these properties in terms of the action of $\G$ on its isotropy bundle has been recently obtained in~\cite{KKLRU}. The case of amenable groupoids, which we are mainly interested in, has been known for quite some time. The goal of this subsection is to explain how the analysis of the ideal structure of groupoid C$^*$-algebras based on the ideal intersection property compares with the approach taken in this paper, along the way clarifying and correcting some statements in the literature.

\smallskip

Assume $\G$ is a Hausdorff locally compact \'etale groupoid. Since this is the only subsection where we explicitly use the reduced C$^*$-algebra of $\G$, let us briefly recall that $C^*_r(\G)$ is defined as the completion of $C_c(\G)$ with respect to the norm
$$
\|f\|_r:=\sup_{x\in\Gu}\|\rho_x(f)\|,
$$
where $\rho_x:=\Ind^\G_{\{x\}}\epsilon_x$ and $\epsilon_x$ is the trivial representation of $\{x\}$. Equivalently, $\rho_x$ can be defined as $\Ind^\G_{\Gxx}\lambda_x$, where $\lambda_x$ is the regular representation of~$\Gxx$. We will omit the subscript $r$ for the norm when it is clear that we work with $C^*_r(\G)$. If $\G$ is amenable, then $C_r^*(\G)=C^*(\G)$.

A groupoid~$\G$ is called \emph{inner exact}, if the sequence
$$
0\to C^*_r(\G_{X^c})\to C^*_r(\G)\to C^*_r(\G_X)\to0
$$
is exact  for every closed invariant subset $X\subset\Gu$. If $\G$ is amenable, then it is inner exact.

A  groupoid $\G$ is called \emph{effective} if $\IsoG^\circ=\Gu$, and it is called \emph{strongly effective} if $\G_X$ is effective for every closed invariant subset $X\subset\Gu$. If $\G$ is second countable or, more generally, it can be covered by countably many open bisections, then effectiveness is equivalent to topological principality, that is, to density of units $x\in\Gu$ with trivial isotropy groups in $\Gu$ (see, e.g., \cite{CN2}*{Lemma~5.2}).

\smallskip

Under suitable amenability assumptions the ideal intersection property is equivalent to effectiveness. One possible rigorous formulation is as follows.

\begin{prop}\label{prop:IIP}
Assume $\G$ is a Hausdorff locally compact \'etale groupoid. Consider the following conditions:
\begin{enumerate}
  \item $\G$ is effective;
  \item $\G$ has the ideal intersection property.
\end{enumerate}
Then $(1)\Rightarrow(2)$. If the set of units $x$ with amenable isotropy groups $\Gxx$ is dense in $\Gu$, then $(2)\Rightarrow(1)$.
\end{prop}

This has been proved with various degrees of generality in~\cite{Rbook}*{Proposition~II.4.6}, \cite{MR1088230}*{Theorem~4.1}, \cite{MR1258035}*{Theorems~1 and~2}, \cite{MR2745642}*{Theorem~4.4}, \cite{SSW}*{Theorems~10.2.7 and ~10.3.3}. Since formally~\cites{Rbook,MR2745642,SSW} rely at some point on second countability and \cites{MR1088230,MR1258035} deal only with transformation groupoids, we will give a short proof along the lines of~\cite{MR1258035} for the reader's convenience. A part of this proof will be needed later.

\bp[Proof of Proposition~\ref{prop:IIP}]
Assume $\G$ is effective, but there is a nonzero ideal $I\subset C^*_r(\G)$ such that $C_0(\Gu)\cap I=0$. Take any nonzero positive element $a\in I$. Since the canonical conditional expectation $E\colon C^*_r(\G)\to C_0(\Gu)$ extending the restriction map $C_c(\G)\to C_c(\Gu)$ is faithful, by rescaling $a$ we may assume that $\|E(a)\|=1$. Let $f\in C_c(\G)$ be such that
\begin{equation}\label{eq:approx}
\|a-f\|<\frac{1}{2}.
\end{equation}

We can find finitely many open bisections $W_1,\dots,W_n$ of $\G$ contained in $\G\setminus\Gu$ such that $\supp f\subset\Gu\cup W_1\cup\dots\cup W_n$. For every $i$, the set
$$
X_i:=\{x\in r(W_i): W_i\cap\Gxx\ne\emptyset\}
$$
is closed in the open set $r(W_i)$ and has empty interior by effectiveness of $\G$, hence it is nowhere dense in $\Gu$. It follows that the set $X:=\Gu\setminus\bigcup^n_{i=1}X_i$ is residual, in particular dense, in $\Gu$.

For $x\in X$, let $\ev_x\colon C_0(\Gu)\to\C$ be the state of evaluation at $x$. Consider the state on the C$^*$-subalgebra $C_0(\Gu)+I\subset C^*_r(\G)$ defined as the composition
$$
C_0(\Gu)+I\to(C_0(\Gu)+I)/I\cong C_0(\Gu)/(C_0(\Gu)\cap I)=C_0(\Gu)\xrightarrow{\ev_x}\C,
$$
and then extend it to a state $\varphi_x$ on $C^*_r(\G)$. By construction $\varphi_x$ extends $\ev_x$ and vanishes on~$I$. By~\cite{CN2}*{Proposition~1.11}, any state on $C^*_r(\G)$ extending~$\ev_x$ is the composition of a state on a completion~$C^*_e(\Gxx)$ of the group algebra of~$\Gxx$ with a completely positive contraction $C_r^*(\G)\to C^*_e(\Gxx)$ that extends the restriction map $C_c(\G)\to\C\Gxx$. Since $f|_{\Gxx}$ is supported on~$\{x\}$, it follows that $\varphi_x(f)=f(x)$. As $\varphi_x(a)=0$, inequality~\eqref{eq:approx} implies then that $|f(x)|<1/2$. Since $X$ is dense in $\Gu$, we conclude that $\|f|_{\Gu}\|_\infty\le1/2$, that is, $\|E(f)\|\le1/2$. On the other hand, since~$E$ is a contraction and $\|E(a)\|=1$, we get from~\eqref{eq:approx} that $\|E(f)\|>1/2$. This contradiction proves the implication (1)$\Rightarrow$(2).

\smallskip

For the opposite implication, assume the set $X$ of units with amenable isotropy groups is dense in $\Gu$ and $\G$ is not effective. Then there is an open bisection $W$ of $\G$ contained in $\IsoG\setminus\Gu$. Take any nonzero function $f\in C_c(W)$. Define a function $f_0\in C_c(r(W))$ by $f_0(r(g)):=f(g)$ for $g\in W$. A simple computation shows that $f-f_0$ is contained in the kernel of every representation~$\Ind^{\G}_{\Gxx}\eps_x$, where $\eps_x$ denotes the trivial representation of~$\Gxx$. The representations $\Ind^{\G}_{\Gxx}\eps_x$ for $x\in X$ factor through~$C^*_r(\G)$, since $\eps_x\prec\lambda_x$ by amenability of~$\Gxx$, and $C_0(\Gu)\cap\bigcap_{x\in X}\ker \Ind^{\G}_{\Gxx}\eps_x=0$ by density of $X$. It follows that the nonzero ideal $I\subset C^*_r(\G)$ generated by~$f-f_0$ has the property $C_0(\Gu)\cap I=0$.
\ep

The following result connects effectiveness and the ideal intersection property to properties of the induction map.

\begin{prop}\label{prop:IIP2}
Assume $\G$ is an amenable Hausdorff locally compact \'etale groupoid. Then the following conditions are equivalent:
\begin{enumerate}
  \item $\G$ is strongly effective;
  \item $\G$ has the residual intersection property;
  \item the induction map $\Ind\colon \Stab(\G)^\prim\to\Prim C^*(\G)$ factors through $\Gu$, that is, $\Ind(x,J)$ depends only on $x$, thus giving a map $\Ind_0\colon\Gu \to\Prim C^*(\G)$.
\end{enumerate}
Moreover, if these conditions are satisfied, then $\Ind_0$ defines a homeomorphism $\Ind_0^\sim$ of $(\G\backslash\Gu)^\sim$ onto $\Ind_0(\Gu)\subset\Prim C^*(\G)$ and, for every closed subset $\Omega\subset\Prim C^*(\G)$, the set $\Ind_0(\Gu)\cap\Omega$ is dense in~$\Omega$.
\end{prop}

\bp
The equivalence of (1) and (2) follows from Proposition~\ref{prop:IIP}.

\smallskip

Assume property (3) holds, but $\G$ is not strongly effective. To arrive at a contradiction we may pass to a closed invariant subset of $\Gu$ and assume that $\G$ is not effective. Construct functions~$f$ and~$f_0$ as in the proof of the implication (2)$\Rightarrow$(1) in the previous proposition. Take $x\in\Gu$ such that $f_0(x)\ne0$ and choose an irreducible unitary representation $\pi$ of $\Gxx$ such that $\pi((f-f_0)|_{\Gxx})\ne0$. Then $f-f_0$ lies in the kernel of $\Ind^{\G}_{\Gxx}\eps_x$, but not in the kernel of $\Ind^\G_{\Gxx}\pi$, since on the subspace $H_\pi\subset\Ind H_\pi$ the operator $(\Ind^\G_{\Gxx}\pi)(f-f_0)$ acts as $\pi((f-f_0)|_{\Gxx})$. This contradicts (3).

\smallskip

Assume now that (2) holds. For every ideal $I\subset C^*(\G)$, denote by $U(I)\subset\Gu$ the open (possibly empty) invariant set such that $C_0(\Gu)\cap I=C_0(U(I))$. The residual intersection property and inner exactness of $\G$ imply that the map $I\mapsto U(I)$ defines an isomorphism of the lattice of ideals in $C^*(\G)$ onto the lattice of open invariant subsets of~$\Gu$, with the inverse map sending an invariant open set $U$ to the ideal $C^*(\G_U)$ generated by $C_0(U)$, see~\cite{Rbook}*{Section~II.4}, \cite{BL}*{Section~3} or~\cite{CN2}*{Remark~4.6}. Since for every $(x,J)\in\Stab(\G)^\prim$ we have $U(\Ind(x,J))=\Gu\setminus\overline{[x]}$, it follows in particular that $\Ind(x,J)$ depends only on $x$, namely, $\Ind(x,J)=C^*(\G_{\overline{[x]}^c})$.

\smallskip

It remains to prove the last part of the proposition. Assume therefore that conditions (1)--(3) are satisfied. Since $U(\ker\rho_x)=\Gu\setminus\overline{[x]}$ for $\rho_x=\Ind^\G_{\{x\}}\epsilon_x$, we have $\Ind_0(x)=\ker\rho_x$. Lemma~\ref{lem:Ind-continuity} implies that the map $x\mapsto\ker\rho_x$ is continuous, hence it induces a continuous map $\Ind_0^\sim\colon (\G\backslash\Gu)^\sim\to\Prim C^*(\G)$.

Next, assume $Y\subset\Gu$ is an invariant subset and $\Ind_0(x)\in\overline{\Ind_0(Y)}$. Then
$$
\Gu\setminus\overline{[x]}=U(\Ind_0(x))\supset U\Big(\bigcap_{y\in Y}\Ind_0(y)\Big)=\Gu\setminus\bar Y,
$$
that is, $x\in\bar Y$. By~\cite{CN3}*{Lemma~1.4} it follows that $\Ind_0^\sim$ is a homeomorphism onto its image.

Finally, we need to show that $\Ind_0(\Gu)\cap\Omega$ is dense in~$\Omega$ for every closed subset $\Omega\subset\Prim C^*(\G)$. Since $\Omega$ is the hull of an ideal $C^*(\G_U)$, by using inner exactness of $\G$ we can pass to~$\G_{U^c}$ and therefore assume that $\Omega=\Prim C^*(\G)$. Then we need to check that $\bigcap_{x\in\Gu}\Ind_0(x)=0$. Since $\Ind_0(x)=\ker\rho_x$ and the representations $\rho_x$ define the reduced C$^*$-algebra of $\G$, amenability of $\G$ implies that $\bigcap_{x\in\Gu}\Ind_0(x)=0$.
\ep

\begin{remark} \label{rem:4nonAm}
A large part of the above proof works under weaker assumptions and shows the following: if an \'etale groupoid $\G$ is inner exact and has the residual intersection property, then we have a well-defined continuous map $\Ind_0\colon\Gu\to\Prim C^*_r(\G)$, $\Ind_0(x):=C^*_r(\G_{\overline{[x]}^c})$, this map defines a homeomorphism of $(\G\backslash\Gu)^\sim$ onto $\Ind_0(\Gu)\subset\Prim C^*_r(\G)$ and, for every closed subset $\Omega\subset\Prim C^*_r(\G)$, the set $\Ind_0(\Gu)\cap\Omega$ is dense in~$\Omega$. This corrects \cite{BL}*{Theorem~3.18}, whose proof works only under some countability assumptions. Similarly, Corollary~\ref{cor:BL} below remains true in the form of a homeomorphism $(\G\backslash\Gu)^\sim\cong\Prim C^*_r(\G)$ if instead of amenability and strong effectiveness we assume that the groupoid is inner exact and has the residual intersection property, cf.~\cite{BL}*{Corollary~3.19}. \ee
\end{remark}

As we used in the proof of Proposition~\ref{prop:IIP2}, the lattice of ideals in $C^*(\G)$ is isomorphic to the lattice of open invariant subsets of~$\Gu$. The latter lattice is, in turn, isomorphic to the lattice of open subsets of $(\G\backslash\Gu)^\sim$ (see~\cite{CN3}*{Corollary~1.5}). The last part of the proposition says a bit more and under suitable countability assumptions gives a complete description of $\Prim C^*(\G)$.

\begin{cor}[{cf.~\cite{EH}*{Corollary~5.16},\cite{Rbook}*{Proposition~II.4.6},\cite{SW}*{Lemma~4.6}}]\label{cor:BL}
Assume $\G$ is a strongly effective amenable Hausdorff locally compact \'etale groupoid with a second countable unit space~$\Gu$. Then the map
$
\Ind_0^\sim\colon (\G\backslash\Gu)^\sim\to\Prim C^*(\G)
$
is a homeomorphism.
\end{cor}

\bp
We need only to check that $\Ind_0$ is surjective. Take a point $I\in\Prim C^*(\G)$. Then its closure $\hull(I)$ is an irreducible closed set, meaning that it is not the union of two proper closed subsets. By the density property of the image of $\Ind_0$, the preimage $X:=\Ind_0^{-1}(\hull(I))$ is an irreducible closed invariant subset of $\Gu$, meaning now that it is not the union of two proper closed invariant subsets. Then $X=\overline{[x]}$ for some $x\in\Gu$ by second countability of~$\Gu$, see, e.g.,~\cite{CN2}*{Lemma~6.2}. It follows that $\Ind_0(x)$ is dense in $\hull(I)$, hence $I=\Ind_0(x)$.
\ep

We have so far used in this subsection an approach to the ideal structure of $C^*(\G)$ inspired by the literature on the ideal intersection property. This approach is largely independent of other results of the paper and, from our perspective, has two main advantages. First, it is known to be applicable beyond the amenable case (see Remark~\ref{rem:4nonAm} and \cite{Rbook,MR2775364,BL,KKLRU}), and second, it does not need surjectivity of $\Ind\colon\Stab(\G)^\prim\to\Prim C^*(\G)$ as a starting point.

Let us now argue that in the second countable case, when we do know that the induction map $\Ind\colon\Stab(\G)^\prim\to\Prim C^*(\G)$ is surjective for amenable groupoids~\cite{IW}, Corollary~\ref{cor:BL} can easily be deduced from the results of the present paper. Indeed, we need to show that if we are given a $\G$-invariant set $\Omega\subset\Stab(\G)^\prim$ and a point $(x,J)\in\Stab(\G)^\prim$, then $\Ind(x,J)\in\overline{\Ind \Omega}$ if and only if $x\in\bar Y$, where $Y$ is the image of $\Omega$ under the projection $\Stab(\G)^\prim\to\Gu$.
By looking at the intersections of ideals with $C_0(\Gu)$, we immediately get the ``only if''  part. Conversely, assume $x\in\bar Y$. Since the set of units $y\in\bar Y$ with trivial isotropy groups is dense in~$\bar Y$, we can find $y_n\in Y$ such that $y_n\to x$ and $\G^{y_n}_{y_n}\to\{x\}$ in $\Sub(\G)$. Take any $J_n\in\Prim C^*(\G^{y_n}_{y_n})$ such that $(y_n,J_n)\in\Omega$. Then $(\G^{y_n}_{y_n},J_n)\to(\{x\},0)$ in $\Sub(\G)^\prim$. This shows that condition~(2) in Theorem~\ref{thm:main-prim} is satisfied. Since the implication (2)$\Rightarrow$(1) in this theorem relies only on Lemma~\ref{lem:Ind-continuity} and does not need any assumptions on the isotropy groups, we conclude that $\Ind(x,J)\in\overline{\Ind \Omega}$.

We note in passing that, with minor modifications in the non-second-countable case, this argument proves also the implication (1)$\Rightarrow$(3) in Proposition~\ref{prop:IIP2} independently of the ideal intersection property.

\bigskip

\section{The primitive spectrum of \texorpdfstring{$\SL_3(\Z)\ltimes C_0(\SL_3(\R)/U_3(\R))$}{SL3}}\label{sec:SL3}

\subsection{Transformation groupoids defined by unipotent flows}\label{ssec:unipotent}

A large class of groupoids to which results of Section~\ref{sec:primitive} apply can be obtained by considering group actions on homogeneous spaces of locally compact groups. Namely, assume $G$ is a second countable locally compact group, $U\subset G$ is a closed amenable subgroup and $\Gamma\subset G$ is a discrete subgroup. The action $\Gamma\curvearrowright G/U$ by left translations is amenable, see \cite{MR1799683}*{Theorem 2.2.17 and Example 2.2.16}. The stabilizer of $gU\in G/U$ is $\Gamma\cap gUg^{-1}$. If these stabilizers happen to be in the class $\M$ introduced in Definition~\ref{def:M}, FC-hypercentral or locally finite, then the primitive spectrum of $\Gamma\ltimes C_0(G/U)$ is described by Theorem~\ref{thm:main-prim},~\ref{thm:main-prim-fc} or~\ref{thm:main-prim-lf}.

For this to happen we may impose additional conditions on $\Gamma$. For example, as we discussed in Sections~\ref{ssec:main} and~\ref{ssec:FC}, we may require $\Gamma$ to be relatively hyperbolic with respect to a family of subgroups that are either FC-hypercentral or of local polynomial growth. Such groups $\Gamma$ include free groups and lattices in semisimple Lie groups of real rank one. We note in passing that every connected Lie group $G$ such that $G/\operatorname{Rad}G$ is noncompact contains discrete free subgroups, since $\operatorname{PSL}_2(\R)$ is a subquotient of $G$ (because every noncompact real semisimple Lie algebra contains a copy of $\mathfrak{sl}_2(\R)$) and $\operatorname{PSL}_2(\Z)$ has plenty of free subgroups.

Another possibility is to take arbitrary discrete $\Gamma\subset G$ and nilpotent~$U$. Then the stabilizers are nilpotent as well and the primitive spectrum is described by Theorem~\ref{thm:main-prim-fc}. For example, we can take $G=\SL_n(\R)$ and $U$ to be the group $U_n(\R)$ of unipotent upper triangular matrices, that is, upper triangular matrices with ones on the diagonal.

For $n=2$ we have $U_2(\R)\cong\R$, so in this case, for any discrete subgroup $\Gamma\subset\SL_2(\R)$, the stabilizers of the action $\Gamma\curvearrowright G/U$ are either trivial or isomorphic to~$\Z$, and we can apply \cite{CN3}*{Corollary~2.8}. Note that neither of our results is needed if the action $\Gamma\curvearrowright G/U$ happens to be minimal, as then $\Gamma\ltimes C_0(G/U)$ is simple. By a classical theorem of Hedlund~\cite{MR1545946}, see also~\cite{MR0393339}, this is in particular the case when $\Gamma\subset\SL_2(\R)$ is a uniform lattice. We remind that a lattice is called uniform or cocompact, if $\Gamma\backslash G$ is compact.

Therefore the smallest $n$ for which we can get noncommutative stabilizers for an action $\Gamma\curvearrowright \SL_n(\R)/U_n(\R)$ is $n=3$. Our goal is to give a complete description of the primitive spectrum of $\SL_3(\Z)\ltimes C_0(\SL_3(\R)/U_3(\R))$. Along the way we will obtain everything that is needed to describe $\Prim(\SL_2(\Z)\ltimes C_0(\SL_2(\R)/U_2(\R)))$ as well, but we leave it to the interested reader to formulate and prove the final result in this case.
%to keep the length of the paper within reasonable bounds we will refrain from any explicit discussion of this spac

\smallskip

The key result making our analysis possible is the following fundamental theorem of Ratner.

\begin{thm}[{\cite{MR1106945}*{Theorem A}}]\label{thm:ratner}
Assume $G$ is a connected Lie group, $\Gamma\subset G$ is a lattice and $U\subset G$ is an $\Ad$-unipotent subgroup. Then, for every $g\in G$, the closure of the $U$-orbit of $\Gamma g$ in $\Gamma\backslash G$ has the form $\Gamma gH$, where $H\subset G$ is a closed subgroup such that $U\subset H$ and $g^{-1}\Gamma g\cap H$ is a lattice in $H$.
\end{thm}

Recall that a subgroup $U\subset G$ is called \emph{$\Ad$-unipotent}, if for every $g\in U$ the transformation $\Ad g\colon\g\to\g$ is unipotent, that is, $\Ad g-\operatorname{id}$ is nilpotent.

We note that since the maps $G\to \Gamma\backslash G$ and $G\to G/U$ are open, we have one-to-one correspondences between the closed $U$-invariant subsets of $\Gamma\backslash G$, the closed $\Gamma\times U$-invariant subsets of~$G$, where $\Gamma$ acts by left translations and $U$ by right translations, and the closed $\Gamma$-invariant subsets of $G/U$. Therefore Ratner's theorem describes equally well the closures of $\Gamma$-orbits in~$G/U$ and the closures of $\Gamma\times U$-orbits in~$G$. As a result we do not have to worry about the meaning of expressions like~$\overline{\Gamma g U}$: this meaning should be clear from the context, but even if the reader is in doubt, all possible interpretations give equivalent statements.

\subsection{Quasi-orbits of the action \texorpdfstring{$\SL_2(\Z)\curvearrowright\R^2\times\T^2$}{SL2}}\label{ssec:SL2-R2-T2}
Our first goal is to understand the quasi-orbit space for the diagonal action of $\SL_2(\Z)$ on $\R^2\times\T^2$, where the action on the second factor is defined by identifying $\T^2$ with $\R^2/\Z^2$. We remind that by the quasi-orbit of a point $x\in X$ for an action $\Gamma\curvearrowright X$ one means the equivalence class of $x$ with respect to the following equivalence relation on $X$:
$$
x\sim y\quad\text{iff}\quad \overline{\Gamma x}=\overline{\Gamma y}.
$$
The space of quasi-orbits is therefore a quotient of $X$, and we equip it with the quotient topology.

\smallskip

We start by recording the following reformulation of a classical result of Hedlund~\cite{MR1545946}, see also~\cite{MR0744294}, which implies Theorem~\ref{thm:ratner} for $G=\SL_2(\R)$, $U=U_2(\R)$ and nonuniform lattices $\Gamma\subset G$ by identifying $G/U$ with $\R^2\setminus\{0\}$.

\begin{lemma}\label{lem:Hedlund}
Assume $\Gamma\subset\SL_2(\R)$ is a nonuniform lattice. Then, for any $v\in\R^2\setminus\{0\}$, either the stabilizer of $v$ in $\Gamma$ is trivial and then the $\Gamma$-orbit of $v$ is dense in $\R^2$, or the stabilizer is nontrivial and then the orbit is closed in $\R^2\setminus\{0\}$.
\end{lemma}

The following result is probably known to experts, cf.~\cite{MR2060024}*{Section~2.6}, \cite{MR2165547}.

\begin{lemma}\label{lem:SL2-1}
Let $p\colon \R^2\to\T^2$ be the quotient map. Assume we are given vectors $v\in\R^2\setminus\R\cdot\Q^2$ and $u\in\R^2$. Then
\begin{enumerate}
\item if $u-tv\in\Q^2$ for some $t\in\R$, then the closure of the $\SL_2(\Z)$-orbit of $(v,p(u))$ in $\R^2\times\T^2$ consists of all points $(v',z'p(tv'))$ such that $z'\in\T^2$ has the same order as $p(u-tv)$;
\item if $u-tv\notin\Q^2$ for all $t\in\R$, then the $\SL_2(\Z)$-orbit of $(v,p(u))$ is dense in $\R^2\times\T^2$.
\end{enumerate}
\end{lemma}

Note that the number $t$ in (1) is uniquely determined, since $v\notin\R\cdot\Q^2$.

\bp
(1) It is easy to see that the set of points $(v',z'p(tv'))$ as in the formulation is closed and contains the $\SL_2(\Z)$-orbit of $(v,p(u))$, so we just have to show that the orbit is dense in it.

Let $n$ be the order of $p(u-tv)$ in $\T^2$. Let us show first that $\SL_2(\Z)$ acts transitively on the elements of $\T^2$ of order $n$.  For this observe that the $\SL_2(\Z)$-orbit of every such element has a point of the form $(z,1)$, since $\Q^2=\SL_2(\Z)\begin{pmatrix}\Q \\ 0 \end{pmatrix}$. As $z=e^{2\pi ia/n}$ for some $a$ with $\gcd(a,n)=1$, we then have $(z,1)=\gamma(e^{2\pi i/n},1)$ for any $\gamma\in\SL_2(\Z) $ of the form $\gamma=\begin{pmatrix}
                          a & * \\
                          n & *
                        \end{pmatrix}$. This proves the claimed transitivity.

Now, consider the element $z:=p(u-tv)$. As we have just shown, if $z'\in\T^2$ is another element of order $n$, then we can find $\gamma\in\SL_2(\Z)$ such that $z'=\gamma z$. Consider the stabilizer $\Gamma\subset\SL_2(\Z)$ of $z'$. It is a finite index subgroup of $\SL_2(\Z)$, hence a lattice in $\SL_2(\R)$. By Lemma~\ref{lem:Hedlund} and the assumption $v\notin\R\cdot\Q^2$, the $\Gamma$-orbit of $\gamma v$ is dense in $\R^2$. Hence
$$
\overline{\SL_2(\Z)(v,p(u))}\supset\overline{\Gamma\gamma(v,zp(tv))}=\overline{\Gamma(\gamma v,z'p(t\gamma v))}=\{(v',z'p(tv')):v'\in\R^2\}.
$$
As $z'$ was an arbitrary element of order $n$, this finishes the proof of (1).

\smallskip

(2) Consider the group $G:=\R^2\rtimes \SL_2(\R)$ and the lattice $\Gamma:=\Z^2\rtimes\SL_2(\Z)$ in it. We view $U:=U_2(\R)$ as a subgroup of $G$. Let $g\in\SL_2(\R)$ be such that $g\begin{pmatrix}1 \\ 0\end{pmatrix}=v$. Then we need to show that the $U$-orbit of $\Gamma(u,g)$ in $\Gamma\backslash G$ is dense. By Theorem~\ref{thm:ratner} the closure of the orbit has the form $\Gamma(u,g)H$ for a closed subgroup~$H$ with properties as stated there, and we want to show that $H=G$.

Since the group $H$ contains a lattice, it is unimodular, hence so is its quotient $H_0$ by $H\cap\R^2$. We view $H_0$ as a Lie subgroup of $\SL_2(\R)$. Since $H_0$ contains $U$ and is unimodular, there are very few options for it: $H_0$ is either $U$, $U\times\{\pm1\}$ or $\SL_2(\R)$. The image of $\Gamma (g,h)H$ in $\SL_2(\Z)\backslash\SL_2(\R)$ is closed by properness of the map $\Gamma\backslash G\to \SL_2(\Z)\backslash\SL_2(\R)$. It follows that the closure of the $U$-orbit of $\SL_2(\Z)g\in \SL_2(\Z)\backslash\SL_2(\R)$ coincides with $\SL_2(\Z)gH_0$. Since we already know that this closure coincides with $\SL_2(\Z)\backslash\SL_2(\R)$, it follows that $H_0=\SL_2(\R)$.

Then $H\cap\R^2$ is an $\SL_2(\R)$-invariant subgroup of $\R^2$, hence it is either $\R^2$ or $\{0\}$. In the first case we get $H=G$ and we are done. Assume therefore that the intersection $H\cap\R^2$ is trivial. Then the quotient map $G\to \SL_2(\R)$ defines an isomorphism $H\cong\SL_2(\R)$, hence the Lie algebra of $H$ has the form $\mathfrak h=\{(\varphi(X),X): X\in\mathfrak{sl}_2(\R)\}$ for a $1$-cocycle $\varphi\colon\mathfrak{sl}_2(\R)\to\R^2$. By Whitehead's lemma we have $\varphi(X)=Xw$ for some $w\in\R^2$. Hence
$$
H=\{(hw-w,h):h\in\SL_2(\R)\}.
$$
Since $H\supset U$, we must have $w=t\begin{pmatrix}1 \\ 0\end{pmatrix}$ for some $t\in\R$.

Therefore the closure of $\Gamma (u,g)U$ in $G$ is equal to
$$
\Gamma(u,g)H=\{(\gamma u+\gamma ghw-\gamma gw+a,\gamma gh):\gamma\in\SL_2(\Z),\ h\in\SL_2(\R),\ a\in\Z^2\}.
$$
In particular, we get
\begin{align*}
\overline{\Gamma (u,g)U}\cap \R^2&=\{\gamma u+w-\gamma gw +a:\gamma\in\SL_2(\Z),\ a\in\Z^2\}\\
&=\{\gamma (u-tv)+w +a:\gamma\in\SL_2(\Z),\ a\in\Z^2\},
\end{align*}
where we used that $gw=tv$. But the last set is certainly not closed in $\R^2$, since $u-tv\not\in\Q^2$ by assumption and hence, as can be easily seen, the $\SL_2(\Z)$-orbit of $p(u-tv)$ is not closed in $\T^2$. This contradiction finishes the proof of the lemma.
\ep

Denote by $\QQ(\R^2\times\T^2)$ the quasi-orbit space for the action of $\SL_2(\Z)$ on $\R^2\times\T^2$. We denote by the same letter $\QQ$ the quotient map $\R^2\times\T^2\to \QQ(\R^2\times\T^2)$. The quotient map is open and as a result we have a bijective correspondence between the $\SL_2(\Z)$-invariant closed subsets of~$\R^2\times\T^2$ and the closed subsets of $\QQ(\R^2\times\T^2)$, see~\cite{CN3}*{Corollary~1.5}.

\begin{lemma}\label{lem:SL2-2}
Assume we are given a vector $v\in\R^2\setminus\R\cdot\Q^2$, a sequence $(z_n)_n$ of elements of finite order in $\T^2$ and a sequence $(t_n)_n$ of real numbers. Then
\begin{enumerate}
  \item if $\ord z_n=m$ for all $n$ and $t_n\to t\in\R$ as $n\to+\infty$, then the cluster points of the sequence $(\QQ(v,z_np(t_nv)))_n$ are the points $\QQ(v',z'p(tv'))$ such that $z'\in\T^2$ has order~$m$;
  \item if either $\ord z_n\to+\infty$ or $|t_n|\to+\infty$ as $n\to+\infty$, then every point of $\QQ(\R^2\times\T^2)$ is a cluster point of the sequence $(\QQ(v,z_np(t_nv)))_n$.
\end{enumerate}
\end{lemma}

\bp
(1) It is easy to see that every cluster point of $(\QQ(v,z_np(t_nv)))_n$ must be of the form $\QQ(v',z'p(tv'))$, with $\ord z'=m$. The fact that every such point $\QQ(v',z'p(tv'))$ is indeed a cluster point follows from Lemma~\ref{lem:SL2-1}(1), since this lemma implies that already the closure of $\{\QQ(v,zp(tv))\}$, where $z\in\T^2$ is an element of order $m$, contains all such points.

\smallskip

(2) Assume first that $\ord z_n\to+\infty$. Then, for any $z\in\T^2$, we can find $z_n'\in\T^2$ such that $\ord z_n'=\ord z_n$ and $z_n'p(t_nv)\to z$, since the distance between two consecutive primitive roots of unity of the same order goes to zero when the order goes to infinity (see, e.g.,~\cite{MR0330079} for a much stronger result). As $\QQ(v,z_np(t_nv))=\QQ(v,z_n'p(t_nv))$ by Lemma~\ref{lem:SL2-1}(1), it follows that $\QQ(v,z_np(t_nv))\to\QQ(v,z)$. As the set of points $\QQ(v,z)$, $z\in\T^2$, is dense in $\QQ(\R^2\times\T^2)$ by the same lemma, we are done.

Assume next that $|t_n|\to+\infty$. By passing to a subsequence we may assume that the sequence~$(z_n)_n$ is convergent and $p(t_nv)\to z$ for some $z\in\T^2$. It is then enough to show that every point $\QQ(v,z')$, $z'\in\T^2$, is a cluster point of $(\QQ(v,p(t_nv)))_n$. Pick $u,u'\in\R^2$ such that $p(u)=z$, $p(u')=z'$. Define $v_n:=v+t_n^{-1}(u'-u)$. Then $(v_n,p(t_nv_n))\to(v,z')$. Since $\QQ(v_n,p(t_nv_n))$ lies in the closure of $\{\QQ(v,p(t_nv))\}$ by Lemma~\ref{lem:SL2-1}(1), it follows that $\QQ(v,p(t_nv))\to\QQ(v,z')$.
\ep

Next we consider points in $\QQ(\R\cdot\Q^2\times\T^2)$. We need the following elementary observation, which will also be useful later.

\begin{lemma} \label{lem:modd}
For any $s,t\in\R^\times$ and any odd integer $m$, there exist $\gamma \in\SL_2(\Z)$ and $u\in U_{2}(\R)$ such that
$$
\gamma
\begin{pmatrix}
s & 0 \\
0 & t
\end{pmatrix} u
=
\begin{pmatrix}
ms & 0 \\
2s & t/m
\end{pmatrix} \; .
$$
\end{lemma}

\begin{proof}
If $m=2n+1$, we can take
$\displaystyle \gamma=\begin{pmatrix}
         m & n \\
         2 & 1
       \end{pmatrix}$, $\displaystyle u=\begin{pmatrix}
         1 & -nt/ms \\
         0 & 1
       \end{pmatrix}$.
\end{proof}

%The following partial description of closures of subsets of $\QQ(\R\cdot\Q^2\times\T^2)$ in $\QQ(\R^2\times\T^2)$ will be enough for our purposes.

\begin{lemma}\label{lem:SL2-3}
Assume we are given a sequence of vectors $v_n=s_n e_1\in\R^2\setminus\{0\}$, where $e_1=\begin{pmatrix}1\\ 0\end{pmatrix}$, and a sequence of points $z_n=(a_n,b_n)\in\T^2$. Then the following properties hold.
\begin{enumerate}
  \item If $s_n\to+\infty$, then $(\QQ(v_n,z_n))_n$ has no cluster points in $\QQ(\R^2\times\T^2)$.
  \item If $s_n\to s\in(0,+\infty)$, then the cluster points of $(\QQ(v_n,z_n))_n$ are contained in $\QQ(s\Z^2\times\T^2)$.
  \item If $s_n\to0$ and $\ord b_n\to+\infty$, then every point of $\QQ(\R^2\times\T^2)$ is a cluster point of~$(\QQ(v_n,z_n))_n$.
  \item Assume $s_n\to0$, $\ord b_n=m\in \N$ for all $n$ and $a_n\to a\in\T$. Then
\begin{enumerate}
  \item if $\ord a=+\infty$ or $\frac{|a_n-a|}{s_n}\to+\infty$, then every point of $\QQ(\R^2\times\T^2)$ is a cluster point of~$(\QQ(v_n,z_n))_n$;
  \item if $\ord a=k\in\N$ and $\frac{a_n-a}{s_n}\to 2\pi it a$ ($t\in\R$), then the cluster points of~$(\QQ(v_n,z_n))_n$ are the points $\QQ(v',z'p(tv'))$ with $\ord z'=\operatorname{lcm}(k,m)$.
\end{enumerate}
\end{enumerate}
\end{lemma}

\bp (1) and (2) are straightforward.

\smallskip

(3) Let us show first that for every $v\in\R^2$ there exist elements $\gamma_n\in\SL_2(\Z)$ such that $\gamma_n v_n\to v$. It is enough to consider $v\in s\Z^2$ for all $s>0$, and since $\SL_2(\Z)$ acts transitively on the primitive elements of the lattice $s\Z^2$, it suffices to take $v=se_1$. Choose odd integers~$m_n$ such that $m_ns_n\to s$. Then by Lemma~\ref{lem:modd} we can find $\gamma_n\in\SL_2(\Z)$ such that
$\gamma_n v_n=\begin{pmatrix}m_ns_n\\ 2s_n\end{pmatrix}\to se_1$.

Now, fix $v\in\R^2\setminus\R\cdot\Q^2$. Let $\gamma_n\in\SL_2(\Z)$ be such that $\gamma_nv_n\to v$. By passing to a subsequence we may assume that $\gamma_n z_n\to z$ for some $z\in\T^2$. Pick $u\in\R^2$ such that $p(u)=z$. Consider the subgroup $T_n\subset\T^2$ of elements of the form~$(c,1)$ with $c^{\ord b_n}=1$, where we use the convention that if $\ord b_n=+\infty$, then $c\in\T$ can be arbitrary. Since the matrices $\begin{pmatrix}
1 & k \\
0 & 1
\end{pmatrix}$, $k\in\Z$, stabilize~$v_n$ and map $z_n$ into $(a_n b_n^k,b_n)$, we have $\QQ(v_n,z_n)=\QQ(v_n,z_nw_n)$ for all $w_n\in T_n$.

By passing to a subsequence we may assume that $\gamma_nT_n\to T$ in $\Sub(\T^2)$ for some closed subgroup $T\subset\T^2$. Then, for every $w\in T$, we can find $w_n\in T_n$ such that $\gamma_nw_n\to w$, hence $\QQ(v_n,z_n)=\QQ(v_n,z_nw_n)=\QQ(\gamma_nv_n,\gamma_n(z_nw_n))\to\QQ(v,zw)$. It follows that if $T=\T^2$, then (3) is proved, since by Lemma~\ref{lem:SL2-1} the closure of the set of points $\QQ(v,z')$, $z'\in\T^2$, is the entire space $\QQ(\R^2\times\T^2)$. Assume therefore that $T\ne\T^2$.

We claim that the group $T$ is infinite. By \cite{CN3}*{Lemma~3.3} we have $(\gamma_nT_n)^\perp\to T^\perp$ in~$\Sub(\Z^2)$. As the group $T^\perp$ is finitely generated, this implies that $T^\perp\subset(\gamma_nT_n)^\perp$ for all $n$ large enough, hence $\gamma_nT_n\subset T$, and the claim follows, since $|T_n|\to+\infty$ by our assumption that $\ord b_n\to+\infty$.

As $T\ne\T^2$, we conclude that the connected component of the identity in $T$ is a circle in $\T^2$, so it has the form $T^\circ=p(\R u')$ for some $u'\in\Z^2$. Since $v\in\R^2\setminus\R\cdot\Q^2$, for every $u''\in\Z^2$ there is a unique $t\in\R$ such that $tv\in u+u''+\R u'$, which implies $p(tv)\in p(u+\R u')=zT^\circ\subset zT$. By varying $u''$ we can then find a sequence of real numbers $t_k$ such that $t_k\to+\infty$ and $p(t_kv)\in zT$ for all $k$. Thus, the set of cluster points of $(\QQ(v_n,z_n))_n$ contains the points $\QQ(v,p(t_kv))$. Since the latter points form a dense subset of $\QQ(\R^2\times\T^2)$ by Lemma~\ref{lem:SL2-2}(2), this finishes the proof of~(3).

\smallskip

(4a) Assume first that $\ord a=+\infty$. As in part (3), for every $s>0$ we can find $\gamma_n\in\SL_2(\Z)$ such that $\gamma_n v_n=\begin{pmatrix}m_ns_n\\ 2s_n\end{pmatrix}$ and $m_ns_n\to s$. Then $\gamma_n z_n=(a_n^{m_n},a_n^2)\gamma_n(1,b_n)$. It follows that for every $s>0$ the sequence $(\QQ(v_n,z_n))_n$ has a cluster point of the form $\QQ(se_1,(*,a^2)z')$ for some~$z'\in\T^2$ of order~$m$. As $\ord a^2=+\infty$, the set of such points is dense in $\QQ(\R^2\times\T^2)$ by (3), and we are done.

Assume next that $\ord a=k\in\N$ and $\frac{|a_n-a|}{s_n}\to+\infty$. Choose $q\in\Q$ such that $a=e^{2\pi i q}$. Let $q_n\in\R$ be such that $e^{2\pi iq_n}=a_n$ and $q_n\to q$. As before, for every $s>0$ we can find $\gamma_n\in\SL_2(\Z)$ such that $\gamma_n v_n=\begin{pmatrix}m_ns_n\\ 2s_n\end{pmatrix}$, $m_n$ is odd and $m_ns_n\to s$. By passing to a subsequence we may assume that $a_n^{m_n}\to e^{2\pi isr}$ for some $r\in\R$. For every $t\in\R$ we can find even integers~$k_n$ such that $k_n(q_n-q)\to s(t-r)$. As $\frac{|q_n-q|}{s_n}\to+\infty$, we have $k_ns_n\to0$.

Let $\gamma_n'\in\SL_2(\Z)$ be such that $\gamma'_n v_n=\begin{pmatrix}(m_n+k_n)s_n\\ 2s_n\end{pmatrix}\to se_1$. Note that
$$
a_n^{m_n+k_n}=a_n^{m_n}e^{2\pi i k_n(q_n-q)}e^{2\pi ik_n q}=a_n^{m_n}e^{2\pi i k_n(q_n-q)}a^{k_n},
$$
and therefore
$$
\gamma'_n z_n=(a_n^{m_n}e^{2\pi i k_n(q_n-q)},1)(a^{k_n},a_n^2)\gamma'_n(1,b_n).
$$
The first factor in the expression above converges to $p(ste_1)$. Recalling that $\ord a=k$ and $\ord b_n=m$, we conclude that
$(\QQ(v_n,z_n))_n$ has a cluster point of the form $\QQ(se_1,p(ste_1)z')$, where $z'$ is a $km$-th root of unity. Since this is true for all $s>0$ and $t\in\R$, it follows that for all $v\in\R^2$ and $t\in\R$ the sequence $(\QQ(v_n,z_n))_n$ has a cluster point of the form $\QQ(v,p(tv)z')$, where~$z'$ is a $km$-th root of unity. Since the set of such points is dense in $\QQ(\R^2\times\T^2)$ by Lemma~\ref{lem:SL2-2}(2), this finishes the proof of (4a).

\smallskip

(4b) As above, choose $q\in\Q$ and $q_n\in\R$ such that $a=e^{2\pi i q}$, $a_n=e^{2\pi iq_n}$ and $q_n\to q$. We can then write
$$
z_n=(e^{2\pi i (q_n-q)},1)(e^{2\pi i q},b_n)=p\Big(\frac{q_n-q}{s_n}v_n\Big)(a,b_n).
$$
As $\frac{q_n-q}{s_n}\to t$ by assumption, we see from this that every cluster point of $(\QQ(v_n,z_n))_n$ must be of the form $\QQ(v,p(tv)z')$ with $z'$ of order $\operatorname{lcm}(k,m)$. For every $v\in\R^2$ we indeed have a cluster point of this form, since, as we already used in the proofs of (3) and (4a), we can find $\gamma_n\in\SL_2(\Z)$ such that $\gamma_nv_n\to v$. But once we get one such cluster point with $v\in\R^2\setminus\Q\cdot\R^2$, we get all points of this form as cluster points by Lemma~\ref{lem:SL2-1}(1).
\ep

\begin{remark}
Since the assumptions on sequences in Lemmas~\ref{lem:SL2-2} and~\ref{lem:SL2-3}(3)--(4b) are also satisfied by every subsequence, the cluster points there are actually limits of the respective sequences. \ee
\end{remark}

We remark that in Lemma~\ref{lem:SL2-3}(2) it is not difficult to obtain a complete description of the cluster points, but we omitted it, since we will not need it. It is also easy to describe the cluster points of the sequences $(\QQ(0,z_n))_n$, but we didn't do this for the same reason. Modulo these small omissions, Lemmas~\ref{lem:SL2-2} and~\ref{lem:SL2-3} allow one, in principle, to completely understand the closures of subsets in $\QQ(\R^2\times\T^2)$, since every sequence in $\QQ\big((\R^2\setminus\{0\})\times\T^2\big)$ has a subsequence as in these lemmas.

\subsection{{Quasi-orbits of the action \texorpdfstring{$\SL_3(\Z)\curvearrowright\SL_3(\R)/U_3(\R)$}{SL3}}}\label{ssec:quasi-orbit-SL3}

Throughout this subsection we assume that
$$
G=\SL_3(\R),\quad U=U_3(\R),\quad\Gamma=\SL_3(\Z).
$$
Define the following subgroups of $G$:
$$
P=P_{0}:
= \left\{
\begin{pmatrix}
* & * & * \\
0 & * & * \\
0 & 0 & *
\end{pmatrix} \right \} \quad , \quad
P_{1}
:= \left\{
\begin{pmatrix}
* & * & * \\
* & * & * \\
0 & 0 & *
\end{pmatrix} \right \}  \quad , \quad
P_{2}
:= \left\{
\begin{pmatrix}
* & * & * \\
0 & * & * \\
0 & * & *
\end{pmatrix} \right\},
$$
$$
Q_{1}
:= \left\{
\begin{pmatrix}
* & * & * \\
* & * & * \\
0 & 0 & 1
\end{pmatrix} \right \}  \quad , \quad
Q_{2}
:= \left\{
\begin{pmatrix}
1 & * & * \\
0 & * & * \\
0 & * & *
\end{pmatrix} \right\} \; .
$$
Then $P_{1}=N(Q_{1})$, $P_{2}=N(Q_{2})$ and $P_{0}=N(U)$, where $N(H)$ denotes the normalizer of a subgroup $H$. Denote by $P_{1,+}\subset P_1$, resp. $P_{2,+}\subset P_2$, the subgroup of elements $g$ such that $g_{33}>0$, resp. $g_{11}>0$.
Denote by $P_+\subset P$ the subgroup of matrices with positive elements on the diagonal. We will also sometimes write $Q_0$ for $U$ and denote by $P_n(\Q)\subset\SL_n(\Q)$ and $P_n(\R)\subset\SL_n(\R)$ the subgroups of upper triangular matrices, so in particular $P_{3}(\mathbb{R})=P$.

\begin{lemma}\label{lem:SLQ-Z}
For every $n\ge2$, we have $\SL_n(\Q)=\SL_n(\Z)P_n(\Q)$.
\end{lemma}

\bp
Take $g\in \SL_n(\Q)$. Since $\SL_n(\Z)$ acts transitively on the primitive elements of $\Z^n$, we can find $\gamma\in\SL_n(\Z)$ such that the first column of $\gamma g$ has only one nonzero entry $(\gamma g)_{11}$. This proves the lemma for $n=2$ and for $n\ge3$ reduces it to $n-1$, as we now have to deal only with the right lower $(n-1)\times(n-1)$ block of~$\gamma g$.
\ep

\begin{lemma} \label{lem:DM}
The orbit closures for the action $\Gamma\curvearrowright G/U$ are described as follows:
\begin{enumerate}
\item $\overline{\Gamma  gU}=\Gamma  gU$ for $g \in \Gamma P$;
\item $\overline{\Gamma  gU}=\Gamma  gQ_{1}$ for $g\in \Gamma P_{1} \setminus \Gamma P$;
\item $\overline{\Gamma  gU}=\Gamma  gQ_{2}$ for $g\in \Gamma P_{2} \setminus \Gamma P$;
\item $\overline{\Gamma  g U}=G$ for $g\in G\setminus (\Gamma  P_{1} \cup \Gamma  P_{2})$.
\end{enumerate}
\end{lemma}
\bp
By Ratner's theorem, for every $g\in G$, we have $\overline{\Gamma  g U}=\Gamma g H$ with $H$ as in Theorem~\ref{thm:ratner}. As was shown by Dani and Margulis even earlier, in~\cite{DM}*{Theorem~1.2}, for $G=\SL_3(\R)$ the only choices for~$H$ are in fact $U$, $Q_1$, $Q_2$ and $G$. Moreover, by~\cite{DM}*{Proposition~1.4}, if we have one point $g\in G$ with $\overline{\Gamma  gU}=\Gamma g Q_i$ for some $0\le i\le 2$ (where $Q_0=U$), then any other $g'\in G$ with $\overline{\Gamma  g'U}=\Gamma g' Q_i$ must lie in $\SL_3(\Q)gP_i$. We remark that here %in translating the results of~\cite{DM} to our particular case
we implicitly use the well-known fact that the commensurator of~$\Gamma$ in~$G$ is~$\SL_3(\Q)$.

Consider first $g=e$. Assume $\gamma_nU\to hU$ in $G/U$ for some $\gamma_n\in\Gamma$ and $h\in G$. Then the first column of $\gamma_n$ converges to that of $h$ and hence eventually stabilizes. Therefore, if $m$ and $n$ are large enough, then $\gamma_{m}^{-1}\gamma_n\in \Gamma\cap Q_2$. But then from the convergence $\gamma_{m}^{-1}\gamma_n U\to \gamma_{m}^{-1}hU$ as $n\to\infty$ we similarly see that the coefficients $(\gamma_{m}^{-1}\gamma_n)_{22}$ and $(\gamma_{m}^{-1}\gamma_n)_{32}$ must stabilize. It follows that if $k$ and $n$ are large enough, then $\gamma_k^{-1}\gamma_n=(\gamma_{m}^{-1}\gamma_k)^{-1}\gamma_{m}^{-1}\gamma_n\in\Gamma\cap U$. Hence $\gamma_nU=\gamma_kU$ and therefore $hU=\gamma_k U$, so the $\Gamma$-orbit of $U$ in $G/U$ is closed. As~$P$ normalizes $U$, by acting by~$P$ on the right we conclude that the same is true for the orbit of $gU$ for all $g\in P$, hence for all $g\in\Gamma P$. By the results of~\cite{DM} we discussed above, any other point $gU$ with a closed orbit must have $g\in\SL_3(\Q)P$. As $\SL_3(\Q)P=\Gamma P$ by Lemma~\ref{lem:SLQ-Z}, we conclude that the elements $g\in\Gamma P$ are exactly the ones for which $\overline{\Gamma g U}=\Gamma g U$.

Next, take $g\in P_{2,+}\setminus\Gamma P$. Consider the $2\times 2$ matrix $g':=g_{11}^{-1/2}(g_{ij})_{i,j=2}^3\in\SL_2(\R)$. Then $g'\notin\SL_2(\Z)P_2(\R)=\SL_2(\Q)P_2(\R)$. By Lemma~\ref{lem:Hedlund} it follows that $\overline{\SL_2(\Z)g' U_2(\R)}=\SL_2(\R)$. This implies that $\overline{\Gamma g U}\supset \Gamma g Q_2$. On the other hand, for the same reasons as in the previous paragraph, if $\gamma_ngU\to hU$ in $G/U$ for some $\gamma_n\in\Gamma$ and $h\in G$, then for all $m$ and $n$ large enough we have $\gamma_{m}^{-1}\gamma_n\in \Gamma\cap Q_2$ and therefore $\gamma_ngU\in \gamma_mQ_2gU=\gamma_m gQ_2$, as $g$ normalizes~$Q_2$. Since $gQ_2\subset G$ is closed, it follows that $hU\in \gamma_m gQ_2$ in $G/U$. Thus, $\Gamma g Q_2$ is closed and hence $\overline{\Gamma g U}=\Gamma g Q_2$. Then the same is true for all $g\in\Gamma (P_{2,+}\setminus\Gamma P)=\Gamma P_2\setminus\Gamma P$. By the results of~\cite{DM} we discussed above, for any fixed $g\in \Gamma P_2\setminus\Gamma P$, any other point $g'U$ with $\overline{\Gamma g' U}=\Gamma g' Q_2$ must have $g'\in\SL_3(\Q)gP_2=\SL_3(\Q)P_2$. As $\SL_3(\Q)P_2=\Gamma P_2$ by Lemma~\ref{lem:SLQ-Z}, we conclude that the elements $g\in\Gamma P_2\setminus\Gamma P$ are exactly the ones for which $\overline{\Gamma g U}=\Gamma g Q_2$.

Similar arguments show that the elements $g\in\Gamma P_1\setminus\Gamma P$ are exactly the ones for which $\overline{\Gamma g U}=\Gamma g Q_1$. We note only that it is convenient to pass to inverses and rather prove that $\overline{Ug^{-1}\Gamma}=Q_1 g^{-1}\Gamma$, since then if $Ug^{-1}\gamma_n\to Uh$ in $U\backslash G$ for some $\gamma_n\in\Gamma$ and $h\in G$, we see that the bottom rows of $\gamma_n$ stabilize.

Finally, if $g\in G\setminus (\Gamma  P_{1} \cup \Gamma  P_{2})$ and  $\overline{\Gamma  g U}=\Gamma g H$, then the only remaining option for $H$ is $H=G$.
\ep

Later we will need another version of the stabilization property used in the above proof.

\begin{lemma}\label{lem:cornerconv}
Assume $g$ and $g_n$ are elements of $P_{i,+}$ ($i\in\{1,2\}$) and $\gamma_n$ are elements of $\Gamma$ such that $\gamma_ng_nU\to gU$ in $G/U $ and  $(g_n)_{33}^{-1}>\delta$, if $i=1$, and $(g_n)_{11}>\delta$, if $i=2$, for all $n$ and some $\delta>0$. Then $\gamma_n\in Q_i$ for all $n$ large enough, $(g_n)_{33}\to g_{33}$, if $i=1$, and $(g_n)_{11}\to g_{11}$, if $i=2$.
\end{lemma}

\bp
For $i=2$, the convergence $\gamma_ng_nU\to gU$ implies that the first column of $\gamma_n$ multiplied by $(g_n)_{11}$ converges to the first column of $g$. As $(g_n)_{11}>\delta$, $g_{21}=g_{31}=0$ and $g_{11}>0$, this is possible only if $(\gamma_n)_{21}=(\gamma_n)_{31}=0$ and $(\gamma_{n})_{11}>0$ for $n$ large enough, hence $\gamma_n\in\Gamma\cap Q_2$. Then we must have $(g_n)_{11}\to g_{11}$. The case $i=1$ is similar, but this time we use that $Ug_n^{-1}\gamma_n^{-1}\to Ug^{-1}$ in $U\backslash G$ and look at the bottom rows of matrices.
\ep

Using Lemma~\ref{lem:DM} we can describe the quasi-orbits of $\Gamma\curvearrowright G/U$. But first we need the following lemma.

\begin{lemma}\label{lem:P1-P2-spaces}
We have the following properties.
\begin{enumerate}
  \item The $P_{1,+}$-homogeneous space $P_{1,+}/U$ can be identified with $(\R^2\setminus\{0\})\times(0,+\infty)$. An element $g\in P_{1,+}$ lies in $\Gamma P$ if and only if $\begin{pmatrix}
         g_{11} \\
         g_{21}
       \end{pmatrix}\in\R\cdot\Q^2$.
  \item The $P_{2,+}$-homogeneous space $P_{2,+}/U$ can be identified with $(0,+\infty)\times(\R^2\setminus\{0\})$. An element $g\in P_{2,+}$ lies in $\Gamma P$ if and only if $\begin{pmatrix}
         g_{22} \\
         g_{32}
       \end{pmatrix}\in\R\cdot\Q^2$.
\end{enumerate}
\end{lemma}

\bp
We define an action of $P_{1,+}$ on $\R^3=\R^2\times\R$ by letting
%$g=\begin{pmatrix}    A & \begin{matrix} * \\ * \end{matrix} \\
%    \begin{matrix} 0 & 0 \end{matrix} & r
%\end{pmatrix}$
$$
g(v,t):=\Big(\begin{pmatrix}
           g_{11} & g_{12} \\
           g_{21} & g_{22}
         \end{pmatrix}v,g_{33}t\Big).
$$
It is easy to see that the $P_{1,+}$-orbit of the point $x_0:=\big(\begin{pmatrix} 1 \\ 0 \end{pmatrix},1\big)$ is $(\R^2\setminus\{0\})\times(0,+\infty)$ and its stabilizer is $U$. Hence we can identify $P_{1,+}/U$ with $(\R^2\setminus\{0\})\times(0,+\infty)$.

Next, take $g\in P_{1,+}$. Observe that $\Gamma P=\Gamma P_+$ and if $g=\gamma h$ for some $\gamma\in\Gamma$ and $h\in P_+$, then $\gamma\in\Gamma\cap P_{1,+}=\Gamma\cap Q_1$. It follows that $g\in\Gamma P$ if and only if $gx_0\in(\Gamma\cap Q_1)P_+x_0$. Since $(\Gamma\cap Q_1)P_+x_0=(\R\cdot\Q^2)\times(0,+\infty)$, this finishes the proof of~(1).

Part (2) is proved similarly by considering the action of $P_{2,+}$ on $\R\times\R^2$ defined by
$$
g(t,v):=\Big(g_{11}t,\begin{pmatrix}
           g_{22} & g_{23} \\
           g_{32} & g_{33}
         \end{pmatrix}v\Big).
$$
\ep

For $\theta\in\R\setminus\Q$ and $r,s,t>0$, define the following elements of $G$:
\begin{equation}\label{eq:A-matrices}
A(s,t):=\begin{pmatrix}
s & 0 & 0 \\
0 & s^{-1}t & 0 \\
0 & 0 & t^{-1}
\end{pmatrix},\quad
A_1(r,\theta):=\begin{pmatrix}
1 & 0 & 0 \\
\theta & r & 0 \\
0 & 0 & r^{-1}
\end{pmatrix},\quad
A_2(r,\theta):=\begin{pmatrix}
r & 0 & 0 \\
0 & 1 & 0 \\
0 & \theta & r^{-1}
\end{pmatrix}.
\end{equation}

\begin{lemma}\label{lem:representatives}
As a set, the quasi-orbit space $\QQ(G/U)$ for the action $\Gamma\curvearrowright G/U$ can be identified with the disjoint union of the sets $P_+/U$, $P_{1,+}/Q_1$, $P_{2,+}/Q_2$ and $\{0\}$. As representatives of the quasi-orbits parameterized by $P_+/U$, $P_{1,+}/Q_1$ and $P_{2,+}/Q_2$ we can take the points $A(s,t)U$, $A_1(r,\theta)U$ and $A_2(r,\theta)U$, resp., where $r,s,t>0$ and $\theta\in\R\setminus\Q$ is fixed. For each $r>0$ and $i=1,2$, the quasi-orbit of $A_i(r,\theta)U$ is independent of the choice of $\theta\in\R\setminus\Q$.
\end{lemma}

\bp
By Lemma~\ref{lem:DM} we have orbit closures of four types. We deal separately with each of them. Every $\Gamma$-orbit in $\Gamma P$ can be represented by an element of $P_+$. Given two such elements $g,h\in P_+$, the $\Gamma$-orbits of $gU$ and $hU$ are closed. If these orbits coincide, then $g\in\gamma hU$ for some $\gamma\in\Gamma$. Then $\gamma\in \Gamma\cap P_+=\Gamma\cap U$ and therefore $g\in h(h^{-1}\gamma h)U=hU$, since~$P_+$ normalizes~$U$. Therefore the quasi-orbits of the first type are parameterized by $P_+/U$. It is immediate that the matrices $A(s,t)$ ($s,t>0$) give representatives of the cosets in $P_+/U$.

Similarly, the quasi-orbits of $gU$ for $g\in\Gamma P_1\setminus\Gamma P$ are parameterized by $(P_{1,+}\setminus\Gamma P)Q_1/Q_1$.
Since $\begin{pmatrix} 1 & 0 \\ \theta & r \end{pmatrix}\SL_2(\R)$ coincides with the set of $2\times2$ matrices of determinant $r$, the elements $A_1(r,\theta)$, for all $r>0$ and any fixed $\theta$, are representatives of the cosets in $P_{1,+}/Q_1$ and, for each $r>0$, the coset $A_1(r,\theta)Q_1$ is independent of the choice of $\theta$. By Lemma~\ref{lem:P1-P2-spaces}(1) we have $A_1(r,\theta)\notin\Gamma P$ if $\theta\in\R\setminus\Q$. Thus, every coset in $P_{1,+}/Q_1$ can be represented by an element of $P_{1,+}\setminus\Gamma P$, that is, we have $(P_{1,+}\setminus\Gamma P)Q_1=P_{1,+}$. Therefore the quasi-orbits of $gU$ for $g\in\Gamma P_1\setminus\Gamma P$ are parameterized by $P_{1,+}/Q_1$ and represented by the points $A_1(r,\theta)U$.

The case of quasi-orbits of $gU$ for $g\in\Gamma P_2\setminus\Gamma P$ is similar. Finally, all elements $g\in G\setminus(\Gamma P_1\cup\Gamma P_2)$ define the same quasi-orbit, which we denote by~$0$.
\ep

We now want to describe the topology on the quasi-orbit space $\QQ(G/U)$. It is clear that the point $0\in\QQ(G/U)$ is dense. To deal with the rest, denote by $\QQ_0(s,t)$, $\QQ_1(r)$ and $\QQ_2(r)$ the quasi-orbits of $A(s,t)U$, $A_1(r,\theta)U$ and $A_2(r,\theta)U$, resp. Let us first observe the following.

\begin{lemma} \label{lem:q1-q2-q0}
For every $r>0$, we have
$$
\overline{\{\QQ_1(r)\}}=\overline{\{\QQ_0(s,r):s>0\}}\quad\text{and}\quad
\overline{\{\QQ_2(r)\}}=\overline{\{\QQ_0(r,t):t>0\}}.
$$
\end{lemma}

\bp
Let us explain the first equality, the second is similar. The closure $\Gamma A_1(r,\theta)Q_1$ of the $\Gamma$-orbit of $A_1(r,\theta)U$ contains the points $A(s,r)U$ for all $s>0$. This proves one inclusion. The opposite inclusion follows from the fact $\SL_2(\Z)P_2(\R)=\SL_2(\Q)P_2(\R)$ is dense in $\SL_2(\R)$, since it implies that we can find $\gamma_n\in\SL_2(\Z)$ and $s_n>0$ such that $\operatorname{diag}(\gamma_n,1)A(s_n,r)U\to A_1(r,\theta) U$.
\ep

We first consider convergence in $\QQ(G/U)$ for sequences of elements of the form $\QQ_0(s,t)$.

\begin{lemma}\label{lem:SL3-convergence-1}
Let $(s_{n})_n$ and $(t_{n})_n$ be sequences in $(0,+\infty)$. Then $\QQ_0(s_n,t_n)\to0$ in $\QQ(G/U)$ if and only if $s_n\to0$ and $t_n\to0$. If $(s_n,t_n)\to(s_0,t_0)\in[0,+\infty]^2\setminus\{(0,0)\}$, then
\begin{enumerate}
    \item $\QQ_0(s_n,t_n)\to\QQ_1(r)$ if and only if $s_0=0$ and $t_0=r$;
    \item $\QQ_0(s_n,t_n)\to \QQ_2(r)$ if and only if $s_0=r$ and $t_0=0$;
    \item $\QQ_0(s_n,t_n)\to\QQ_0(s,t)$ if and only if $s_0\in\{0,s\}$ and $t_0\in\{0,t\}$.
\end{enumerate}
\end{lemma}

\bp
Assume that $\QQ_0(s_{n}, t_{n})\to 0$, but $s_{n}\not\to0$. Then we can pass to a subsequence and assume that $s_{n}>\delta>0$ for all $n$. Since $\QQ_0(s_{n}, t_{n})\to 0$, for any $s,t>0$ we should have elements $\gamma_{n} \in \Gamma$ such that $\gamma_{n} A(s_n,t_n)U \to A(s, t)U$. By Lemma \ref{lem:cornerconv} it follows that $s_n\to s$. As $s>0$ can be arbitrary, this is a contradiction. Hence $s_n\to0$. In a similar way one shows that $t_n\to0$.

Assume now that $s_{n} \to 0$ and $t_{n} \to 0$. We want to show that $\QQ_0(s_{n}, t_{n})\to0$. By passing to subsequences we may assume that one of the following is true: (i) $s_{n}^{-1}t_{n} \to +\infty$, (ii) $s_{n}^{-1}t_{n} \to 0$ or (iii) $s_{n}^{-1}t_{n} \to r \in (0, +\infty)$.

In the case (i), take $s,t >0$ and pick odd numbers $k_{n}, m_{n}$ such that $m_{n}s_{n} \to s$ and $k_{n} t_{n} \to t$. By Lemma \ref{lem:modd}, for every $n$, there are elements $\gamma\in\SL_2(\Z)$ and $u\in U_2(\R)$ such that
$$
\begin{pmatrix}
  \gamma & \begin{matrix}
             0 \\
             0
           \end{matrix}\\
  \begin{matrix}
    0 & 0
  \end{matrix} & 1
\end{pmatrix}
\begin{pmatrix}
s_{n} & 0 & 0 \\
0 & s_{n}^{-1} t_{n} & 0 \\
0 & 0 & t_{n}^{-1}
\end{pmatrix}
\begin{pmatrix}
  u & \begin{matrix}
             0 \\
             0
           \end{matrix}\\
  \begin{matrix}
    0 & 0
  \end{matrix} & 1
\end{pmatrix}
=
\begin{pmatrix}
m_{n}s_{n} & 0 & 0 \\
2s_{n}  & \frac{s_{n}^{-1} t_{n}}{m_{n}} & 0 \\
0 & 0 & t_{n}^{-1}
\end{pmatrix}.
$$
Applying Lemma \ref{lem:modd} again, we can find $\gamma'=\begin{pmatrix}
k_{n} & c \\
2 & d
\end{pmatrix}\in\SL_2(\Z)$ and $u'\in U_2(\R)$ such that
$$
\gamma'
\begin{pmatrix}
 \frac{s_{n}^{-1} t_{n}}{m_{n}} & 0 \\
  0 & t_{n}^{-1}
\end{pmatrix} u'
=
\begin{pmatrix}
 \frac{k_ns_{n}^{-1} t_{n}}{m_{n}} & 0 \\
  \frac{2s_{n}^{-1} t_{n}}{m_{n}} & \frac{t_{n}^{-1}}{k_{n}}
\end{pmatrix}.
$$
Then
$$
\begin{pmatrix}
1 & 0 & 0 \\
0 &  k_{n} & c \\
0 &  2 & d
\end{pmatrix}
\begin{pmatrix}
m_{n}s_{n} & 0 & 0 \\
2s_{n}  & \frac{s_{n}^{-1} t_{n}}{m_{n}} & 0 \\
0 & 0 & t_{n}^{-1}
\end{pmatrix}\begin{pmatrix}
               1 & \begin{matrix}
                     0 & 0
                   \end{matrix}  \\
               \begin{matrix}
                 0 \\
                 0
               \end{matrix} & u'
             \end{pmatrix}
=
\begin{pmatrix}
m_{n}s_{n} & 0 & 0 \\
2k_{n}s_{n}  & \frac{k_n t_{n}}{m_{n}s_{n}} & 0 \\
4s_{n} & \frac{2 t_{n}}{m_{n}s_{n}} & \frac{1}{k_{n}t_{n}}
\end{pmatrix}.
$$
The last expression converges to $A(s,t)$ by our assumptions and the choice of $k_n$ and $m_n$; note in particular that $k_{n} s_{n} = k_{n}t_{n}(t_{n}^{-1} s_{n} ) \to 0$. Hence $\QQ_0(s_n,t_n)\to\QQ_0(s,t)$. Since the elements $A(s,t)$ represent the $\Gamma$-orbits in $\Gamma P/U$ and, by Lemma~\ref{lem:SLQ-Z}, $\Gamma P=\SL_{3}(\Q)P$ is dense in $G=\SL_3(\R)$, it follows that the sequence $(\QQ_0(s_n,t_n))_n$ converges to every point in $\QQ(G/U)$.

Case (ii) is similar and even slightly easier. We take $s,t >0$ and pick odd numbers $k_{n}, m_{n}$ such that $m_{n}s_{n} \to s$ and $k_{n} t_{n} \to t$. For every $n$, using Lemma \ref{lem:modd} we find $\gamma \in \SL_2(\Z)$ and $u\in U_2(\R)$ such that
$$
\begin{pmatrix}
  1 & \begin{matrix}
        0 & 0
      \end{matrix} \\
  \begin{matrix}
    0 \\
    0
  \end{matrix} & \gamma
\end{pmatrix}
\begin{pmatrix}
s_{n} & 0 & 0 \\
0 & s_{n}^{-1} t_{n} & 0 \\
0 & 0 & t_{n}^{-1}
\end{pmatrix}
\begin{pmatrix}
  1 & \begin{matrix}
        0 & 0
      \end{matrix} \\
  \begin{matrix}
    0 \\
    0
  \end{matrix} & u
\end{pmatrix}
=
\begin{pmatrix}
s_{n} & 0 & 0 \\
0  & k_{n}s_{n}^{-1} t_{n} & 0 \\
0 & 2s_{n}^{-1} t_{n} & \frac{1}{k_{n} t_{n}}
\end{pmatrix} \; .
$$
Using Lemma \ref{lem:modd} one more time we find $\gamma'\in\SL_2(\Z)$ and $u' \in U_{2}(\R)$ such that
$$
\gamma'
\begin{pmatrix}
 s_{n} & 0 \\
  0 & k_{n}s_{n}^{-1}t_{n}
\end{pmatrix} u'
=
\begin{pmatrix}
 m_{n} s_{n} & 0 \\
  2s_{n} & \frac{k_{n}t_{n}}{s_{n}m_{n}}
\end{pmatrix}.
$$
Then
$$
\begin{pmatrix}
  \gamma' & \begin{matrix}
             0 \\
             0
           \end{matrix}\\
  \begin{matrix}
    0 & 0
  \end{matrix} & 1
\end{pmatrix}
\begin{pmatrix}
s_{n} & 0 & 0 \\
0  & k_{n}s_{n}^{-1} t_{n} & 0 \\
0 & 2s_{n}^{-1} t_{n} & \frac{1}{k_{n} t_{n}}
\end{pmatrix}\begin{pmatrix}
  u' & \begin{matrix}
             0 \\
             0
           \end{matrix}\\
  \begin{matrix}
    0 & 0
  \end{matrix} & 1
\end{pmatrix}
=
\begin{pmatrix}
m_{n}s_{n} & 0 & 0 \\
2s_{n}  & \frac{k_{n} t_{n}}{m_{n} s_{n}} & 0 \\
0 & \frac{2t_{n}}{s_n} & \frac{1}{k_{n} t_{n}}.
\end{pmatrix}
$$
The last expression converges to $A(s,t)$ and, as in the case (i), this leads to the conclusion that $\QQ_0(s_n,t_n)\to0$.

In the case (iii), fix $\theta\in\R\setminus\Q$. Consider the matrices
$$
h_n:=s_n^{-1/2}t_n^{1/2}\begin{pmatrix}
 s_n & 0 \\
  0 & t_n^{-1}
 \end{pmatrix}=\begin{pmatrix}
 (s_nt_n)^{1/2} & 0 \\
  0 & (s_nt_n)^{-1/2}
 \end{pmatrix}, \quad h:=r^{1/2}\begin{pmatrix}
1 & 0 \\
\theta & r^{-1}.
\end{pmatrix}
$$
As $s_nt_n\to0$, by Lemma~\ref{lem:SL2-3}(3) we can find $\gamma_n\in\SL_2(\Z)$ such that $\gamma_nh_nU_2(\R)\to hU_2(\R)$ in $\SL_2(\R)/U_2(\R)$. This implies that
$\gamma_n'A(s_n,t_n)U\to gU$, where
$$
\gamma_n':=\begin{pmatrix}
(\gamma_n)_{11} & 0  & (\gamma_n)_{12} \\
0  & 1 & 0 \\
(\gamma_n)_{21} & 0 & (\gamma_n)_{22}
\end{pmatrix},\quad
g:=\begin{pmatrix}
1 & 0 & 0 \\
0  & r & 0 \\
\theta & 0 & r^{-1}
\end{pmatrix}.
$$
To finish the proof of the convergence $\QQ_0(s_n,t_n)\to0$ it suffices to show that the quasi-orbit of~$gU$ is $0\in\QQ(G/U)$, that is, $g\not\in\Gamma P_1\cup\Gamma P_2$. For this observe that for every element $h\in\Gamma P_2$ the first column of $h$ lies in $\R\cdot\Z^3$, while for every $h\in\Gamma P_1$ the last row of $h^{-1}$ lies in $\R\cdot\Z^3$. Since~$g$ does not have either of these properties, we are done.

\smallskip

For the rest of the proof assume that $(s_n,t_n)\to(s_0,t_0)\in[0,+\infty]^2\setminus\{(0,0)\}$.

\smallskip

(3) Assume first that $s_0=0$ and $t_0=t$. By Lemma~\ref{lem:SL2-3}(3) we can find elements $\gamma_{n} \in \SL_{2}(\mathbb{R})$ and $u_{n}\in U_{2}(\R)$ such that
$$
\gamma_{n}
\begin{pmatrix}
s_{n}t_{n}^{-1/2} & 0 \\
0 & s_{n}^{-1}t_{n}^{1/2}
\end{pmatrix}
u_{n}
\to
t^{-1/2}\begin{pmatrix}
s & 0 \\
0 & s^{-1}t
\end{pmatrix}.
$$
Then $\diag(\gamma_n,1)A(s_n,t_n)\diag(u_n,1)\to A(s,t)$, hence $\QQ_0(s_n,t_n)\to\QQ_0(s,t)$. Similarly, if $s_0=s$ and $t_0=0$, then $\QQ_0(s_n,t_n)\to\QQ_0(s,t)$. It is also obvious that if $s_0=s$ and $t_0=t$, then $\QQ_0(s_n,t_n)\to\QQ_0(s,t)$.

Conversely, assume $\QQ_0(s_n,t_n)\to\QQ_0(s,t)$. If $s_0\ne0$, then $s_0=s$ by Lemma \ref{lem:cornerconv}. Similarly, if $t_0\ne0$, then $t_0=t$.

\smallskip

(1) By Lemma~\ref{lem:q1-q2-q0} we have $\QQ_0(s_n,t_n)\to\QQ_1(r)$ if and only if $\QQ_0(s_n,t_n)\to\QQ_0(s,r)$ for all $s>0$. By (3) and the assumption $(s_0,t_0)\ne(0,0)$, this happens if and only if $s_0=0$ and $t_0=r$. (2) is similar.
\ep

Next we consider sequences of elements of the form $\QQ_1(r)$.

\begin{lemma} \label{lem:Q1}
Let $(r_{n})_{n}$ be a sequence in $(0,+\infty)$. Then $\QQ_1(r_n)\to0$ in $\QQ(G/U)$ if and only if $r_n\to0$. If $(r_{n})_{n}$ is bounded away from zero, then
\begin{enumerate}
\item $(\mathcal{Q}_{1}(r_{n}))_n$ has no cluster points of the form $\mathcal{Q}_{2}(r)$;
\item $\mathcal{Q}_{1}(r_{n}) \to \mathcal{Q}_{1}(r)$ if and only if $r_{n} \to r$;
\item $\mathcal{Q}_{1}(r_{n}) \to \QQ_0(s,t)$ if and only if $r_{n} \to t$.
\end{enumerate}
\end{lemma}
\bp
By Lemma~\ref{lem:q1-q2-q0} we have $\QQ_1(r_n)\to x$ for some $x\in\QQ(G/U)$ if and only if there exist $s_n>0$ such that $\QQ_0(s_n,r_n)\to x$. Using this we see that all statements of the lemma are consequences of Lemma~\ref{lem:SL3-convergence-1}.
\ep

Similarly we get the following result.

\begin{lemma}\label{lem:SL3-convergence-3}
Let $(r_{n})_{n}$ be a sequence in $(0,+\infty)$. Then $\QQ_2(r_n)\to0$ in $\QQ(G/U)$ if and only if $r_n\to0$. If $(r_{n})_{n}$ is bounded away from zero, then
\begin{enumerate}
\item $(\mathcal{Q}_{2}(r_{n}))_n$ has no cluster points of the form $\mathcal{Q}_{1}(r)$;
\item $\mathcal{Q}_{2}(r_{n}) \to \mathcal{Q}_{2}(r)$ if and only if $r_{n} \to r$;
\item $\mathcal{Q}_{2}(r_{n}) \to \QQ_0(s,t)$ if and only if $r_{n} \to s$.
\end{enumerate}
\end{lemma}

\subsection{Stabilizer groups}
We continue to denote $\SL_3(\R)$, $\SL_3(\Z)$ and $U_3(\R)$ by $G$, $\Gamma$ and $U$.
We need to compute the stabilizers for the action $\Gamma\curvearrowright G/U$. Consider the following subgroups of~$\Gamma$:
\begin{equation}\label{eq:Gamma-groups}
H_3(\Z):=
\begin{pmatrix}
1 & \mathbb{Z} & \mathbb{Z} \\
0 & 1 & \mathbb{Z} \\
0 & 0 & 1
\end{pmatrix}, \quad
\Gamma_{1}: =
\begin{pmatrix}
1 & 0 & \mathbb{Z} \\
0 & 1 & \mathbb{Z} \\
0 & 0 & 1
\end{pmatrix}, \quad
\Gamma_{2} :=
\begin{pmatrix}
1 &  \mathbb{Z} & \mathbb{Z} \\
0 & 1 & 0 \\
0 & 0 & 1
\end{pmatrix}.
\end{equation}
Thus, $H_3(\Z)=\SL_3(\Z)\cap U_3(\R)$ is the discrete Heisenberg group.

\begin{lemma}\label{lem:stabilizers}
We have:
\begin{enumerate}
  \item if $g\in P_+$, then the stabilizer of $gU$ is $H_3(\Z)$;
  \item if $g\in P_{i,+}\setminus\Gamma P$ ($i\in\{1,2\}$), then the stabilizer of $gU$ is $\Gamma_i$;
  \item if $g\in G\setminus(\Gamma P_1\cup\Gamma P_2)$, then the stabilizer of $gU$ is trivial.
\end{enumerate}
\end{lemma}

\bp
(1) The stabilizer of $U\in G/U$ is $\Gamma\cap U=H_3(\Z)$. Since $P_+$ normalizes $U$, by acting by~$P_+$ on the right we conclude that the same is true for every point in $P_+/U\subset G/U$.

\smallskip

(2) First of all note that if $\gamma\in\Gamma$ stabilizes $gU$ for some $g\in P_{1,+}\setminus\Gamma P$, then $\gamma\in\Gamma\cap P_{1,+}=\Gamma\cap Q_1$. By Lemma~\ref{lem:P1-P2-spaces}(1) we can identify $(P_{1,+}\setminus\Gamma P)/U$ with $(\R^2\setminus\R\cdot\Q^2)\times(0,+\infty)$. Since the stabilizer in $\SL_2(\Z)$ of any vector in $\R^2\setminus\R\cdot\Q^2$ is trivial, the stabilizer in $\Gamma\cap Q_1$ of any point in $(\R^2\setminus\R\cdot\Q^2)\times(0,+\infty)$ is $\Gamma_1$. This proves (2) for $i=1$. The case $i=2$ is similar.

\smallskip

(3) Take $g\in G\setminus(\Gamma P_1\cup\Gamma P_2)$. We want to show that the stabilizer of $gU$ is trivial. If $\gamma\in\Gamma$ stabilizes $gU$, then $\gamma$ stabilizes the first column of $g$. If the entries of this column are linearly independent over $\Q$, this already implies that $\gamma=e$. On the other hand, this column cannot lie in $\R\cdot\Q^3$, since otherwise we would have $g\in\Gamma P_2$. Therefore we may assume that the entries of the first column of $g$ span a $2$-dimensional space over $\Q$. Choose $a_{1}, a_{2}, a_{3} \in \mathbb{Z}$ with $\gcd(a_{1}, a_{2}, a_{3})=1$ such that $a_{1}g_{11}+a_{2}g_{21}+a_{3}g_{31}=0$. Take any $\gamma'\in\Gamma$ with the last row $(a_1,a_2,a_3)$. Replacing $g$ by $\gamma' g$ we may then assume that $g_{31}=0$. By acting by a diagonal matrix on the right we may also assume that $g_{11}=1$.

Thus, we may assume that $g$ has the form
$$
g=
\begin{pmatrix}
1 & a &* \\
\theta & b & * \\
0 & c & *
\end{pmatrix},
$$
with $\theta\in\R\setminus\Q$. As $g\not\in P_1$, we must have $c\ne0$. Now, as we have already observed, if $\gamma\in\Gamma$ stabilizes~$gU$, then $\gamma$ stabilizes the first column of $g$. This implies that $\gamma\in\Gamma_1$. Hence we have
$$
\begin{pmatrix}
1 & 0 & \gamma_{13} \\
0 & 1 & \gamma_{23} \\
0 & 0 & 1
\end{pmatrix}
\begin{pmatrix}
1 & a &* \\
\theta & b & * \\
0 & c & *
\end{pmatrix}
=
\begin{pmatrix}
1 & a &* \\
\theta & b & * \\
0 & c & *
\end{pmatrix}
\begin{pmatrix}
1 & r & *\\
0 & 1 & * \\
0 & 0 & 1
\end{pmatrix}
$$
for some $r\in\R$. This gives $\gamma_{13}c=r$ and $\gamma_{23}c = \theta r$. As $c\ne0$ and $\theta\not\in\Q$, it follows that $\gamma_{13}=\gamma_{23}=0$. In conclusion, $\gamma =e$.
\ep

\subsection{The discrete Heisenberg group}\label{ssec:Heisenberg}
Next we recall the structure of the primitive spectrum for the discrete Heisenberg group $H_3(\Z)$. This seems to go back at least to the thesis of Howe~\cite{Howe}, but we will rely on the paper of Baggett and Packer~\cite{BP}.

The group $H_3(\Z)$ is generated by the elements
$$
X:=\begin{pmatrix}
    1 & 0 & 0 \\
    0 & 1 & 1 \\
    0 & 0 & 1
  \end{pmatrix},\quad
Y:=\begin{pmatrix}
    1 & 1 & 0 \\
    0 & 1 & 0 \\
    0 & 0 & 1
  \end{pmatrix},\quad
Z:=\begin{pmatrix}
    1 & 0 & 1 \\
    0 & 1 & 0 \\
    0 & 0 & 1
  \end{pmatrix}.
$$
The element $Z$ is central and $YX=ZXY$. We parameterize $\Prim C^*(H_3(\Z))$ by the triples $(z,a,b)\in\T^3$ such that $a=b=1$ if $\ord z=+\infty$, as follows.

Given an irreducible unitary representation $\pi\colon H_3(\Z)\to U(H)$, we have $\pi(Z)=z1$ for some $z\in\T$. If $\ord z=n\in \N$, then $\pi(X^n)$ and $\pi(Y^n)$ lie in the center of $C^*_\pi(H_3(\Z))$, hence $\pi(X^n)=a1$ and $\pi(Y^n)=b1$ for some $a,b\in\T$. The triple $(z,a,b)$ (with $a=b=1$ if $\ord z=+\infty$) determines~$\pi$ up to weak equivalence.

An irreducible representation $\pi$ corresponding to a given triple $(z,a,b)$ can be obtained as follows. If $\ord z=+\infty$, we consider the quotient of $C^*(H_3(\Z))$ by the ideal generated by $Z-z1$. This quotient is an irrational rotation algebra, so it is simple, and as the representation $\pi$ we can take any irreducible representation of this C$^*$-algebra. If $\ord z=n\in\N$, we take unitary generators $u$ and $v$ of $\Mat_n(\C)$ satisfying the relations $u^n=v^n=1$ and $vu=zuv$, and define
$$
\pi(X):=a^{1/n}u,\quad, \pi(Y):=b^{1/n}v,\quad \pi(Z):=z1,
$$
where $a^{1/n}$ and $b^{1/n}$ are arbitrary $n$-th roots of $a$ and $b$.

We will need the following immediate consequence of this discussion. Recall from~\eqref{eq:Gamma-groups} that we introduced subgroups $\Gamma_1=\langle X, Z\rangle$ and $\Gamma_2=\langle Y,Z\rangle$ of $H_3(\Z)$. We identify $\widehat \Gamma_1$ with~$\T^2$, namely, the character corresponding to $(a,z)\in\T^2$ maps $X$ to $a$ and $Z$ to $z$. Similarly, we identify~$\widehat \Gamma_2$ with $\T^2$.

\begin{lemma}\label{lem:H3-hull}
Let $J\in\Prim C^*(H_3(\Z))$ be the primitive ideal corresponding to parameters $(z,a,b)\in\T^3$. Then
$\hull\big(\operatorname{Res}^{H_3(\Z)}_{\Gamma_1}J\big)$ consists of the points $(c,z)\in\T^2$ with $c^{\ord z}=a$, and $\hull\big(\operatorname{Res}^{H_3(\Z)}_{\Gamma_2}J\big)$ consists of the points $(c,z)\in\T^2$ with $c^{\ord z}=b$.
\end{lemma}

As before, here we use the convention that if $\ord z=+\infty$, then there are no restrictions on $c$.

\smallskip

Applying \cite{BP}*{Theorem~1.6} we get the following description of the Jacobson topology.
%We note that this description is also not difficult to obtain from Lemma~\ref{lem:H3-hull} alone.

\begin{lemma}
Identifying $\Prim C^*(H_3(\Z))$ with the set of triples $(z,a,b)$ as described above, we have $(z_n,a_n,b_n)\to (z,a,b)$ if and only if the following properties hold: $z_n\to z$ and, if $\ord z=m\in\N$, whenever we have a subsequence $((z_{n_k},a_{n_k},b_{n_k}))_k$ with $\ord z_{n_k}=m$ for all~$k$, we must have $a_{n_k}\to a$ and $b_{n_k}\to b$.
\end{lemma}

Note that since the set of elements of order $m$ is finite, we in fact have $z_{n_k} = z$ for all $k$ large enough.

\bp This is a routine translation of \cite{BP}*{Theorem~1.6} to our setting. Since~\cite{BP} deals with a more general class of groups, let us only explain the connection between  our triples $(z,a,b)$ and the triples $(\lambda,N_\lambda,\varphi)$ used in \cite{BP}. The center of $H_3(\Z)$ is generated by~$Z$, so every $z\in\T$ defines a character $\lambda$ of the center such that $\lambda(Z)=z$. The group $N_\lambda\subset H_3(\Z)$ is by definition the preimage of the center of $H_3(\Z)/\ker\lambda$. If $\ord z=+\infty$, then $N_\lambda=\langle Z\rangle$. If $\ord z= m\in\N$, then $N_\lambda=\langle X^m,Y^m,Z\rangle$. In the former case $\varphi=\lambda$, and in the latter case $\varphi$ is the character of~$N_\lambda$ given by $\varphi(X^m)=a$, $\varphi(Y^m)=b$, $\varphi(Z)=z$.
\ep

\subsection{The primitive spectrum}
We are now ready to describe the primitive spectrum of $\SL_3(\Z)\ltimes C_0(\SL_3(\R)/U_3(\R))$. In order to formulate our results, we need to introduce some notation.

Consider the set $(\N\times\R)^-:=(\N\times\R)\cup\{\infty\}$, so $(\N\times\R)^-$ is just the one-point compactification of $\N\times\R$, but later we will introduce a different topology on it.

For $\theta\in\R\setminus\Q$, define a map $\QQ_\theta\colon \T^2\to(\N\times\R)^-$ as follows. If $z\in\T^2$ is an element of finite order and $t\in\R$, then we let
$$
\QQ_\theta(zp(tv_\theta)):=(\ord z,t),\quad \text{where}\quad v_\theta:=\begin{pmatrix}1 \\ \theta \end{pmatrix}
$$
and $p$ denotes the quotient map $\R^2\to\T^2$. For $z\in\T^2\setminus p(\Q^2+\R v_\theta)$ we let $\QQ_\theta(z):=\infty$.

Recall also the elements $A(s,t),A_i(r,\theta)\in\SL_3(\R)$ given by~\eqref{eq:A-matrices} and the subgroups $\Gamma_i\subset H_3(\Z)$ given by~\eqref{eq:Gamma-groups}. We do not distinguish between $\Prim C^*(\Gamma_i)$ and $\widehat\Gamma_i$ and identify the latter space with $\T^2$ as in the previous subsection.

\begin{thm}\label{thm:SL3-1}
As a set, the primitive spectrum of $\SL_3(\Z)\ltimes C_0(\SL_3(\R)/U_3(\R))$ can be identified with the disjoint union of $\{0\}$, two copies $(0,+\infty)_1\times (\N\times\R)^-$ and $(0,+\infty)_2\times (\N\times\R)^-$ of $(0,+\infty)\times (\N\times\R)^-$, and $(0,+\infty)^2\times\Prim C^*(H_3(\Z))$. The identification is described as follows:
\begin{enumerate}
  \item the zero ideal corresponds to the point $0$;
  \item the ideal corresponding to $(r,x)\in (0,+\infty)_i\times (\N\times\R)^-$ ($i=1,2$) is $\Ind(A_i(r,\theta)U_3(\R),z)$, where $\theta\in\R\setminus\Q$ is arbitrary and $z=(a,b)\in\T^2=\widehat\Gamma_i$ is any point such that $\QQ_\theta(a,\bar b)=x$ for $i=1$ and $\QQ_\theta(a,b)=x$ for $i=2$;
  \item the ideal corresponding to $(s,t,J)\in (0,+\infty)^2\times\Prim C^*(H_3(\Z))$ is $\Ind(A(s,t)U_3(\R),J)$.
\end{enumerate}
\end{thm}

\bp
As before, we write $\Gamma$, $G$, $U$ for $\SL_3(\Z)$, $\SL_3(\R)$, $U_3(\R)$. The spaces $\Stab(\G)^\prim$ and $\Prim C^*(\G)$ for $\G:=\Gamma\ltimes G/U$ are fibered over the space of quasi-orbits for the action $\Gamma\curvearrowright G/U$. The space of quasi-orbits is described in Section~\ref{ssec:quasi-orbit-SL3}. It consists of four pieces and we need to understand what happens over each of them.

\smallskip

Consider first the points in $G/U$ with dense orbits. By Lemmas~\ref{lem:DM} and~\ref{lem:stabilizers}(3), every such point has trivial stabilizer, hence it defines the zero ideal.

\smallskip

Next, consider the points with closed orbits. Every such orbit is discrete, since otherwise it would have no isolated points and, being also a second countable locally compact space, it would have to be uncountable.  By Lemmas~\ref{lem:DM} and~\ref{lem:representatives} the closed orbits are represented by the matrices $A(s,t)$. By Lemma~\ref{lem:stabilizers}(1) the points $A(s,t)U\in G/U$ have stabilizers~$H_3(\Z)$. By Theorem~\ref{thm:main-prim-fc} it follows that the corresponding part of the primitive spectrum can be identified with $(0,+\infty)^2\times\Prim C^*(H_3(\Z))$, as described in (3).

\smallskip

Next, consider the points $gU$ with $g\in P_{1,+}\setminus\Gamma P$. By Lemma~\ref{lem:stabilizers}(2) the stabilizer of every such point is~$\Gamma_1$. If we are given two such points $gU$ and $hU$, and $\gamma_n gU\to hU$ for some $\gamma_n\in\Gamma$, then by Lemma~\ref{lem:cornerconv} we eventually have $\gamma_n\in\Gamma\cap Q_1$. By Theorem~\ref{thm:main-prim-fc} (or already by~\cite{CN3}), and the fact that $\Prim C^*(\G)$ is a $T_0$-space, it follows that if $\chi,\eta\in\widehat\Gamma_1$, then $\Ind(gU,\chi)=\Ind(hU,\eta)$ if and only if $(gU,\chi)$ and $(hU,\eta)$ have the same orbit closures for the diagonal action $\Gamma\cap Q_1\curvearrowright (P_{1,+}/U)\times\widehat\Gamma_1$, where the action on the second factor is given by conjugation, so $\gamma.\chi=\chi^\gamma=\chi(\gamma^{-1}\cdot\gamma)$. As the action of $\Gamma_1$ is trivial, we in fact have an action of $(\Gamma\cap Q_1)/\Gamma_1\cong\SL_2(\Z)$. Equivalently, we act by the block-diagonal $3\times 3$ matrices $\operatorname{diag}(\gamma,1)$ with $\gamma\in\SL_2(\Z)$.

By Lemma~\ref{lem:P1-P2-spaces}(1) we can identify $(P_{1,+}\setminus\Gamma P)/U$ with $(\R^2\setminus\R\cdot\Q^2)\times(0,+\infty)$. It follows that the part of the primitive spectrum corresponding to the points $gU$ with $g\in P_{1,+}\setminus\Gamma P$ can be parameterized by $\QQ((\R^2\setminus\R\cdot\Q^2)\times\widehat\Gamma_1)\times (0,+\infty)$, where $\QQ((\R^2\setminus\R\cdot\Q^2)\times\widehat\Gamma_1)$ denotes the quasi-orbit space for the diagonal action of $\SL_2(\Z)$ on $(\R^2\setminus\R\cdot\Q^2)\times\widehat\Gamma_1$. Observe now that under the identification of $\widehat\Gamma_1$ with $\T^2\cong\R^2/\Z^2$, the action of $\operatorname{diag}(\gamma,1)$ on $\widehat\Gamma_1$ corresponds to the action of
$$
\begin{pmatrix}
  0 & 1 \\
  1 & 0
\end{pmatrix}(\gamma^\intercal)^{-1}
\begin{pmatrix}
  0 & 1 \\
  1 & 0
\end{pmatrix}
=
\begin{pmatrix}
  1 & 0 \\
  0 & -1
\end{pmatrix}\gamma
\begin{pmatrix}
  1 & 0 \\
  0 & -1
\end{pmatrix}
$$
on $\R^2/\Z^2$. The quasi-orbits of the standard action of $\SL_2(\Z)$ on $(\R^2\setminus\R\cdot\Q^2)\times\T^2$ are described by Lemma~\ref{lem:SL2-1}, which shows that the space of quasi-orbits can be identified with $(\N\times\R)^-$ once we fix a vector $v\in\R^2\setminus(\R\cdot\Q^2)$. Taking $v=v_\theta$ we get (2) for $i=1$.

\smallskip

The remaining case of points $gU$ with $g\in P_{2,+}\setminus\Gamma P$ is similar. We note only that in this case the action of $(\Gamma\cap Q_2)/\Gamma_2\cong\SL_2(\Z)$ on $\widehat\Gamma_2$ gives the usual action of $\SL_2(\Z)$ on $\R^2/\Z^2$, which explains the difference between the parameterizations in (2) for $i=1$ and $i=2$.
\ep

We remark that our analysis of the points $A(s,t)U$ in the above proof used only the essentially trivial consequence of Theorem~\ref{thm:main-prim-fc} saying that $\Ind(x,I)\ne\Ind(y,J)$ if the orbits of $x$ and $y$ are closed and either these orbits are different, or $x=y$ and $I\ne J$. Therefore for the above theorem it suffices to have results of~\cite{CN3} rather than those of the present paper, thanks to the fact that the points with noncommutative stabilizers have closed orbits. This is no longer the case if we want to describe the topology on the primitive spectrum, since these points form a dense set.

In order to describe this topology, define the topology on $(\N\times\R)^-$ such that the closed sets are the empty set, the entire space $(\N\times\R)^-$ and the compact subsets of $\N\times\R$ in the usual product-topology. We parameterize $\Prim C^*(H_3(\Z))$ by triples $(z,a,b)$ as described in Section~\ref{ssec:Heisenberg}. To have a more concise formulation, we also fix $\theta\in\R\setminus\Q$ and use the quasi-orbit space $\QQ(\R^2\times\T^2)$ analyzed in Section~\ref{ssec:SL2-R2-T2}.

\begin{thm}\label{thm:SL3-2}
Using the description of $\Prim(\SL_3(\Z)\ltimes C_0(\SL_3(\R))/U_3(\R))$ from Theorem~\ref{thm:SL3-1}, the Jacobson topology is described as follows.
\begin{enumerate}
  \item The point $0$ is everywhere dense.
  \item A sequence $((r_n,x_n))_n$ in $(0,+\infty)_1\times(\N\times\R)^-$ converges to $0$ if and only if $r_n\to0$. If $(r_n)_n$ is bounded away from zero, then
  \begin{enumerate}
    \item $((r_n,x_n))_n$ has no cluster points in $(0,+\infty)_2\times(\N\times\R)^-$;
    \item $(r_n,x_n)\to(r,x)\in (0,+\infty)_1\times(\N\times\R)^-$ if and only if $r_n\to r$ and $x_n\to x$ in $(\N\times\R)^-$;
    \item $(r_n,x_n)\to (s,t,(z,a,b))\in(0,+\infty)^2\times\Prim C^*(H_3(\Z))$ if and only if $r_n\to t$ and $\QQ(v_\theta,z_n)\to\QQ(se_1,(a^{1/\ord z},\bar z))$ in $\QQ(\R^2\times\T^2)$, where $z_n\in\T^2$ is any point such that $\QQ_\theta(z_n)=x_n$, $v_\theta=\begin{pmatrix}
                                                                                                              1 \\
                                                                                                              \theta
                                                                                                            \end{pmatrix}$ and $e_1=\begin{pmatrix}
                                                                                                              1 \\
                                                                                                              0
                                                                                                            \end{pmatrix}$.
  \end{enumerate}
  \item A sequence $((r_n,x_n))_n$ in $(0,+\infty)_2\times(\N\times\R)^-$ converges to $0$ if and only if $r_n\to0$. If $(r_n)_n$ is bounded away from zero, then
  \begin{enumerate}
    \item $((r_n,x_n))_n$ has no cluster points in $(0,+\infty)_1\times(\N\times\R)^-$;
    \item $(r_n,x_n)\to(r,x)\in (0,+\infty)_2\times(\N\times\R)^-$ if and only if $r_n\to r$ and $x_n\to x$;
    \item $(r_n,x_n)\to (s,t,(z,a,b))\in(0,+\infty)^2\times\Prim C^*(H_3(\Z))$ if and only if $r_n\to s$ and $\QQ(v_\theta,z_n)\to\QQ(s^{-1}te_1,(b^{1/\ord z},z))$, where $z_n\in\T^2$ is any point such that $\QQ_\theta(z_n)=x_n$.
  \end{enumerate}
  \item A sequence $\big((s_n,t_n,(z_n,a_n,b_n))\big)_n$ in $(0,+\infty)^2\times\Prim C^*(H_3(\Z))$ converges to~$0$ if and only if $s_n\to0$ and $t_n\to0$. If $(s_n,t_n)\to(s_0,t_0)\in[0,+\infty]^2\setminus\{(0,0)\}$, then
  \begin{enumerate}
    \item $(s_n,t_n,(z_n,a_n,b_n))\to (r,x)\in (0,+\infty)_1\times(\N\times\R)^-$ if and only if $s_0=0$, $t_0=r$ and $\QQ(s_ne_1,(a_n^{1/\ord z_n},\bar z_n))\to \QQ(v_\theta,z)$, where $z\in\T^2$ is any point such that $\QQ_\theta(z)=x$;
    \item $(s_n,t_n,(z_n,a_n,b_n))\to (r,x)\in (0,+\infty)_2\times(\N\times\R)^-$ if and only if $s_0=r$, $t_0=0$ and $\QQ(s_n^{-1}t_ne_1,(b_n^{1/\ord z_n},z_n))\to \QQ(v_\theta,z)$, where $z\in\T^2$ is any point such that $\QQ_\theta(z)=x$;
    \item $(s_n,t_n,(z_n,a_n,b_n))\to(s,t,(z,a,b))\in (0,+\infty)^2\times\Prim C^*(H_3(\Z))$ if and only if one of the following holds:
    \begin{enumerate}
      \item $s_0=0$, $t_0=t$ and $\QQ(s_ne_1,(a_n^{1/\ord z_n},\bar z_n))\to \QQ(se_1,(a^{1/\ord z},\bar z))$;
      \item  $s_0=s$, $t_0=0$ and $\QQ(t_ne_1,(b_n^{1/\ord z_n},z_n))\to \QQ(te_1,(b^{1/\ord z},z))$;
      \item $s_0=s$, $t_0=t$ and $(z_n,a_n,b_n)\to(z,a,b)$ in $\Prim C^*(H_3(\Z))$.
    \end{enumerate}
  \end{enumerate}
\end{enumerate}
\end{thm}

Here an expression like $a^{1/\ord z}$ means any $(\ord z)$-th root of $a$ if $\ord z\in\N$. If $\ord z=+\infty$, then $a^{1/\ord z}\in\T$ is an arbitrary element. In both cases the points $\QQ(se_1,(a^{1/\ord z},z))$ are independent of any choices, as we already used in the proof of Lemma~\ref{lem:SL2-3}(3). Note also that Lemmas~\ref{lem:SL2-2} and~\ref{lem:SL2-3} give a description of convergence in $\QQ(\R^2\times\T^2)$ needed in a number of cases in Theorem~\ref{thm:SL3-2}, so it is possible to give a formulation of this theorem that does not involve the space $\QQ(\R^2\times\T^2)$.

We remark that parts (i) and (ii) of 4(c) in the above theorem are probably the most interesting, since they involve only points with noncommutative stabilizers $H_3(\Z)$, yet the convergence there is not dictated by the topology on $\Prim C^*(H_3(\Z))$.

\bp[Proof of Theorem~\ref{thm:SL3-2}]
(1) is obvious. Next, note that a necessary condition for a convergence $\Ind(g_nU,J_n)\to\Ind(gU,J)$ is that the quasi-orbit of $g_nU$ converges to that of $gU$. By Theorem~\ref{thm:main-prim-fc} (or already by Lemma~\ref{lem:Ind-continuity}) this condition is also sufficient if the stabilizer of~$gU$ is trivial. Therefore a number of statements of the theorem follow immediately from our description of the quasi-orbit space for the action $\Gamma\curvearrowright G/U$. Namely, this is the case for the convergence to~$0$ in~(2)--(4) by Lemmas~\ref{lem:SL3-convergence-1}, \ref{lem:Q1} and~\ref{lem:SL3-convergence-3}, as well as for the absence of convergence in~(2a) and~(3a) by Lemmas~\ref{lem:Q1} and~\ref{lem:SL3-convergence-3}.

\smallskip

(2b) Similarly to the proof of Theorem~\ref{thm:SL3-1}(2), once $r_n$ are bounded away from zero, the limit points of $((r_n,x_n))_n$ in $(0,+\infty)_1\times(\N\times\R)^-$ are described using the topology on the quasi-orbit space for the action of $(\Gamma\cap Q_1)/\Gamma_1\cong\SL_2(\Z)$ on $(P_{1,+}\setminus\Gamma P)/U\times\widehat\Gamma_1$. We get that $(r_n,x_n)\to(r,x)$ if and only if $r_n\to r$ and $\QQ(v_\theta,z_n)\to\QQ(v_\theta,z)$, where $z_n$ and $z$ are any points in $\T^2$ such that $\QQ_\theta(z_n)=x_n$ and $\QQ_\theta(z)=x$. But by Lemma~\ref{lem:SL2-2} we have $\QQ(v_\theta,z_n)\to\QQ(v_\theta,z)$ if and only if $x_n\to x$ in $(\N\times\R)^-$. This proves (2b). (3b) is proved similarly.

\smallskip

(2c) By Lemma~\ref{lem:Q1}, once $(r_n)_n$ is bounded away from zero, a necessary and sufficient condition for the quasi-orbits of $A_1(r_n,\theta)U$ to converge to that of $A(s,t)U$ is that $r_n\to t$. Assume therefore that $r_n\to t$. By Lemma~\ref{lem:cornerconv}, if $\gamma_nA_1(r_n,\theta)U\to A(s,t)U$ for some $\gamma_n\in\Gamma$, then eventually we must have $\gamma_n\in \Gamma\cap Q_1$. Hence, for $n$ large enough, the stabilizer of $\gamma_nA_1(r_n,\theta)U$ is $\gamma_n\Gamma_1\gamma_n^{-1}=\Gamma_1$. By Theorem~\ref{thm:main-prim-fc} it follows that a sequence $(\Ind(A_1(r_n,\theta)U,\chi_n))_n$, with $\chi_n\in\widehat\Gamma_1$, converges to $\Ind(A(s,t)U,J)$ for some $J\in\Prim C^*(H_3(\Z))$ if and only if for every subsequence we can choose a subsequence $(\Ind(A_1(r_{n_k},\theta)U,\chi_{n_k}))_{k}$ of that subsequence and elements $\gamma_{n_k}\in\Gamma\cap Q_1$ such that $\gamma_{n_k}A_1(r_{n_k},\theta)U\to A(s,t)U$ and the sequence $(\chi_{n_k}^{\gamma_{n_k}})_k$ converges to a point in $\hull\big(\operatorname{Res}^{H_3(\Z)}_{\Gamma_1}J\big)$. Similarly to~(2b), the requirement on the subsequence means that the $(\Gamma\cap Q_1)$-quasi-orbit of $(A(r_{n_k},\theta)U,\chi_{n_k})\in P_{1,+}/U\times\widehat\Gamma_1$ converges to that of $(A(s,t)U,\chi)$ for some $\chi\in\hull\big(\operatorname{Res}^{H_3(\Z)}_{\Gamma_1}J\big)$. By Lemma~\ref{lem:H3-hull}, if~$J$ corresponds to the parameters $(z,a,b)$, then the last hull consists of the characters $Z\mapsto z$, $X\mapsto a^{1/\ord z}$. This implies~(2c). Similar considerations prove also (3c), (4a) and (4b).

\smallskip

(4c) By Lemma~\ref{lem:SL3-convergence-1}, once $(s_0,t_0)\ne(0,0)$, the quasi-orbits of $A(s_n,t_n)U$ converge to the quasi-orbit of $A(s,t)U$ if and only if $s_0$ and $t_0$ are as in one of the cases (i)--(iii). Consider these three cases separately.

\smallskip

(i) Assume $s_n\to0$ and $t_n\to t$. Assume first that
$$
\QQ(s_ne_1,(a_n^{1/\ord z_n},\bar z_n))\to \QQ(se_1,(a^{1/\ord z},\bar z)).
$$
Choose a sequence of vectors $v_k\in\R^2\setminus\R\cdot\Q^2$ such that $v_k\to se_1$. From Lemma~\ref{lem:SL2-3}(3),(4) we see that we can choose points $w_k\in\T^2$ such that
$w_k\to (a^{1/\ord z},\bar z)$ and $\QQ(s_ne_1,(a_n^{1/\ord z_n},\bar z_n))\to \QQ(v_k,w_k)$ for all $k$. Then $\QQ(v_k,w_k)\to \QQ(se_1,(a^{1/\ord z},\bar z))$. By Lemma~\ref{lem:SL2-1} we have $\QQ(v_k,w_k)=\QQ(v_\theta, w_k')$ for some $w_k'\in\T^2$. Let $x_k:=\QQ_\theta(w_k')$. Then, on the one hand, by (4a) we have $(s_n,t_n,(z_n,a_n,b_n))\to(t,x_k)\in(0,+\infty)_1\times(\N\times\R)^-$ for all $k$. On the other hand, by (2c) we have $(t,x_k)\to (s,t,(z,a,b))$. Hence $(s_n,t_n,(z_n,a_n,b_n))\to(s,t,(z,a,b))$.

Conversely, assume $(s_n,t_n,(z_n,a_n,b_n))\to(s,t,(z,a,b))$. Since this assumption is stable under passing to subsequences, in order to show that $\QQ(s_ne_1,(a_n^{1/\ord z_n},\bar z_n))\to \QQ(se_1,(a^{1/\ord z},\bar z))$ it suffices to prove this convergence for some subsequence. Denote by $\pi_n$ and $\pi$ irreducible unitary representations of $H_3(\Z)$ corresponding to $(z_n,a_n,b_n)$ and $(z,a,b)$. By Theorem~\ref{thm:main-prim-fc}, by passing to a subsequence we can find $\gamma_n\in\Gamma$ such that $\gamma_n A(s_n,t_n)U\to A(s,t)U$, the stabilizers $S_n:=\gamma_n H_3(\Z)\gamma_n^{-1}$ of $\gamma_nA(s_n,t_n)U$ converge to a subgroup $S\subset H_3(\Z)$, and $(S_n,\ker\pi_n^{\gamma_n})\to(S,\ker\rho)$ in $\Sub(H_3(\Z))^\prim$ for an irreducible unitary representation $\rho$ such that $\rho\prec\pi|_S$. By Lemma~\ref{lem:cornerconv} we eventually have $\gamma_n\in\Gamma\cap Q_1$. Since $\Gamma\cap Q_1$ normalizes $\Gamma_1$, it follows that, for $n$ large enough, we have $\Gamma_1\subset S_n$. Hence also $\Gamma_1\subset S$. By restricting representations to $\Gamma_1$ we then conclude from the convergence $(S_n,\ker\pi_n^{\gamma_n})\to(S,\ker\rho)$ that for any $\chi\in\hull(\ker\rho|_{\Gamma_1})\subset\hull(\ker \pi|_{\Gamma_1})$ we can find $\chi_n\in\hull(\ker\pi_n|_{\Gamma_1})$ such that $\chi_n^{\gamma_n}\to \chi$. Similarly to our considerations in~(2b) and~(2c), this is possible only when $\QQ(s_ne_1,(a_n^{1/\ord z_n},\bar z_n))\to \QQ(se_1,(a^{1/\ord z},\bar z))$. This finishes the proof of (i). Case (ii) is analyzed in a similar way.

\smallskip

(iii) Assume $s_n\to s$, $t_n\to t$. Then $A(s_n,t_n)U\to A(s,t)U$ and, moreover, by Lemma~\ref{lem:cornerconv}, if $\gamma_n A(s_n,t_n)U\to A(s,t)U$ for some $\gamma_n\in\Gamma$, then eventually $\gamma_n\in\Gamma\cap Q_1\cap Q_2=H_3(\Z)$. Since~$H_3(\Z)$ is the stabilizer of all these points $A(s_n,t_n)U$ and $A(s,t)U$, it is then immediate from Theorem~\ref{thm:main-prim-fc} that $(s_n,t_n,(z_n,a_n,b_n))\to(s,t,(z,a,b))$ if and only if $(z_n,a_n,b_n)\to(z,a,b)$ in $\Prim C^*(H_3(\Z))$.
\ep

\bigskip


\bigskip

\begin{bibdiv}
\begin{biblist}

\bib{MR1799683}{book}{
   author={Anantharaman-Delaroche, C.},
   author={Renault, J.},
   title={Amenable groupoids},
   series={Monographies de L'Enseignement Math\'ematique %[Monographs of   L'Enseignement Math\'ematique]
   },
   volume={36},
   %note={With a foreword by Georges Skandalis and Appendix B by E. Germain},
   publisher={L'Enseignement Math\'ematique, Geneva},
   date={2000},
   pages={196},
   isbn={2-940264-01-5},
   review={\MR{1799683}},
}

\bib{MR1258035}{article}{
   author={Archbold, R. J.},
   author={Spielberg, J. S.},
   title={Topologically free actions and ideals in discrete $C^*$-dynamical
   systems},
   journal={Proc. Edinburgh Math. Soc. (2)},
   volume={37},
   date={1994},
   number={1},
   pages={119--124},
   issn={0013-0915},
   review={\MR{1258035}},
   doi={10.1017/S0013091500018733},
}

\bib{MR0409720}{article}{
   author={Baggett, Lawrence},
   title={A description of the topology on the dual spaces of certain
   locally compact groups},
   journal={Trans. Amer. Math. Soc.},
   volume={132},
   date={1968},
   pages={175--215},
   issn={0002-9947},
   review={\MR{0409720}},
   doi={10.2307/1994889},
}

\bib{MR1473630}{article}{
   author={Baggett, Lawrence},
   author={Kaniuth, Eberhard},
   author={Moran, William},
   title={Primitive ideal spaces, characters, and Kirillov theory for
   discrete nilpotent groups},
   journal={J. Funct. Anal.},
   volume={150},
   date={1997},
   number={1},
   pages={175--203},
   issn={0022-1236},
   review={\MR{1473630}},
   doi={10.1006/jfan.1997.3115},
}

\bib{BP}{article}{
   author={Baggett, Lawrence},
   author={Packer, Judith},
   title={The primitive ideal space of two-step nilpotent group
   $C^*$-algebras},
   journal={J. Funct. Anal.},
   volume={124},
   date={1994},
   number={2},
   pages={389--426},
   issn={0022-1236},
   review={\MR{1289356}},
   doi={10.1006/jfan.1994.1112},
}

\bib{MR0998613}{article}{
   author={Bekka, Mohammed E. B.},
   title={A characterization of locally compact amenable groups by means of
   tensor products},
   journal={Arch. Math. (Basel)},
   volume={52},
   date={1989},
   number={5},
   pages={424--427},
   issn={0003-889X},
   review={\MR{0998613}},
   doi={10.1007/BF01198348},
}

\bib{MR1047140}{article}{
   author={Bekka, Mohammed E. B.},
   title={Amenable unitary representations of locally compact groups},
   journal={Invent. Math.},
   volume={100},
   date={1990},
   number={2},
   pages={383--401},
   issn={0020-9910},
   review={\MR{1047140}},
   doi={10.1007/BF01231192},
}

\bib{MR1195714}{article}{
   author={Boshernitzan, Michael D.},
   title={Elementary proof of Furstenberg's Diophantine result},
   journal={Proc. Amer. Math. Soc.},
   volume={122},
   date={1994},
   number={1},
   pages={67--70},
   issn={0002-9939},
   review={\MR{1195714}},
   doi={10.2307/2160842},
}

\bib{MR2922380}{article}{
   author={Bowditch, B. H.},
   title={Relatively hyperbolic groups},
   journal={Internat. J. Algebra Comput.},
   volume={22},
   date={2012},
   number={3},
   pages={1250016, 66},
   issn={0218-1967},
   review={\MR{2922380}},
   doi={10.1142/S0218196712500166},
}

\bib{MR2391387}{book}{
   author={Brown, Nathanial P.},
   author={Ozawa, Narutaka},
   title={$C^*$-algebras and finite-dimensional approximations},
   series={Graduate Studies in Mathematics},
   volume={88},
   publisher={American Mathematical Society, Providence, RI},
   date={2008},
   pages={xvi+509},
   isbn={978-0-8218-4381-9},
   isbn={0-8218-4381-8},
   review={\MR{2391387}},
   doi={10.1090/gsm/088},
}

\bib{BS}{article}{
   author={Bruce, Chris},
   author={Scarparo, Eduardo},
   title={A tracial characterization of {F}urstenberg's $\times p, \times q$ conjecture},
  journal={Canad. Math. Bull.},
   volume={67},
   number={1},
   date={2024},
   pages={244--256},
  issn={0008-4395,1496-4287},
  review={\MR{4706815}},
  doi={10.4153/s0008439523000693},
}

\bib{MR0352326}{article}{
   author={Brown, Ian D.},
   title={Dual topology of a nilpotent Lie group},
   journal={Ann. Sci. \'Ecole Norm. Sup. (4)},
   volume={6},
   date={1973},
   pages={407--411},
   issn={0012-9593},
   review={\MR{0352326}},
}

\bib{BM}{article}{
   author={Buss, Alcides},
   author={Martinez, Diego},
   title={Essential groupoid amenability and nuclearity of groupoid C$^*$-algebras},
   how={preprint},
   date={2025},
   eprint={\href{https://arxiv.org/abs/2501.01775}{\texttt{2501.01775 [math.OA]}}},
}

\bib{BL}{article}{
   author={B\"onicke, Christian},
   author={Li, Kang},
   title={Ideal structure and pure infiniteness of ample groupoid
   $C^*$-algebras},
   journal={Ergodic Theory Dynam. Systems},
   volume={40},
   date={2020},
   number={1},
   pages={34--63},
   issn={0143-3857},
   review={\MR{4038024}},
   doi={10.1017/etds.2018.39},
}

\bib{CN2}{article}{
 author={Christensen, Johannes},
   author={Neshveyev, Sergey},
   title={Isotropy fibers of ideals in groupoid $\rm C^*$-algebras},
   journal={Adv. Math.},
   volume={447},
   date={2024},
   pages={Paper No. 109696, 32},
   issn={0001-8708},
   review={\MR{4742724}},
   doi={10.1016/j.aim.2024.109696},
}

\bib{CN3}{article}{
      author={Christensen, Johannes},
   author={Neshveyev, Sergey},
   title={The primitive spectrum of C$^{*}$-algebras of \'etale groupoids with abelian isotropy},
        how={preprint},
        date={2024},
      eprint={\href{https://arxiv.org/abs/2405.02025}{\texttt{2405.02025 [math.OA]}}},
}

\bib{MR2359724}{article}{
   author={Clark, Lisa Orloff},
   title={CCR and GCR groupoid $C^*$-algebras},
   journal={Indiana Univ. Math. J.},
   volume={56},
   date={2007},
   number={5},
   pages={2087--2110},
   issn={0022-2518},
   review={\MR{2359724}},
   doi={10.1512/iumj.2007.56.2955},
}

\bib{MR4052213}{article}{
   author={Courtney, Kristin},
   author={Shulman, Tatiana},
   title={Free products with amalgamation over central
   $\rm{C}^*$-subalgebras},
   journal={Proc. Amer. Math. Soc.},
   volume={148},
   date={2020},
   number={2},
   pages={765--776},
   issn={0002-9939},
   review={\MR{4052213}},
   doi={10.1090/proc/14746},
}

\bib{MR2154349}{article}{
   author={Dahmani, Fran\c cois},
   author={Yaman, Asl\i},
   title={Bounded geometry in relatively hyperbolic groups},
   journal={New York J. Math.},
   volume={11},
   date={2005},
   pages={89--95},
   review={\MR{2154349}},
}

\bib{DM}{article}{
   author={Dani, S. G.},
   author={Margulis, G. A.},
   title={Orbit closures of generic unipotent flows on homogeneous spaces of
   ${\rm SL}(3,{\bf R})$},
   journal={Math. Ann.},
   volume={286},
   date={1990},
   number={1-3},
   pages={101--128},
   issn={0025-5831},
   review={\MR{1032925}},
   doi={10.1007/BF01453567},
}

\bib{MR0744294}{article}{
   author={Dani, S. G.},
   author={Smillie, John},
   title={Uniform distribution of horocycle orbits for Fuchsian groups},
   journal={Duke Math. J.},
   volume={51},
   date={1984},
   number={1},
   pages={185--194},
   issn={0012-7094},
   review={\MR{0744294}},
   doi={10.1215/S0012-7094-84-05110-X},
}

\bib{MR1066810}{article}{
   author={Echterhoff, Siegfried},
   title={On maximal prime ideals in certain group $C^*$-algebras and
   crossed product algebras},
   journal={J. Operator Theory},
   volume={23},
   date={1990},
   number={2},
   pages={317--338},
   issn={0379-4024},
   review={\MR{1066810}},
}

\bib{EE}{article}{
   author={Echterhoff, Siegfried},
   author={Emerson, Heath},
   title={Structure and $K$-theory of crossed products by proper actions},
   journal={Expo. Math.},
   volume={29},
   date={2011},
   number={3},
   pages={300--344},
   issn={0723-0869},
   review={\MR{2820377}},
   doi={10.1016/j.exmath.2011.05.001},
}

\bib{MR3012147}{article}{
   author={Echterhoff, Siegfried},
   author={Kl\"uver, Helma},
   title={A general Kirillov theory for locally compact nilpotent groups},
   journal={J. Lie Theory},
   volume={22},
   date={2012},
   number={3},
   pages={601--645},
   issn={0949-5932},
   review={\MR{3012147}},
}

\bib{EH}{book}{
   author={Effros, Edward G.},
   author={Hahn, Frank},
   title={Locally compact transformation groups and $C\sp{\ast} $- algebras},
   series={},
   volume={No. 75},
   publisher={American Mathematical Society, Providence, R.I.},
   date={1967},
   pages={92},
   review={\MR{0227310}},
}

\bib{MR2060024}{article}{
   author={Elkies, Noam D.},
   author={McMullen, Curtis T.},
   title={Gaps in ${\sqrt n}\bmod 1$ and ergodic theory},
   journal={Duke Math. J.},
   volume={123},
   date={2004},
   number={1},
   pages={95--139},
   issn={0012-7094},
   review={\MR{2060024}},
   doi={10.1215/S0012-7094-04-12314-0},
}

\bib{MR2165547}{article}{
   author={Elkies, Noam D.},
   author={McMullen, Curtis T.},
   title={Correction to: ``Gaps in $\sqrt{n}\bmod1$ and ergodic theory''
   [Duke Math. J. {\bf 123} (2004), no. 1, 95--139; MR2060024]},
   journal={Duke Math. J.},
   volume={129},
   date={2005},
   number={2},
   pages={405--406},
   issn={0012-7094},
   review={\MR{2165547}},
   doi={10.1215/S0012-7094-05-12927-1},
}

\bib{MR2745642}{article}{
   author={Exel, R.},
   title={Non-Hausdorff \'etale groupoids},
   journal={Proc. Amer. Math. Soc.},
   volume={139},
   date={2011},
   number={3},
   pages={897--907},
   issn={0002-9939},
   review={\MR{2745642}},
   doi={10.1090/S0002-9939-2010-10477-X},
}

\bib{MR0146681}{article}{
   author={Fell, J. M. G.},
   title={The dual spaces of $C\sp{\ast} $-algebras},
   journal={Trans. Amer. Math. Soc.},
   volume={94},
   date={1960},
   pages={365--403},
   issn={0002-9947},
   review={\MR{0146681}},
   doi={10.2307/1993431},
}

\bib{MR0139135}{article}{
   author={Fell, J. M. G.},
   title={A Hausdorff topology for the closed subsets of a locally compact
   non-Hausdorff space},
   journal={Proc. Amer. Math. Soc.},
   volume={13},
   date={1962},
   pages={472--476},
   issn={0002-9939},
   review={\MR{0139135}},
   doi={10.2307/2034964},
}

\bib{MR0155932}{article}{
   author={Fell, J. M. G.},
   title={Weak containment and Kronecker products of group representations},
   journal={Pacific J. Math.},
   volume={13},
   date={1963},
   pages={503--510},
   issn={0030-8730},
   review={\MR{0155932}},
}

\bib{MR0159898}{article}{
   author={Fell, J. M. G.},
   title={Weak containment and induced representations of groups. II},
   journal={Trans. Amer. Math. Soc.},
   volume={110},
   date={1964},
   pages={424--447},
   issn={0002-9947},
   review={\MR{0159898}},
   doi={10.2307/1993690},
}

\bib{MR0213508}{article}{
   author={Furstenberg, Harry},
   title={Disjointness in ergodic theory, minimal sets, and a problem in
   Diophantine approximation},
   journal={Math. Systems Theory},
   volume={1},
   date={1967},
   pages={1--49},
   issn={0025-5661},
   review={\MR{0213508}},
   doi={10.1007/BF01692494},
}

\bib{MR0393339}{article}{
   author={Furstenberg, Harry},
   title={The unique ergodicity of the horocycle flow},
   conference={
      title={Recent advances in topological dynamics (Proc. Conf.
      Topological Dynamics, Yale Univ., New Haven, Conn., 1972; in honor of
      Gustav Arnold Hedlund)},
   },
   book={
      series={Lecture Notes in Math.},
      volume={Vol. 318},
      publisher={Springer, Berlin-New York},
   },
   date={1973},
   pages={95--115},
   review={\MR{0393339}},
}

\bib{GlH}{collection}{
   title={Sur les groupes hyperboliques d'apr\`es Mikhael Gromov},
   language={French},
   series={Progress in Mathematics},
   volume={83},
   editor={Ghys, \'E.},
   editor={de la Harpe, P.},
   %note={Papers from the Swiss Seminar on Hyperbolic Groups held in Bern,
   %1988;
   %Edited by \'E. Ghys and P. de la Harpe},
   publisher={Birkh\"auser Boston, Inc., Boston, MA},
   date={1990},
   pages={xii+285},
   isbn={0-8176-3508-4},
   review={\MR{1086648}},
   doi={10.1007/978-1-4684-9167-8},
}

\bib{MR0146297}{article}{
   author={Glimm, James},
   title={Families of induced representations},
   journal={Pacific J. Math.},
   volume={12},
   date={1962},
   pages={885--911},
   issn={0030-8730},
   review={\MR{0146297}},
}

\bib{MR2966476}{article}{
   author={Goehle, Geoff},
   title={The Mackey machine for crossed products by regular groupoids. II},
   journal={Rocky Mountain J. Math.},
   volume={42},
   date={2012},
   number={3},
   pages={873--900},
   issn={0035-7596},
   review={\MR{2966476}},
   doi={10.1216/RMJ-2012-42-3-873},
}

\bib{MR0335681}{article}{
   author={Gootman, Elliot C.},
   title={The type of some $C\sp{\ast} $ and $W\sp{\ast} $-algebras
   associated with transformation groups},
   journal={Pacific J. Math.},
   volume={48},
   date={1973},
   pages={93--106},
   issn={0030-8730},
   review={\MR{0335681}},
}

\bib{GR}{article}{
   author={Gootman, Elliot C.},
   author={Rosenberg, Jonathan},
   title={The structure of crossed product $C\sp{\ast} $-algebras: a proof
   of the generalized Effros--Hahn conjecture},
   journal={Invent. Math.},
   volume={52},
   date={1979},
   number={3},
   pages={283--298},
   issn={0020-9910},
   review={\MR{0537063}},
   doi={10.1007/BF01389885},
}

\bib{MR0246999}{article}{
   author={Greenleaf, F. P.},
   title={Amenable actions of locally compact groups},
   journal={J. Functional Analysis},
   volume={4},
   date={1969},
   pages={295--315},
   issn={0022-1236},
   review={\MR{0246999}},
   doi={10.1016/0022-1236(69)90016-0},
}

\bib{MR0623534}{article}{
   author={Gromov, Mikhael},
   title={Groups of polynomial growth and expanding maps},
   journal={Inst. Hautes \'Etudes Sci. Publ. Math.},
   number={53},
   date={1981},
   pages={53--73},
   issn={0073-8301},
   review={\MR{0623534}},
}

\bib{MR0147925}{article}{
   author={Guichardet, Alain},
   title={Caract\`eres des alg\`ebres de Banach involutives},
   language={French},
   journal={Ann. Inst. Fourier (Grenoble)},
   volume={13},
   date={1963},
   pages={1--81},
   issn={0373-0956},
   review={\MR{0147925}},
}

\bib{MR0330079}{article}{
   author={Hausman, Miriam},
   author={Shapiro, Harold N.},
   title={On the mean square distribution of primitive roots of unity},
   journal={Comm. Pure Appl. Math.},
   volume={26},
   date={1973},
   pages={539--547},
   issn={0010-3640},
   review={\MR{0330079}},
   doi={10.1002/cpa.3160260407},
}

\bib{HH}{article}{
   author={Healy, Brendan Burns},
   author={Hruska, G. Christopher},
   title={Cusped spaces and quasi-isometries of relatively hyperbolic groups},
        how={preprint},
        date={2020},
      eprint={\href{https://arxiv.org/abs/2010.09876}{\texttt{2010.09876 [math.GR]}}},
}

\bib{MR1545946}{article}{
   author={Hedlund, Gustav A.},
   title={Fuchsian groups and transitive horocycles},
   journal={Duke Math. J.},
   volume={2},
   date={1936},
   number={3},
   pages={530--542},
   issn={0012-7094},
   review={\MR{1545946}},
   doi={10.1215/S0012-7094-36-00246-6},
}

\bib{Howe}{thesis}{
   author={Howe, Roger Evans},
   title={On representations of nilpotent groups},
   type={Ph.D. Thesis},
   date={1969},
   organization={University of California, Berkeley},
}

\bib{IW0}{article}{
   author={Ionescu, Marius},
   author={Williams, Dana P.},
   title={Irreducible representations of groupoid $C^*$-algebras},
   journal={Proc. Amer. Math. Soc.},
   volume={137},
   date={2009},
   number={4},
   pages={1323--1332},
   issn={0002-9939},
   review={\MR{2465655}},
   doi={10.1090/S0002-9939-08-09782-7},
}

\bib{IW}{article}{
   author={Ionescu, Marius},
   author={Williams, Dana P.},
   title={The generalized Effros--Hahn conjecture for groupoids},
   journal={Indiana Univ. Math. J.},
   volume={58},
   date={2009},
   number={6},
   pages={2489--2508},
   issn={0022-2518},
   review={\MR{2603756}},
   doi={10.1512/iumj.2009.58.3746},
}

\bib{MR0769602}{article}{
   author={Kaniuth, Eberhard},
   title={Weak containment and tensor products of group representations. II},
   journal={Math. Ann.},
   volume={270},
   date={1985},
   number={1},
   pages={1--15},
   issn={0025-5831},
   review={\MR{0769602}},
   doi={10.1007/BF01455523},
}

\bib{MR0884559}{article}{
   author={Kaniuth, Eberhard},
   title={On topological Frobenius reciprocity for locally compact groups},
   journal={Arch. Math. (Basel)},
   volume={48},
   date={1987},
   number={4},
   pages={286--297},
   issn={0003-889X},
   review={\MR{0884559}},
   doi={10.1007/BF01195101},
}

%\bib{Kaw}{misc}{
%      author={Kawabe, Takuya},
%       title={Uniformly recurrent subgroups and the ideal structure of reduced crossed products},
%         how={preprint},
%        date={2017},
%      eprint={\href{https://arxiv.org/abs/1701.03413}{\texttt{1701.03413 [math.OA]}}},
%}

\bib{MR1088230}{article}{
   author={Kawamura, Shinz\={o}},
   author={Tomiyama, Jun},
   title={Properties of topological dynamical systems and corresponding
   $C^*$-algebras},
   journal={Tokyo J. Math.},
   volume={13},
   date={1990},
   number={2},
   pages={251--257},
   issn={0387-3870},
   review={\MR{1088230}},
   doi={10.3836/tjm/1270132260},
}

\bib{KKLRU}{misc}{
      author={Kennedy, Matthew},
      author={Kim, Se-Jin},
      author={Li, Xin},
      author={Raum, Sven},
      author={Ursu, Dan},
       title={The ideal intersection property for essential groupoid C$^{*}$-algebras},
         how={preprint},
        date={2021},
      eprint={\href{https://arxiv.org/abs/2107.03980v3}{\texttt{2107.03980v3 [math.OA]}}},
}

\bib{MR0031489}{article}{
   author={Mackey, George W.},
   title={Imprimitivity for representations of locally compact groups. I},
   journal={Proc. Nat. Acad. Sci. U.S.A.},
   volume={35},
   date={1949},
   pages={537--545},
   issn={0027-8424},
   review={\MR{0031489}},
   doi={10.1073/pnas.35.9.537},
}

\bib{MR0084498}{article}{
   author={McLain, D. H.},
   title={Remarks on the upper central series of a group},
   journal={Proc. Glasgow Math. Assoc.},
   volume={3},
   date={1956},
   pages={38--44},
   issn={2040-6185},
   review={\MR{0084498}},
}

\bib{MR0419675}{article}{
   author={Moore, Calvin C.},
   author={Rosenberg, Jonathan},
   title={Groups with $T\sb{1}$ primitive ideal spaces},
   journal={J. Functional Analysis},
   volume={22},
   date={1976},
   number={3},
   pages={204--224},
   issn={0022-1236},
   review={\MR{0419675}},
   doi={10.1016/0022-1236(76)90009-4},
}

\bib{MR0667315}{article}{
   author={Poguntke, Detlev},
   title={Discrete nilpotent groups have a $T\sb{1}$\ primitive ideal space},
   journal={Studia Math.},
   volume={71},
   date={1981/82},
   number={3},
   pages={271--275},
   issn={0039-3223},
   review={\MR{0667315}},
   doi={10.4064/sm-71-3-271-275},
}

\bib{MR1081649}{article}{
   author={Ramsay, Arlan},
   title={The Mackey--Glimm dichotomy for foliations and other Polish
   groupoids},
   journal={J. Funct. Anal.},
   volume={94},
   date={1990},
   number={2},
   pages={358--374},
   issn={0022-1236},
   review={\MR{1081649}},
   doi={10.1016/0022-1236(90)90018-G},
}

\bib{MR1106945}{article}{
   author={Ratner, Marina},
   title={Raghunathan's topological conjecture and distributions of
   unipotent flows},
   journal={Duke Math. J.},
   volume={63},
   date={1991},
   number={1},
   pages={235--280},
   issn={0012-7094},
   review={\MR{1106945}},
   doi={10.1215/S0012-7094-91-06311-8},
}

\bib{Rbook}{book}{
   author={Renault, Jean},
   title={A groupoid approach to $C^{\ast} $-algebras},
   series={Lecture Notes in Mathematics},
   volume={793},
   publisher={Springer, Berlin},
   date={1980},
   pages={ii+160},
   isbn={3-540-09977-8},
   review={\MR{584266}},
}

\bib{R}{article}{
   author={Renault, Jean},
   title={The ideal structure of groupoid crossed product $C^\ast$-algebras},
   note={With an appendix by Georges Skandalis},
   journal={J. Operator Theory},
   volume={25},
   date={1991},
   number={1},
   pages={3--36},
   issn={0379-4024},
   review={\MR{1191252}},
}

\bib{Sau}{article}{
   author={Sauvageot, Jean-Luc},
   title={Id\'eaux primitifs de certains produits crois\'es},
   language={French},
   journal={Math. Ann.},
   volume={231},
   date={1977/78},
   number={1},
   pages={61--76},
   issn={0025-5831},
   review={\MR{0473355}},
   doi={10.1007/BF01360030},
}

\bib{Scar}{article}{
   author={Scarparo, Eduardo},
   title={A torsion-free algebraically C$^{*}$-unique group},
   %language={French},
   journal={Rocky Mountain J. Math.},
   volume={50},
   date={2020},
   number={5},
   pages={1813--1815},
  % issn={0022-1236},
   review={\MR{4170690}},
   doi={10.1216/rmj.2020.50.1813},
}

\bib{MR2775364}{article}{
   author={Sierakowski, Adam},
   title={The ideal structure of reduced crossed products},
   journal={M\"unster J. Math.},
   volume={3},
   date={2010},
   pages={237--261},
   issn={1867-5778},
   review={\MR{2775364}},
}

\bib{SSW}{collection}{
   author={Sims, Aidan},
   author={Szab\'{o}, G\'{a}bor},
   author={Williams, Dana},
   title={Operator algebras and dynamics: groupoids, crossed products, and
   Rokhlin dimension},
   series={Advanced Courses in Mathematics. CRM Barcelona},
   editor={Perera, Francesc},
   %note={Lecture notes from the Advanced Course held at Centre de Recerca
   %Matem\`atica (CRM) Barcelona, March 13--17, 2017;
   %Edited by Francesc Perera},
   publisher={Birkh\"{a}user/Springer, Cham},
   date={2020},
   pages={x+163},
   isbn={978-3-030-39712-8},
   isbn={978-3-030-39713-5},
   review={\MR{4321941}},
   doi={10.1007/978-3-030-39713-5},
}

\bib{SW}{article}{
   author={Sims, Aidan},
   author={Williams, Dana P.},
   title={The primitive ideals of some \'etale groupoid C$^{*}$-algebras},
   journal={Algebr. Represent. Theory},
   volume={19},
   date={2016},
   number={2},
   pages={255--276},
  % issn={1867-5778},
 %  review={\MR{3549528}},
   doi={10.1007/s10468-015-9573-4},
}

\bib{MR0160118}{article}{
   author={Thoma, Elmar},
   title={\"Uber unit\"are Darstellungen abz\"ahlbarer, diskreter Gruppen},
   language={German},
   journal={Math. Ann.},
   volume={153},
   date={1964},
   pages={111--138},
   issn={0025-5831},
   review={\MR{0160118}},
   doi={10.1007/BF01361180},
}

\bib{MR1313451}{article}{
   author={Tukia, Pekka},
   title={Convergence groups and Gromov's metric hyperbolic spaces},
   journal={New Zealand J. Math.},
   volume={23},
   date={1994},
   number={2},
   pages={157--187},
   issn={1171-6096},
   review={\MR{1313451}},
}

\bib{MR3835454}{article}{
   author={van Wyk, Daniel W.},
   title={The orbit spaces of groupoids whose $C^*$-algebras are GCR},
   journal={J. Operator Theory},
   volume={80},
   date={2018},
   number={1},
   pages={167--185},
   issn={0379-4024},
   review={\MR{3835454}},
   doi={10.7900/jot.2017sep11.2185},
}

\bib{MR4395600}{article}{
   author={van Wyk, Daniel W.},
   author={Williams, Dana P.},
   title={The primitive ideal space of groupoid $C^*$-algebras for groupoids
   with abelian isotropy},
   journal={Indiana Univ. Math. J.},
   volume={71},
   date={2022},
   number={1},
   pages={359--390},
   issn={0022-2518},
   review={\MR{4395600}},
   doi={10.1512/iumj.2022.71.9523},
}

\bib{MR0617538}{article}{
   author={Williams, Dana P.},
   title={The topology on the primitive ideal space of transformation group
   $C\sp{\ast} $-algebras and C.C.R. transformation group $C\sp{\ast}
   $-algebras},
   journal={Trans. Amer. Math. Soc.},
   volume={266},
   date={1981},
   number={2},
   pages={335--359},
   issn={0002-9947},
   review={\MR{0617538}},
   doi={10.2307/1998427},
}

\bib{MR0248688}{article}{
   author={Wolf, Joseph A.},
   title={Growth of finitely generated solvable groups and curvature of
   Riemannian manifolds},
   journal={J. Differential Geometry},
   volume={2},
   date={1968},
   pages={421--446},
   issn={0022-040X},
   review={\MR{0248688}},
}

%\bib{MR0241994}{article}{
%   author={Zeller-Meier, G.},
%   title={Produits crois\'es d'une $C\sp{\ast} $-alg\`ebre par un groupe
%   d'automorphismes},
%   language={French},
%   journal={J. Math. Pures Appl. (9)},
%   volume={47},
%   date={1968},
%   pages={101--239},
%   issn={0021-7824},
%   review={\MR{0241994}},
%}

\end{biblist}
\end{bibdiv}
\end{document}